\documentclass{amsart}
\usepackage{amscd}
\numberwithin{equation}{section}
\usepackage{amsmath}
\usepackage{latexsym}
\usepackage{graphicx}
\usepackage{subfigure}
\usepackage{color}
\usepackage{psfrag}
\usepackage{amssymb,amsfonts,amstext,graphicx}
\usepackage[curve]{xypic}
\usepackage[all]{xy}

\let\cal\mathcal

\def\Lie{\mathop{\text{Lie}}}

\def\gkdim{\operatorname {GK dim}}
\def\gldim{\operatorname {gl\,dim}}
\def\deg{\operatorname {deg}}

\newcommand{\Span}{\mathop{\mathrm{Span}}}

\newtheorem{lemma}{Lemma}[section]
\newtheorem{proposition}[lemma]{Proposition}
\newtheorem{theorem}[lemma]{Theorem}

\newtheorem*{theorem1}{Theorem I}
\newtheorem*{theorem2}{Theorem II}
\newtheorem*{theorem3}{Theorem III}
\newtheorem*{theorem4}{Theorem IV}

\newtheorem*{proofTheor}{Proof of Theorem}

\newtheorem{hypothesis}[lemma]{Hypothesis}

\newtheorem*{proof1}{Proof of Theorem I}
\newtheorem*{proof2}{Proof of Theorem II}
\newtheorem*{proof3}{Proof of Theorem III}

\newtheorem{question}[lemma]{Question}
\newtheorem{corollary}[lemma]{Corollary}

\newtheorem{conjecture}[lemma]{Conjecture}
\newtheorem{fact}[lemma]{Fact}
\newtheorem{facts}[lemma]{Facts}

\theoremstyle{definition}

\newtheorem{example}[lemma]{Example}
\newtheorem{definition}[lemma]{Definition}

\newtheorem{notation}[lemma]{Notation}

\newtheorem{defconvention}[lemma]{Definition-Convention}
\newtheorem{defnotation}[lemma]{Definition-Notation}
\newtheorem{notconvention}[lemma]{Notation-Convention}

\newtheorem{openquestion}[lemma]{Open Question}
\newtheorem{problem}[lemma]{Problem}

{

}

\theoremstyle{remark}

\newtheorem{remark}[lemma]{Remark}

\newcommand{\fib}{f}

\newcommand{\Ical}{\mbox{$\cal I$}}

\newcommand{\Lcal}{\mbox{$\cal L$}}

\newcommand{\Qcal}{\mbox{$\cal Q$}}

\newcommand{\asX}{\langle X \rangle}

\newdimen\uboxsep \uboxsep=1ex
\def\uboxn#1{\vtop to 0pt{\hrule height 0pt depth 0pt\vskip\uboxsep
\hbox to 0pt{\hss #1\hss}\vss}}

\def\uboxs#1{\vbox to 0pt{\vss\hbox to 0pt{\hss #1\hss}
\vskip\uboxsep\hrule height 0pt depth 0pt}}

\title
[Algebras defined by Lyndon words ] {Algebras defined by Lyndon words and Artin-Schelter regularity}


\thanks{The author was partially supported by the Max Planck Institute for Mathematics (MPIM), Bonn, by the Max Planck Institute for Mathematics in the Sciences,
MiS, Leipzig, and by Grant KP-06 N 32/1 of 07.12.2019 of the Bulgarian National
Science Fund.}
\author{Tatiana Gateva-Ivanova}
\address{Max Planck Institute for Mathematics\\
Vivatsgasse 7\\
53111 Bonn\\
Germany\\
and\\
American University in Bulgaria\\
2700, Blagoevgrad\\
Bulgaria
 }
\email{tatyana@aubg.edu}

\begin{document}

\date{\today}
\begin{abstract}
Let $X= \{x_1, x_2, \cdots,  x_n\}$  be a finite alphabet, and let $K$ be a field. We study classes $\mathfrak{C}(X, W)$ of graded
$K$-algebras $A = K\langle X\rangle / I$, generated by $X$ and with \emph{a fixed set of obstructions} $W$.  Initially we do not impose restrictions on $W$ and investigate the case when the algebras in $\mathfrak{C} (X, W)$ have polynomial growth and finite global dimension $d$. Next we consider classes $\mathfrak{C} (X, W)$
of algebras whose sets of obstructions $W$ are antichains of Lyndon words. The central question is "\emph{when a class $\mathfrak{C} (X, W)$ contains Artin-Schelter regular algebras?}"
Each class $\mathfrak{C} (X, W)$ defines a Lyndon pair $(N,W)$, which, if $N$ is finite, determines uniquely the global dimension, $\gldim A$, and the Gelfand-Kirillov dimension, $\gkdim A$, for every $A \in \mathfrak{C}(X, W)$. We find a combinatorial condition in terms of $(N,W)$, so that the class $\mathfrak{C}(X, W)$ contains the enveloping algebra $U\mathfrak{g}$, of a Lie algebra $\mathfrak{g}$. We introduce monomial Lie algebras defined by Lyndon words, and prove results on Gr\"{o}bner-Shirshov bases of Lie ideals generated by Lyndon-Lie monomials.
 Finally we classify all two-generated Artin-Schelter regular algebras of global dimension $6$ and $7$ occurring as enveloping $U = U\mathfrak{g}$ of \emph{standard monomial Lie algebras}. The classification is made in terms of their Lyndon pairs $(N, W)$, each of which determines also the explicit relations of $U$.
\end{abstract}

\maketitle

\setcounter{tocdepth}{1}
\tableofcontents

\section{Introduction}

Let $X= \{x_1, x_2, \cdots,  x_n\}$  be a finite alphabet.
We denote by $X^*$ the free monoid generated by $X$, and by
 $X^{+}$ -the free semigroup generated by $X$,
$X^{+}= X^* -\{1\}$, (the empty word is
denoted by $1$). Throughout the paper  $K\asX$ will denote the
free associative $K$-algebra generated by $X$, where $K$ is a field.
As usual, the length of a word $w\in X^{+}$ is denoted by $|w|$. We
shall consider the canonical grading by length on $K \asX$, assuming
each $x \in X$ has degree $1$.

In the paper we shall use three distinct orderings on $X^{+}$ (denoted by
distinct notation) which will be introduced gradually and used in
different contexts.

The  first order is the so called  "divisibility order", "$\sqsubset$".
This is a partial ordering on the set $X^{+}$ defined as:
$a\sqsubset b$ \emph{iff} $a$ is a proper subword (segment) of $b$,
i.e. $b = uav$, $|b|> |a|$, $u = 1$, or $v = 1$ is possible. In
the case when $b = av, a, v  \in X^{+} $,   $a$ is called \emph{a
proper left factor (segment) of} $b$. Proper right factors (segments) are
defined analogously.

Let $V$ be a set of monomials in $X^{+}$, $V \neq \emptyset$. A
monomial $w \in V$ is \emph{a minimal element of $V$ with respect to the ordering
"$\sqsubseteq$}" if it does not contain as a proper subword any
monomial  $u \in V$. Clearly, for any $u \in V$ there exists a unique
minimal element $w \in V$, with $w \sqsubseteq u.$

A nonempty subset $W\subseteq X^{+}$ is called \emph{an antichain of
monomials} (or shortly \emph{an antichain}) if any two elements $u,
v\in W, u \neq v$ are not comparable with respect to ``$\sqsubset$".
In other words, $W$ is an antichain of monomials if every monomial
$a \in W$ is minimal with respect to  ``$\sqsubset$".

The following is straightforward.
\begin{lemma}
\label{W_def} Let $V$ be  a nonempty set of monomials in   $X^{+}$.
Let
\[
\label{eqW_S} W=W_V= \{u \in V \mid u\; \text{is minimal in} \; V\;
\text{with respect to}\; \sqsubset\}. \] Then $W$ is  the unique maximal
antichain of monomials in $V$ such that every $u \in
V$ has a subword $w \in W$.
\end{lemma}

Let $V$, and $W$ be as above. A monomial $a\in X^{\ast}$ is
$V$-\emph{normal} ($V$-\emph{standard}) if $a$ does not contain as a
subword any $u \in V$. Clearly $a$ is $V$-normal \emph{iff} $a$ is
$W$-normal. Denote by $\mathfrak{N}(W)$  the set of $W$-normal words
\[\mathfrak{N}(W) = \{a \in X^{\ast}\mid a \; \text{is $W$-normal}\}.\]
\emph{The set} $\mathfrak{N}(W)$ \emph{is closed under taking subwords},
Anick calls such a set  \emph{an order ideal of monomials}, or  \emph{o.i.m.},
\cite{Anick86}.

Conversely, each   nonempty order ideal of monomials $\mathfrak{N}
\subset X^{+}$ determines uniquely an antichain of monomials $W=
W(\mathfrak{N})$, so that $\mathfrak{N}= \mathfrak{N}(W)$. Indeed,
the complement $V = X^{\ast} - \mathfrak{N} $ has a unique maximal
antichain of monomials $W= W_V$, and the following implications are
in force.
\[\begin{array}{c}
a \in \mathfrak{N} \Longleftrightarrow \; a \;\text{is $W$- normal};\quad
v \in V \Longleftrightarrow w \sqsubseteq v, \quad\text{for
some}\quad w \in W. \end{array}\]
In this case   $W$ is called
\emph{the set of obstructions for} $\mathfrak{N}$, \cite{Anick86}.
Note that $W$ satisfies the following condition
\[
 W  = \{w \in X^{\ast}\mid w \; \text{is not in }\;
\mathfrak{N}\;\text{but any proper subword $u$ of $w$ is in}\;
\mathfrak{N} \}.
\]

The natural duality between $\mathfrak{N}$ and $W=W(\mathfrak{N})$ is discussed with more details in Section
 \ref{Lyndon_pairs_sec}.
We recall briefly some facts from (noncommutative)
Gr\"{o}bner bases theory.
By convention throughout the paper,
we consider
\emph{the degree-lexicographic ordering} $\prec$ on $X^{\ast}$,
extending the (inverse) ordering
\begin{equation}
\label{Ordeq1}
x_1 \succ  x_2 \succ\cdots
\succ   x_{g-1}\succ x_{g}
\end{equation}
 on $X$.
Then $(X^{\ast}, \prec)$ is a well-ordered monoid
that is "$\prec$" is a total ordering on $\asX$ compatible with the multiplication
and such that every non-empty subset of $\asX$  has a least element.
In all cases when
Gr\"{o}bner basis of an  ideal $I$ in $K \asX$ is considered we
shall use the degree-lexicographic ordering $\prec$ on $X^{\ast}$, extending (\ref{Ordeq1}).

Suppose $f \in K\asX$ is a nonzero polynomial, then its highest
monomial w.r.t. $\prec$ will be denoted by $\overline{f}$. One has
$f = \alpha\overline{f} + \sum_{1 \leq i\leq m} \alpha_i u_i$, where
$\alpha, \alpha_i \in K$, $\alpha \neq 0,$ $u_i\in X^{\ast}, u_i
\prec \overline{f}$. Given a set $F \subseteq K \asX$ of
noncommutative polynomials,
 $\overline{F}$ denotes the set
 \[\overline{F} = \{\overline{f}\mid f \in F, f \neq 0\}.\]

The free associative algebra $\mathfrak{A}=K \asX$ is naturally graded by length,
\[\mathfrak{A} = \bigoplus_{m\ge 0}\mathfrak{A}_m, \; \text{where}\;   \mathfrak{A}_0= K,\;
\mathfrak{A}_m = \Span \{u \in X^m, m \geq 1\}.\]
Let  $I$ be a two sided graded ideal in $K \asX$.  We shall consider
graded algebras, with a minimal presentation, see Notation-convention \ref{notconventionA},
so  without loss of generality we may assume that
$I$ \emph{is generated by homogeneous polynomials of degree $\geq 2$ },
hence $I = \bigoplus_{m\ge 2}I_m$, with $I_m = I\bigcap \mathfrak{A}_m$. Then the quotient
algebra $A = K \asX/ I$ is finitely generated and inherits its
grading $A=\bigoplus_{m\ge 0}A_m$ from $\mathfrak{A}= K \asX$. We shall work with
the so called \emph{normal} $K$-basis of $A$.

A monomial $u \in X^{\ast}$ is said to be \emph{normal mod I }
\emph{iff}  it is $\overline{I}$-normal. The set
$\mathfrak{N}(I)$ of normal modulo $I$ monomials coincides with
$\mathfrak{N}(\overline{I})$.

\begin{remark}
Let $W=W(\overline{I})$ be the
 maximal antichain of monomials in $\overline{I}$, then $u$ is normal (mod I)
\emph{iff} $u$ is $W$-normal.
Hence  $W$ is the obstruction set for
$\mathfrak{N}(I)$ and there are equalities of sets
\begin{equation}
\label{N1eq1}
\mathfrak{N}(I) = \mathfrak{N}(W) = \{u \in X^{\ast}\mid u \;
\text{is normal mod} \; I \}.
\end{equation}
\end{remark}

In particular, the free
monoid $X^{\ast}$ splits as a disjoint union
\begin{equation}
\label{X1eq2}
X^{\ast}=  \mathfrak{N}(I)\bigsqcup \overline{I} =\mathfrak{N}(W)\bigsqcup \overline{I}.
\end{equation}
The free associative algebra $K \asX $ splits into a direct sum of $K$-
  subspaces
  $K \asX \simeq  \Span_{K} \mathfrak{N}\bigoplus I$,
and there are isomorphisms of vector spaces
\begin{equation}
\label{A-Neq4}
\begin{array}{c}
A \simeq \Span_{K} \mathfrak{N}\\
\\
A_m \simeq \Span_{K} \mathfrak{N}^{(m)}, \; \dim A_m =
|\mathfrak{N}^{(m)}|, \; m= 0, 1, 2, 3, \cdots ,
\end{array}
\end{equation}
where  $\mathfrak{N}^{(m)}$ is \emph{the set of all normal words of length}
$m$.

\begin{definition}
\label{obstructions2def}
Given $A = K \asX/ I$, $\mathfrak{N} = \mathfrak{N}(I),$  and $W= W(I)$ as above we shall
also
refer to
 $W$ as \emph{the set of obstructions for $A$},  \cite{GIF}.  The monomial algebra $A_W := K \asX/ (W)$ is called \emph{the monomial algebra associated to} $A$.
\end{definition}

\begin{definition}
\label{s.f.p.def}
A subset
$G \subseteq I$
of monic polynomials \emph{is a Gr\"{o}bner
basis of $I$ (with respect to $\prec$)} if (1) $G$ generates $I$ as a
two-sided ideal, and (2) There is an equality of sets of normal words $\mathfrak{N}(\overline{G}) =
\mathfrak{N}(I)$.

A Gr\"{o}bner basis $G$ of
$I$ is called \emph{reduced} if each  $f \in G$ is reduced modulo
$G - \{f\}$, that is $f$ is a linear combination of normal modulo $G
- \{f\}$ monomials.

It is known  that each ideal $I$ of $K\asX$ has a uniquely determined  reduced
Gr\"{o}bner basis $G= G(I)$ (with respect to $\prec$). However, $G$ may be infinite.

Moreover, the set of obstructions $W= W(I)$ coincides with the set of highest monomials $\overline{G(I)}$ of the reduced
Gr\"{o}bner basis $G= G(I)$, that is
\[
W(I)= \overline{G(I)}.
\]
So, the reduced Gr\"{o}bner basis $G(I)$
is finite if and only if the  set of obstructions $W$ is finite. In
this case the algebra $A = K\langle X\rangle / (G)$ is
called \emph{a standard finitely presented algebra}, or shortly
\emph{an s.f.p. algebra}.
\end{definition}
In \cite{GIF} we initiate the study of
algebraic and homological properties \emph{of graded associative
algebras for which the set of obstructions consists of Lyndon
words}. Lyndon words and Lyndon-Shirshov bases are widely used in
the context of Lie algebras and their enveloping algebras, and also
for PI algebras (see for example the celebrated Shirshov theorem of
heights, \cite{BLSZ}). It will be interesting to explore the remarkable
combinatorial properties of Lyndon words in a more general context
of graded associative algebras.

\begin{notconvention}
\label{notconventionA}
We fix an enumeration on the set $X$ of generators, $X = \{x_1, x_2, \cdots, x_n\}$, $g \geq 2$.  By $\mathfrak{C} (X, W)$ we shall denote \emph{the class of all graded}
$K$-\emph{algebras} $A = K\langle X\rangle / I$, with a finite set of generators $X$ and a set of obstructions $W$ of arbitrary cardinality.
$\mathfrak{N} = \mathfrak{N}(W)$ will denote the set of normal words (mod $W$). $\mathfrak{N}$ is a $K$-basis for each algebra $A \in \mathfrak{C} (X, W)$.
By convention we
consider only \emph{minimal presentations of} $A$, so
\[W\bigcap X = \emptyset, \quad \text{and therefore}\; X \subset
\mathfrak{N}.\]
Without loss of generality we shall always assume that
$I$ \emph{is generated by a set $\mathfrak{R}_0$ of homogeneous polynomials of degree $\geq 2$ }.
This implies that $G(I)$,  the unique reduced Gr\"{o}bner basis of $I$, (w.r.t.
the degree-lexicographic ordering $\prec$ on $X^{\ast}$, extending (\ref{Ordeq1})), is also a set of homogeneous polynomials of degree $\geq 2$  .
Each class $\mathfrak{C} (X, W)$ contains unique monomial algebra $A_W := K\langle X\rangle / (W)$.
Clearly, all algebras $A \in
\mathfrak{C} (X, W)$ share the same associated monomial algebra is $A_W$, and the same normal $K$-basis $\mathfrak{N}$.
There are equalities of sets
\[\overline{G(I)} = W,\quad\quad \mathfrak{N}(I) = \mathfrak{N}(W) = \mathfrak{N}.\]
As a consequence, all $A \in\mathfrak{C} (X, W)$ have the same Gelfand-Kirillov dimension, $\gkdim A = \gkdim A_W$, and the same Hilbert series $H_A(z) = H_{A_W}$.
When we work with Lyndon words we often consider a third order, "$ < $" on the set $X^{+}$, namely, the pure lexicographic order
which extends the ordering $x_1 < x_2 < \cdots < x_n$ on $X$ as
follows: for any $u,v \in X^{+}\;$ $u < v$ \emph{iff } either $u$ is
a proper left factor of $v$ ($v = ub, b \in X^{+}$), or
$u = axb, v = ayc$, where  $x<y; \; x,y \in X,\; a,b,c \in X^*$.
\end{notconvention}

We study classes $\mathfrak{C} (X, W)$ such that every algebra $A \in \mathfrak{C} (X, W)$ has polynomial growth and finite global dimension.
Moreover, starting from Section 3, throughout the paper
we are interested in the special case when the
obstructions set $W$ consists of Lyndon words, see Definition \ref{Lyndonworddef}.
One of the central questions in our study is: "\emph{When a given class $\mathfrak{C} (X, W)$ contains an Artin-Schelter regular algebra?}" .
\begin{remark}
The problem of studying classes $\mathfrak{C} (X, W)$, as above was posed first in \cite{GIF} and in several talks given by the author (Bedlewo, 2013, \cite{TGIBedlewo},  ICTP
2013). On these talks we reported the results of Theorem II, and Theorem III, (2).
\end{remark}
We recall the definition of a regular algebra in the sense of Artin and Schelter, see \cite{AS}, p.171.

Let $A = K \oplus A_1 \oplus A_2 \oplus \cdots$ be a finitely presented graded algebra over a field $K$. The
algebra $A$ is called \emph{Artin-Schelter
regular} (or \emph{AS regular}) if
\begin{enumerate}
\item[(i)]
$A$ has {\em finite global dimension} $\;d$.
\item[(ii)]
$A$ has \emph{finite Gelfand-Kirillov dimension}, i.e.  $A$ has polynomial growth.
\item[(iii)]
$A$ is \emph{Gorenstein}, meaning that $Ext^q_A(K,A)=0$ if
$q\ne d$ and $Ext^d_A(K,A)\cong K$.
\end{enumerate}

AS regular algebras were introduced  and studied first in \cite{AS,
ATV1, ATV2}. Since then AS regular algebras and their geometry have
intensively been studied. When $d\le3$ all regular algebras are
classified. The problem of classification of regular algebras seems
to be difficult and remains open even for regular algebras of global
dimension $5$, see for example \cite{Li}. The study of Artin-Schelter regular algebras, their
classification, and finding new classes of such algebras is one of
the basic problems in noncommutative geometry. Numerous works on
this topic appeared during the last two decades, see for example
\cite{ATV1, ATV2}, \cite{Levasseur, Michel-Tate, LPWZ, FlVa, GI12, MichelAiM17}, and references therein.

The main results of the paper are Theorems I, II, III and IV. Theorem IV is formulated and proven in Section \ref{class_sec}. Here we state
Theorems I, II and III
proven under the following hypothesis:

\begin{hypothesis}
\label{hypothesis}
In  Notation-convention \ref{notconventionA}.
Suppose $X = \{x_1, x_2, \cdots, x_n\}$, and $W\subset X^{+}\setminus \,  X$ is an antichain of words of arbitrary
cardinality.
Denote by $\mathfrak{C} (X, W)$ the class of all graded
$K$-algebras $A = K\asX/ I $, generated by $X$ whose set of obstructions is $W$  (with respect to
 the degree-lexicographic ordering $\prec$ on $X^{\ast}$, where $x_1 \succ x_2 \succ\cdots \succ x_n$).
$A_W =K\asX/(W)$ is the monomial algebra determined by $W$, $A_W \in \mathfrak{C} (X, W)$.
Let $\mathfrak{N}= \mathfrak{N}(W)$ be the set of normal words modulo $W$, and let
$N= \mathfrak{N}\bigcap L$ be the set of normal Lyndon words.
\end{hypothesis}

\begin{theorem1}
\label{theorem1}
Assumptions as in Hypothesis \ref{hypothesis}.
Suppose there are no $d$-chains on $W$, but there exists a $d-1$-chain on $W$.
  Let $A = K \asX/ I$ be an arbitrary algebra in $\mathfrak{C} (X, W)$, assume that
  $A$ has polynomial growth.
  Then the following conditions hold.
\begin{enumerate}
\item
\label{gldim}
There are equalities
\begin{equation}
\label{gldim1}
\gldim A =d =\gkdim A.
\end{equation}
\item
\label{atoms_Anick}
There exists a finite set $M$ of normal words, called \emph{"atoms", in the sense of Anick}, \cite{Anick85}:
$M = \{a_1, a_2, \cdots, a_d\}\subset \mathfrak{N}, \; X\subseteq M$,
 such that:
 \begin{enumerate}
\item
\label{atoms_Anick_a}
The normal $K$-basis $\mathfrak{N}$ of $A$ (and $A_W$) can be expressed in
terms of  Anick's atoms as "a Poincar\'{e}-Birkhoff-Witt" basis
\begin{equation}
\label{normalbasis_Anickeq}
\mathfrak{N}= \{ a_1^{k_1} a_2^{k_2} \cdots  a_d^{k_d}\mid
\;  k_i \geq 0,\; 1\leq i\leq d \}.
\end{equation}
 \item
 \label{atoms_Anick_b}
The Hilbert series of $A$ satisfies
\begin{equation}
\label{Hilberteq}
H_A(z) = \prod _{1 \leq i \leq d} (1-z^{e_i})^{-1},\quad\text{where}\;\; e_i =|a_i| \geq 1, \; 1 \leq i \leq d.
\end{equation}
 \item
 \label{atoms_Anick_c}
  Any word $w \in W$ can be written as a product $w=a_ja_i, 1 \leq i < j \leq d$.
  \end{enumerate}
\item
\label{sfp_eq}
 $A$ is \emph{standard finitely presented}- the ideal $I$ has a finite reduced Gr\"{o}bner basis $\mathfrak{R}$, with
 \begin{equation}
\label{sfp_eq1}
|\mathfrak{R}|=|W|\leq d(d-1)/2.
\end{equation}
 \item
\label{main}
The following conditions are equivalent.
\begin{enumerate}
\item
\label{main1}
$|W|= d(d-1)/2$.
\item
\label{main2}
There is no atom $a\in M$ whose length is $2$.
\item
\label{main3}
$M=X$ and there exists a, possibly new, ordering of $X$:
\begin{equation}
\label{obstructeq2}
X=\{y_1 \leftarrow y_2 \leftarrow \cdots \leftarrow y_n\}, \; \text{such that}\; W= \{y_jy_i \mid 1 \leq i < j \leq n\}.
\end{equation}
\item
\label{main4}
$\gldim A = n$.
\item
\label{main_e}
$A$ is a PBW algebra in the sense of Priddy (see \cite{Priddy, PP}), equivalently, the reduced Gr\"{o}bner basis $\mathfrak{R}$ of $I$  consists of quadratic
homogeneous polynomials.
\item
\label{main6}
$H_A(z) = \frac{1}{(1-z)^n}$.
\end{enumerate}
\item
\label{main7}
Each of the equivalent conditions in part (\ref{main}) implies that $A$ is Koszul.
\end{enumerate}
\end{theorem1}
The proof of the Theorem is given in Subsection
\ref{proof1}

\begin{corollary}
\label{cor_theorem1}
Given a class $\mathfrak{C} (X, W)$, where $W$ is an arbitrary chain of monomials in $X^{+}$, suppose the monomial algebra
$A_W= K\asX/(W) \in \mathfrak{C} (X, W)$ has global dimension $\gldim A =d$ and polynomial growth.
Then every algebra $A \in \mathfrak{C} (X, W)$ is standard finitely presented and
 $\gldim A = d = \gkdim A$.
\end{corollary}

\begin{theorem2}
\label{theorem2}
Assumptions as in Hypothesis \ref{hypothesis}. Let $W\subset X^{+}$ be an antichain of arbitrary monomials.
\begin{enumerate}
\item
 \label{MonLynTheorem11}
$W$ is an anticain of Lyndon words if and only if the set of normal words
 $ \mathfrak{N}$ has the shape
\begin{equation}\label{normalbasiseq_inf}
\mathfrak{N}= \{ l_1^{k_1} l_2^{k_2} \cdots  l_s^{k_s} \mid\;l_i \in N, \;   l_1 > l_2 > \cdots >l_s,\; s \geq
1,\; k_i \geq 0, \; 1
\leq i\leq s\},
\end{equation}
where the set of normal Lyndon words $N$ is not necessarily finite.

Suppose that $W$ is an antichain of Lyndon words, so $(N,W)$ is a Lyndon pair, see Definition \ref{Lyndon_pair_def}.
 Let $A = K \asX/ I \in \mathfrak{C} (X, W)$.
\item
\label{theorem2_2}
$A$ has polynomial growth of degree $d$ if and only if $N$ has order $d$, so $N = \{l_1 > l_2> \cdots > l_d\}$.
In this case the following conditions hold.
\begin{enumerate}
\item
\label{theorem2_2a}
$\gldim A= d$;
\item
\label{theorem2_2b}
$A$ has a Poincar\'{e}-Birkhoff-Witt type $K$-basis
\begin{equation}
\label{normalbasiseq}
\mathfrak{N}= \{l_1^{k_{1}}l_2^{k_{2}}\cdots
l_d^{k_{d}}\mid k_{i} \geq 0,  1 \leq i \leq
d\}, \;\text{and}\; H_A(t) = \prod_{1 \leq i \leq d} \;  1/(1-t^{|l_i|}).
\end{equation}
\item
\label{theorem2_c}
The algebra $A$ is
standard finitely presented, more precisely,
the reduced Gr\"{o}bner basis of $I$ consists of exactly $|W|$ polynomials, where
\begin{equation}
\label{bounds_W_eq}
d-1 \leq |W| \leq \frac{d(d-1)}{2}.
\end{equation}
\end{enumerate}
\item
\label{theorem2_3}
The class $\mathfrak{C} (X, W)$ contains Artin-Schelter regular algebras, whenever \[|W| = \frac{d(d-1)}{2}\;\;\text{or}\;\; |W|= d-1, \; \text{and $N$ is connected, see Def \ref{connect_def}}.\]
\end{enumerate}
\end{theorem2}
The theorem is proven in Section \ref{More results}.
\begin{theorem3}
\label{theorem3}
 Suppose the class $\mathfrak{C}(X, W)$ is defined by an antichain of Lyndon words $W$, where
the set of Lyndon atoms $N = N(W)$ has a finite order $|N|=d$.
\begin{enumerate}
\item
The following conditions are
equivalent:
\begin{enumerate}
\item
$|W| = d-1$ and $N$ is \emph{a connected set of Lyndon atoms}, see Def \ref{connect_def};
\item
$X = \{x < y\}$, and (up to isomorphism of monomial algebras $A_W$) the set $N$ is determined uniquely:
 \[N = \{x<xy<xy^{2}<\cdots < xy^{d-2}<y\};\]
\item
$X = \{x < y\}$, and (up to isomorphism of monomial algebras  $A_W$) the set $W$ is determined uniquely:
\[W = \{xy^{i}xy^{i+1}\mid 0 \leq i \leq d-3 \}\bigcup\{xy^{d-1}\}. \]
\end{enumerate}
In this case the class $\mathfrak{C} (X, W)$ contains
the universal enveloping $U=U \Lcal_{d-1}$,  where $\Lcal_{d-1}$ is the standard filiform
algebra of dimension $d$ and nilpotency class $d-1$. $U$ is an AS-regular algebra with $\gldim U =d$.
\item
\label{MonLynTheorem4}
The following conditions are
equivalent:
\begin{enumerate}
\item $|W| = d(d-1)/2$;
\item $N= X$, and $d = n$;
\item $W = \{x_i x_j \mid 1 \leq i < j \leq d\}$.
\end{enumerate}
In this case the class $\mathfrak{C} (X, W)$ contains an abundance of (non isomorphic)  Artin-Schelter regular PBW algebras of global dimension $d$,
each of them is presented as a skew polynomial ring with square-free binomial relations, (in the sense of \cite{GI96})
and defines a solution of the Yang-Baxter equation.
\end{enumerate}
\end{theorem3}
The theorem is proven in subsection \ref{proof3}.

The paper is organized as follows.
In Section \ref{gldimsection} we study graded algebras with polynomial growth and finite global dimension and prove Theorem I.
 In Section \ref{Lyndon_pairs_sec} we recall the notion of Lyndon words and prove some new results. We introduce the notion of Lyndon pairs,
which are central for our theory. In Section \ref{More results} we prove Theorem II, and Theorem
\ref{Minim_W_Theor1} which is a purely combinatorial result on Lyndon pairs.
In Section
\ref{SectionLyndon words and  Lie algebras} we give some basic information on classes $\mathfrak{C}(X, W)$ containing the enveloping algebra $U$ of a Lie algebra (possibly infinite dimensional). We show that
for every such a class the set of obstruction
$W$ is an antichain of Lyndon words, and the corresponding set of Lyndon atoms $N=N(W)$ \emph{must be} \emph{connected}, see Definition \ref{connect_def}.
and Proposition \ref{VIPproposition}. In Section \ref{Sec.MonLiealgebras} we introduce monomial Lie algebras $\mathfrak{g}$ defined by Lyndon words -these algebras and their enveloping algebras $U\mathfrak{g}$ are central for the paper. A natural problem arises: "\emph{Classify Lyndon pairs $(N, W)$ such that the class $\mathfrak{C}(X, W)$ contains an Artin-Schelter regular algebra occurring as an enveloping
$U\mathfrak{g}$ of a concrete monomial Lie algebra $\mathfrak{g}= \Lie(X)/([W])_{Lie}$}". This motivates our
study of Gr\"{o}bner-Shirshov bases for monomial Lie ideals. In this case we find natural combinatorial conditions in terms of the Lyndon pair $(N,W)$. Propositions \ref{theor_easy_vip} and \ref{SimplifyProp}, give purely combinatorial conditions which in various cases help to decide easily and without computations, whether a set of bracketed monomials $[W]$ is a Gr\"{o}bner-Shirshov basis of the Lie ideal $J=([W])_{Lie}$ of the free Lie algebra $\Lie(X)$. In Section \ref{nilpotentsec} we present some classical Lie algebras as standard monomial Lie algebras, $\Lie(X)/([W])_{Lie}$, where $[W]$ is a finite Gr\"{o}bner-Shirshov basis of the ideal $J=([W])_{Lie}$. Our first applications of the general results on GS bases  are Corollary \ref{freeNilp_cor} and Theorem \ref{filiformProp}  which show that
the free nilpotent Lie algebra $\mathfrak{L}(m)$ of nilpotency class $m$ and the Filiform Lie algebras $\Lcal_m$ and $\Qcal_m$ are standard $\mathbb{Z}^n$ multigraded  monomial Lie algebras. In Section \ref{more.applications} we introduce the notions of \emph{a regular Lyndon pair}, \emph{a standard Lyndon pair}, and \emph{a nontrivial disconnected extension of a standard Lyndon pair}, which are essential for our classification.
Important general results are Proposition
\ref{connected_Prop}, and especially Proposition \ref{FiliformWProp2}. In Section \ref{class_sec} we prove one of the main results of the paper, Theorem IV, which gives
a classification of the two-generated Artin-Schelter regular algebras
of global dimensions $6$ and $7$, occurring as enveloping of standard monomial Lie algebras. Our general results on Lyndon pairs from the previous sections are crucial for the proof.

\section{Graded algebras with polynomial growth and finite global dimension}
\label{gldimsection}

\subsection{Anick's results on global dimension and growth of augmented graded algebras}

Given a finitely generated augmented graded algebra $A= K\asX/ I$
with a set of obstructions $W$, Anick introduces the notion of
\emph{an $n$-chain on $W$} and uses the $n$-chains to construct a resolution
of the field $K$ considered as an $A$-module. He obtains
important results on augmented graded algebras with finite
global dimension, see \cite{Anick86}, and \cite{Anick85}. It is
known that \emph{Anick's resolution is minimal in the cases when (a) $A$
is a monomial algebra, that is  $A= K\asX/ (W)$ and (b) $A$ is a PBW
algebra}, that is the ideal  $I$ is generated by quadratic
homogeneous polynomials which form a Gr\"{o}bner basis w.r.t.
degree-lexicographic ordering $\prec$ on $X^{\ast}$.
We recall now some definitions and results from \cite{Anick85} and
\cite{Anick86}.

By Notation-Convention  \ref{notconventionA} throughout the paper we assume
\[
W\bigcap X = \emptyset.
\]
\begin{definition}
\label{nchaindef}  \cite{Anick85} The set of  $n$-chains on $W$ is
defined recursively. A
 $(-1)$-chain is the monomial $1$, a
$0$-chain is any element of $X$, and a $1$-chain is a word in $W$.
An (n+1)-prechain is a word $w \in X^{+}$, which can be factored in
two different ways $w = uvq=ust$ such that $t \in W$, $u$ is an
$n-1$ chain, $uv$ is an $n$-chain, and $s$ is a proper left segment
of $v$. An $(n+1)$-prechain is an $(n+1)$-chain if no proper left
segment of it is an $n+1$-prechain. In this case the monomial $q$ is
called \emph{the tail of the $n$-chain $w$}.
\end{definition}

In fact the global dimension of a finitely presented monomial algebra can be effectively computed using some combinatorial properties of its defining relations, as shows the following corollary extracted from \cite{Anick85}, Theorem 4.
\begin{corollary}
\label{nchaincor} \cite{Anick85}
Suppose $W \subset X^{+}$ is an antichain of monomials. The monomial algebra $A_W = K\asX/(W)$ has global dimension $d$ \emph{iff} there are no $d$-chains on $W$ but there exists a
$d-1$ chain on $W$.
\end{corollary}

\subsection{Connected graded algebras with polynomial growth and finite global dimension (the general case)}
\label{proof1}
In this section we shall study finitely generated graded algebras with polynomial growth and finite global dimension.
We prove Theorem I
assuming Hypothesis \ref{hypothesis}.
Note that the set of obstructions $W$ given in the hypothesis is an arbitrary antichain of monomials in $X^{+}$, no restrictions like "$W$ consists of Lyndon words", or
"$W$ is finite" have been imposed.

\begin{proof1}
The algebra $A$ and its associated monomial algebra $A_W$ share the same set of obstructions $W$, hence the same $K$-normal basis, $\mathfrak{N}= \mathfrak{N}(W)= \mathfrak{N}(I)$, so $\gkdim A = \gkdim
A_W$,
and therefore $A_W$ has polynomial growth.
 Moreover, by hypothesis  there are no $d$-chains on $W$, but there exists a $d-1$-chain on $W$, hence, by Corollary
\ref{nchaincor}
 $\gldim A_W <\infty$.  Now a result of the author,  \cite{GI89}, Theorem II, implies that
 there are equalities
 \begin{equation}
\label{glimeq1}
 \gldim A = \gldim A_W = d = \gkdim A= \gkdim A_W,
  \end{equation}
 which proves part (\ref{gldim}).
 It is clear, that $A_W$ does not contain a free-subalgebra on two monomials (since $A_W$ has polynomial growth)
and therefore the monomial algebra $A_W$ satisfies the hypothesis of \cite{Anick85},  Theorem 5. It follows then
that there exists a finite, ordered set $M = \{a_1\leftarrow a_2 \leftarrow\cdots \leftarrow a_d\}\subset \mathfrak{N}$
 consisting of \emph{atoms} in the sense of Anick, where $\leftarrow$ is a linear ordering on $M$, and $X \subseteq M.$
$M$ satisfies the good properties listed in \cite{Anick85}, p. 297,  which straightforwardly imply (\ref{atoms_Anick}) (a), (b), (c).
By (c) each obstruction $w \in W$ can be presented in terms of atoms as $w = a_ja_i,$ where  $1 \leq i< j \leq d$
 and since there are
$\binom{d}{2}$ such pairs $i,j$, one has $|W|\leq d(d-1)/2$.

  We know that the reduced Gr\"{o}bner basis $\mathfrak{R}$ of $I$ satisfies $\overline{\mathfrak{R}} = W$, where
  $\overline{\mathfrak{R}}$ is the
  set of highest monomials of the elements of $\mathfrak{R}$.
  Moreover, any two elements $f,g \in \mathfrak{R}, f \neq g$ satisfy $\overline{f} \neq \overline{g}$. Hence
  $|\mathfrak{R}|= |W| \leq \binom{d}{2}$, and therefore $A$ is standard finitely presented. This proves part (\ref{sfp_eq}).

(\ref{main}).
We recall some details from Anick's results, \cite{Anick85}.
 A partial relation $\longrightarrow$ on the set of normal words, is introduced on p. 297, \cite{Anick85}, and the set of atoms $M$ is defined recursively as follows
$M = \bigcup_{s=1}^{\infty} M_s$, where $M_1 = X$, and given $M_{s-1}$, $M_s$ is defined by
\[
\begin{array}{ll}
M_s &= M_{s-1}\bigcup\widetilde{M_s},\;\;\text{where}\\
 \widetilde{M_s}&= \{ ab \mid a,b \in  M_{s-1}, b\leftarrow a, ab \in \mathfrak{N}, \; \text{and}\; |ab|=s \}.
 \end{array}
 \]
It is proven that the set of atoms satisfy several properties, among which the following two.

 $E_s$ (Exclusivity): for $a,b \in M_s$ \emph{exactly one of the relations} $a \leftarrow b, a=b$ and $a \rightarrow b$ \emph{is valid}.

 $R_s$ (The relations in $W$): Any word $w \in W$, with $|w|\leq s$ is a "would-be-atom", i.e.$\;w = ab,$ for some $a,b \in M_{s-1}$, with $b \leftarrow a.$

 By part (2) the set $X \subset M,$ thus $n \leq |M|= d$, and we use the restriction of the (new) linear ordering $\rightarrow$ on $X$ to "rename" the elements of $X$.
so that \[X = \{y_1\leftarrow y_2 \leftarrow\cdots \leftarrow y_n\}.\]

 (\ref{main1}) $\Longrightarrow$ (\ref{main2}). Assume $|W|= d(d-1)/2$. It follows from $R_s$, and $|M|= d$, that for each pair $a_j, a_i \in M$ with $j > i$ (equivalently,
 $a_i
 \leftarrow a_j$) one has $a_ja_i \in W$ (moreover, $a_ja_i \neq a_pa_q$, whenever $(j,i) \neq (p,q)$). In particular, the set
 \begin{equation}
\label{atomseq1}
\{ y_jy_i \mid 1 \leq i < j \leq n \}\subseteq W.
\end{equation}
Note that the inclusion (\ref{atomseq1}) is in force if and only if there is an equality
\begin{equation}
\label{atomseq2}
\widetilde{M_2}= \{ ab \mid a,b \in  X, b\leftarrow a, ab \in \mathfrak{N}\}= \emptyset
\end{equation}
Moreover, (\ref{atomseq2}) is equivalent to (\ref{main2}). Indeed,
by the recursive definition, $M_1=X$, and $M_2 = M_1\bigcup \widetilde{M_2}$, so there are no atoms of length $2$ \emph{iff}
$\widetilde{M_2}=  \emptyset$.

(\ref{main2}) $\Longrightarrow$ (\ref{main3}).
Assume (\ref{main2}). Then  $\widetilde{M_2}= \emptyset$, hence $M_2 = M_1 = X$, and  (\ref{atomseq1}) is in force.
Using induction on $s$ one verifies that $\widetilde{M_s}= \emptyset$, for all $s \geq 2$, and therefore $M= M_1 =X = \{y_1 \leftarrow y_2 \leftarrow\cdots \leftarrow y_n\}$.
Moreover $W = \{y_jy_i\mid 1 \leq i < j \leq n\}$ which proves (\ref{main3}).

(\ref{main3}) $\Longrightarrow$ (\ref{main4}).
Assume (\ref{main3}). Then $A_W$ is a quadratic monomial algebra with $\gkdim A_W = n$, and $\gldim A_W = d$, hence, $d = n$, and $\gldim A = \gldim A_W = n$.

(\ref{main4}) $\Longrightarrow$ (\ref{main_e}).
Assume $\gldim A =n$. Part (1) implies $\gldim A =\gldim A_W =\gkdim A_W = d$, by part (2) $|M| =d$, hence $n=|M|$.
   Now $X \subseteq M$, and $|X| =n$, imply  $M = X$. It follows that the set of obstructions $W$ satisfies (\ref{obstructeq2}). In particular, the reduced Gr\"{o}bner basis $\mathfrak{R}$ of $I$ consists of
homogeneous polynomials of degree $2$, so $A$ is a PBW algebra in the sense of Priddy, \cite{Priddy}.

(\ref{main_e}) $\Longrightarrow$ (\ref{main3}).
Assume $A$ is a PBW algebra in the sense of Priddy. Then $W$ consists of monomials of length $2$, and  $A_W$ is a quadratic monomial algebra.
By the hypothesis of the theorem $A_W$ has polynomial growth and finite global dimension. Hence $W$ does not contain squares.
By \cite{GI12}, Theorem 3.7, p 2163 there is a permutation $y_1, \cdots , y_n$ of $X$, such that $W$ has the shape (\ref{obstructeq2}).

The implications
(\ref{main3}) $\Longrightarrow$ (\ref{main6}) and (\ref{main3}) $\Longrightarrow$ (\ref{main1}) are clear.

(\ref{main6}) $\Longrightarrow$ (\ref{main4}).
Suppose $H_A(z) = \frac{1}{(1-z)^n}$.
Then $A$ has polynomial growth of degree $n$, hence, by part (\ref{gldim}), $\gldim A =\gkdim A = n$.
This proves part (\ref{main}).

It is known that every (quadratic) PBW algebra $A$ in the sense of Priddy,
is Koszul, see \cite{Priddy}, and also \cite{PP}, Theorem 3.1, p 84.
This proves part (5).
The theorem has been proved. $\quad \quad\quad\quad\quad\quad\quad\quad\quad \quad\quad\quad\quad\quad \quad\quad\quad\quad\quad\quad\quad\quad\quad\quad \Box$
\end{proof1}

\begin{corollary}
\label{Cor-AS}
 Let $A= K\asX/I$ be an Artin-Shelter regular algebra of global dimension $d$, and with a set of obstructions $W$.
 Suppose that there are no $d$-chain on $W$.
 Then conditions (1) through (5) of Theorem I, \ref{theorem1} are satisfied. In particular,
(i) there are equalities
 $\gkdim A = \gldim A = d$ and $\gldim A =\gldim A_W$;
(ii) $A$ is standard finitely presented and $|\mathfrak{R}|=|W|\leq d(d-1)/2$;
(iii) Moreover, if  $|W|= d(d-1)/2$, then $d= n = |X|$, and $A$ is Koszul.
\end{corollary}

Recall the following conjecture of Anick, \cite{Anick85}, p.301.
\begin{conjecture}
\label{conj_Anick}
Suppose $A$ is an augmented connected graded $K$-algebra with polynomial growth and with global dimension $d < \infty.$ Then the Hilbert series of $A$ is given by
(\ref{Hilberteq}) for some positive integers $\{e_i\}$.
\end{conjecture}
Anick states in \cite{Anick85} that  the conjecture has been verified for the cases when $A$ is commutative, $A = U$ is an enveloping algebra of a Lie algebra, when $A$ is a
monomial algebra, and also in the case when $A$ is is a Noetherian PI ring (an unpublished result of Lorenz).
Our results verify Anick's conjecture in the case when $A$ is an algebra whose associated monomial algebra $A_W$ has finite global dimension.

\section{Lyndon words, Lyndon atoms, and Lyndon pairs $(N,W)$}
\label{Lyndon_pairs_sec}
\subsection{Lyndon words and Lyndon atoms}
\label{MonLynSubsecLyn}
We recall first the notion of Lyndon words and some of their basic
properties (we refer to \cite{Lo}).
Consider
the \emph{pure} \emph{lexicographic ordering} "$ < $" on the set $X^{+}$
which extends the ordering $x_1 < x_2 < \cdots < x_n$ on $X$ as
follows: for any $u,v \in X^{+}\;$ $u < v$ \emph{iff } either $u$ is
a proper left factor of $v$ ($v = ub, b \in X^{+}$), or
$u = axb, v = ayc$, where  $x<y; \; x,y \in X,\; a,b,c \in X^*$.

Recall that
\[
\begin{array}{ll}
L 1: &\forall u \in X^{\ast}, \; a < b \Longleftrightarrow ua < ub.\\
L 2: & \text{if}\; b \notin aX^{+}, \;  \forall u, v
\in X^*,\; a < b  \Longrightarrow
 au < bv.
\end{array}
\]
So  $<$ is a linear ordering on the set $X^{+}$ compatible with the
left multiplication in $X^{+}$.
However, in the case when $b = au, u
\neq 1$, the right multiplication does not necessarily preserve
inequalities.
(For example,  $a< ax_1, \; \text{but} \;  ax_2 >
ax_1x_2$).
Furthermore, the decreasing chain condition is not
satisfied on $(X^{+}, <)$, for if $x, y \in X, x<y$, one has  $xy>
x^2y> x^3y> \cdots$.

The following simple relation between the two (fixed) orderings $\prec$ and $<$ on the free semigroup $X^{+}$
is in force.
\begin{equation}
 \label{orderdef}
u \prec v  \; \text{\emph{iff}}\;   |u| < |v|,\;
                                        \text{or}\;
                                        |u| = |v|\; \text{and}\; u >
                                        v.
\end{equation}
For example, if $X= \{x < y\}$, one has $xxy < xyy$, but $xxy \succ xyy$.

\begin{defnotation}
\label{Lyndonworddef} \cite{Lo}
A nonperiodic word $u \in
X^{+}$ is \emph{a Lyndon word} if it is minimal, with respect to
$<$, in its conjugacy class. In other words, $u = ab, a,b\in X^{+}$
implies $u < ba$. The set of Lyndon words in $X^{+}$ will be denoted
by $L$, $L_s$ denotes the set of all Lyndon words of length $s, s = 1, 2, 3, \cdots$.
\end{defnotation}
We shall use notation, terminology and some results from \cite{GIF}.
\begin{defnotation}
\label{Lyndonatomsdef} \cite{GIF}
Given an antichain $W$ of Lyndon words, the set of all $W$-\emph{normal
Lyndon words} is denoted by $N= N(W)$, we shall refer to it as \emph{the set of Lyndon atoms corresponding to} $W$. By definition the set of Lyndon atoms $N$
satisfies
\begin{equation}
 \label{liealgeq01}
N =N(W) = L  \bigcap \mathfrak{N}(W) \; \text{where $L$ is the set of Lyndon words} .
\end{equation}
Assume that $A = k \asX/ I$ is a graded $K$-algebra such that its set of
obstructions  $W$ consists of Lyndon monomials, then we shall also refer to the normal Lyndon words
 $N = N(W)$ as \emph{the set of Lyndon atoms for} $A$.
\end{defnotation}

\begin{definition}
\label{Lyndon_pair_def}
Let $W$ be an antichain of Lyndon words, and  $N =N(W)\subset L$ be the corresponding set of Lyndon atoms. Then the pair of sets $(N,W)$
will be called \emph{a Lyndon pair}.
\end{definition}
The canonical duality between the notions "an antichain $W$ of Lyndon words" and "its set of Lyndon atoms, $N=N(W)$"  is discussed in the next subsection, see also \cite{GIF}.
\begin{fact}
\label{LyndonThm}
\emph{(Lyndon's Theorem)}, \cite{Lo02} 11.5.
Every word $w \in X^{+}$ can be written uniquely as a non increasing product $w = l_1
l_2 \cdots l_s$ of Lyndon words, $l_1\geq l_2\geq \cdots \geq l_s$.
\end{fact}
The following useful facts are
extracted from  \cite{Lo},  Section 5.1..

\begin{facts}
\label{fact_lyndonwords}
\begin{enumerate}
\item
\label{fact_lyndonwords3}
For all $w \in L$, the equality $w = ab$,
with $a,b\in X^{+}$, implies $a < w < b.$
\item
\label{fact_lyndonwords1}
 If $a<b$ are Lyndon words then $ab$ is a
Lyndon word, so $a < ab < b$.
\item
\label{fact_lyndonwords4}
 If $b$ is the longest proper right segment
of $w$ which is a Lyndon word, then $w = ab$, where $a$ is a Lyndon
word. This is called \emph{the  (right) standard factorization of $w$} and
denoted as $w =(a,b) =(a,b)_r$.
\item
Analogously, suppose $u$ is the longest proper left segment of $w$ which is a Lyndon word, then $w = uv$, where $v$ is also a Lyndon
word. This determines the \emph{left standard factorisation of $w$} denoted by $w=(u,v)_{l}$
\end{enumerate}
\end{facts}
\begin{notation}
\label{overlapnotation} \cite{GIF} For the monomials $a, b\in
X^{+}$ we  write $\;\overbrace{a,b}\;$ if  $a = uv$, $b = vw$, where
$u,w \in X^{\ast}$, $v, uw \in X^{+}$ ($b = aw,$ or $a = ub$ is
possible).
We write $\overbrace{a, \tau, b}$ if  $\tau \sqsubset ab,$ and
$\tau$ overlaps with $a$ and $b$ so that $(\overbrace{a,\tau})$ and
$(\overbrace{\tau, b})$.
\end{notation}

\begin{lemma}
\label{MonLynLemOverlap} \cite{GIF}
\begin{enumerate}
\item
\label{MonLynLemOverlap1} Let $uv$ and $vw$ be Lyndon words, $u, v \in X^{+}$. Then
$uvw$ is a Lyndon word.
\item
\label{MonLynLemOverlap2} If $a, b \in L$ and
 $\;\overbrace{a,b}\;$, then $a<b$.
\item
\label{MonLynLemOverlap3} Suppose that $a, b, w \in L$, with
$\overbrace{a, w, b}$. Then $a < w < b$.
\item
\label{MonLynLemOverlap4} Let $a < b$ be Lyndon words. Then $a^kb^l$
are Lyndon words for all $k,l \geq 1$.
\item
\label{MonLynLemOverlap5} If $w = (a,b)$ is the right standard factorization of $w$, then the (right)
standard factorization of $(a^kb)$ is $(a, a^{k-1}b)$.  Analogously, if $w = (a,b)_l$ is the left standard factorization
of the Lyndon word $w$ then  the left standard factorization of $ab^k$ is
$(ab^k)_l = (ab^{k-1}, b)$.
\end{enumerate}
\end{lemma}

\begin{lemma} \label{MonLynLemFactorization}\cite{GIF}
Let $l_1 \geq l_2 \geq \cdots \geq l_s$ be Lyndon words, $s \geq 2$.
If a Lyndon word $u$ is a subword of $l_1 l_2 \cdots l_s$, then $u$
is a subword of $l_i$, for some $i$, $1 \leq i \leq s$
\end{lemma}

\subsection{More results on Lyndon words}
We shall prove some additional results on Lyndon words which will be used in the paper.
\begin{lemma}
\label{VIP_lemma11}
Suppose $a, w\in L,  |a|, |w|\geq 2$,  and $a$  is a proper subword of $w$.
\begin{enumerate}
\item If $w = (u,v)$ is the right standard factorization of $w$,
and $a$ is not a left segment of $w$,
then either $a \sqsubset u$, or $a \sqsubseteq v.$
\item If $w = (p,q)_l$ is the left standard factorization of $w$,
and $a$ is not a right segment of $w$,
then either $a \sqsubseteq p,$ or $a \sqsubset q$.
\end{enumerate}
\end{lemma}
\begin{proof}
We sketch a proof of (1).
Assume  the contrary. Then $a$ overlaps with $v$, so that $a = a_1a_2 \in L$, $v = a_2v_2\in L$, where $|a_1| \geq 1, |a_2| \geq 1$, and $a_1a_2v_2$ is a proper segment of $w$.
By Lemma  \ref{MonLynLemOverlap} (2) the product $a_1a_2v_2 = a_1v$ is a Lyndon word, moreover,
$a_1v$ is a proper right segment of $w$ with $|a_1v| >|v|$. By assumption $v$ is the longest
Lyndon word which occurs as a right segment of $w$, a contradiction.
The proof of part (2) is analogous.
\end{proof}
\begin{proposition}
\label{N_W_Prop}
Let $N$ be a  set of Lyndon words closed under
taking Lyndon subwords, let $W= W(N)$ be the corresponding antichain
of Lyndon words, see Remark \ref{W(N)remark} (2).
\begin{enumerate}
\item
\label{N_W_Prop21}
If $u, v \in N, u< v$ and no Lyndon atom $a \in
N$ satisfies $u < a <v$, then $uv \in W$.
\item
\label{N_W_Prop22} Suppose $N_0 = \{l_1<l_2< \cdots < l_k\} \subseteq N$ is a discrete convex in
$N$,  that is $N$ does not contain a word $u$ with $l_i < u < l_{i+1}$.
Then $l_il_{i+1} \in W$ for every pair $l_i, l_{i+1} \in N_0$.
\item
\label{N_W_Prop23}
If $N$ is a finite set of order $d$, then $d-1 \leq |W|$.
\end{enumerate}
\end{proposition}
\begin{proof}
(\ref{N_W_Prop21}).  By hypothesis there is
no Lyndon atom $a$ such that $u < a < v$. By Facts \ref{fact_lyndonwords} $uv \in L$ and $u
< uv < v$, hence  $uv \in (W)$. We have to show that $uv
\in W.$
Assume the contrary, then (since $uv\in (W)$) some $w\in W$, is a proper subword of
$uv$. Let $w = ab$, $a, b \in L$ be the standard factorization of
$w$. By Remark \ref{N_W_Prop_Fact} (\ref{N_W_Prop1}),  $a,b \in
N,$ and by Facts \ref{fact_lyndonwords} $a < ab<b$.  Three cases are
possible: (i) $a$ overlaps with $u$ and $v$, so that
$\overbrace{u,a,v}$ ; (ii) $b$ overlaps with $u$ and $v$, so that
$\overbrace{u, b, v}$; (iii) $u = u_1a$,  $\;v=bv_1$, $u_1v_1 \neq
1$. Assume (i) is in force. Then Lemma \ref{MonLynLemOverlap}
(\ref{MonLynLemOverlap3}) implies $u < a< v$, which contradicts the
hypothesis. Similarly, in case (ii) the relation $\overbrace{u, b,
v}$ implies  $u <b<v$, which is impossible. Suppose (iii) holds.
Without loss of generality we may assume that $a$ is a proper right
segment of $u \in L$, so $u < a$. Now the inequalities
\begin{equation}
\label{eqinequalities} u < a < ab < b \leq v; \quad  a \in N,
\end{equation}
where $a,b \in N$ contradict the hypothesis. The case when $b$ is a
proper left segment of $v$  is analogous, which proves (\ref{N_W_Prop21}).
Part (\ref{N_W_Prop22}) follows straightforwardly from
(\ref{N_W_Prop21}). Clearly, (\ref{N_W_Prop22}) implies (\ref{N_W_Prop23}).
\end{proof}

\begin{remark}
\label{N_W_Prop_Fact} \cite{GIF} Let $(N,W)$ be a Lyndon pair.
Then
\begin{enumerate}
\item
\label{N_W_Prop1} If $u$ is a proper Lyndon subword of some $w \in
W$ then
$u \in N$. In particular, if $w \in
W$ then the (right) standard factorisation  $w = (u,v)$
and the left standard factorisation $w = (a,b)_l$ satisfy  $a,b, u, v \in N.$
\item
\label{N_W_Prop3a}  If $N$ has finite order $d$, and $s$ is the
maximal length of words in $N$ then $W$ is also finite with $|W|
\leq d(d-1)/2$, and the length of each $w\in W$ is at most $2s.$
\item
\label{N_W_Prop4} Assume  $W$ is finite and let $m$ be the maximal
length of words in $W$. Then $N$ is finite if and only if every word
$l \in N$ has length $|l|\leq m-1$.
\end{enumerate}
\end{remark}

\subsection{Canonical dualities on sets of monomials}
For completeness  we consider the canonical dualities between (a) the notions "obstruction set", and "order ideal of monomials", see \cite{Anick86},
and (b) the notions "antichains $W$ of Lyndon words" and "Lyndon atoms, $N=N(W)$", \cite{GIF}.

Let  $I$ be an ideal in $K\asX$, $A =
K\asX/I.$
We use the notation and terminology from the previous sections, in particular (\ref{N1eq1}), (\ref{X1eq2}), and Definition-notation \ref{Lyndonatomsdef}.
We have seen that the set of monomials
 $\overline{I}$ contains
uniquely determined maximal antichain of monomials $W= W(I)$
(with respect to $\sqsubset$).
$W$ satisfies  (i) every element of $W$ is
minimal with respect to "$\sqsubset$", and (ii) each $a\in \overline{I}$ contains
as a subword (a segment) some $v \in W$, i.e. $v \sqsubseteq a$. Denote by $(W)$ the two-sided ideal of $K\asX$ generated by $W$.

Clearly, the set of normal words $\mathfrak{N}=\mathfrak{N}(W)$ is uniquely determined by the set of obstructions $W$
and  is characterized by the properties:
\[
\begin{array}{llll}
\textbf{N1.}\;\; X \subseteq \mathfrak{N}; &\textbf{N2.}\;\; u \in \mathfrak{N}, v \sqsubseteq u \Longrightarrow v \in \mathfrak{N}; &
\textbf{N3.}\;\; u \in  \mathfrak{N} \Longleftrightarrow u \notin (W).
\end{array}
\]
We have seen in the introduction that every
set of monomials  $\mathfrak{N} \subseteq X^{\ast}$
which satisfies conditions \textbf{N1} and \textbf{N2} determines uniquely an antichain of monomials $W\subset X^{+}$, so that condition \textbf{N3} holds and $W$ is the
set of obstructions for $\mathfrak{N}$.

 \emph{An obstruction set} $W= W(\mathfrak{N})$ \emph{is characterized by the following properties:}
\[\begin{array}{l}
\textbf{V1.}\quad  W \bigcap X = \emptyset.\\
\textbf{V2.}\quad \forall\; u \in X^{\ast}, v \in W, u \sqsubset v,  u \neq v \Longrightarrow u \in \mathfrak{N}.\\
\textbf{V3.}\quad \forall a \in X^{+},  v \in W, v \sqsubseteq a \Longrightarrow a \notin \mathfrak{N}.
\end{array}
\]
\begin{definition}
\label{obstructions_and_oim_def}
\begin{enumerate}
\item
A set of monomials  $\mathfrak{N} < X^{+}$ which satisfies conditions
\textbf{N1} and \textbf{N2} is called \emph{an order ideal of monomials} (o.i.m.), see \cite{Anick86}.
 \item
The antichain of monomials  $W= W(\mathfrak{N})$ described above is called \emph{the set of obstructions for} $\mathfrak{N}$.
When $W$ is considered as $W = W(I)$, it is also referred to as \emph{the set of obstructions for} \emph{the algebra} $A =K\asX /I.$
\end{enumerate}
\end{definition}
\begin{remark}
\label{W(N)remark}
\begin{enumerate}
\item
Suppose that $W \subset X^{+}$ is an antichain of Lyndon monomials, and $\mathfrak{N}=\mathfrak{N}(W)$.
Then $W$ determines
uniquely a set of Lyndon atoms $N=N(W):= \mathfrak{N}(W) \bigcap L$. It satisfies
\[\begin{array}{ll}
\text{\textbf{C1.}}& X \subseteq N.\\
\text{\textbf{C2.}}&\forall v \in L , \forall u \in N,   v
\sqsubseteq u \Longrightarrow v \in N.\\
\text{\textbf{C3.}}& u \in N \; \Longleftrightarrow \; u \in L \;
\;\text{and}\;\; u \notin (W).
\end{array}
\]
\item
Conversely, each set $N$ of Lyndon words satisfying conditions
\textbf{C1} and \textbf{C2} determines uniquely an antichain of
Lyndon monomials $W= W(N)$, such that condition \textbf{C3} holds,
and $N$ is the set of Lyndon atoms corresponding to $W$.
In fact $W$ is the unique maximal antichain of the complement $L - N.$
\end{enumerate}
\end{remark}
We systematize the above discussion and add some results from \cite{GIF}
in the following proposition.

\begin{proposition}
\label{N-WProp0}
\begin{enumerate}
\item
There exists a one-to-one correspondence between
the set $\mathbb{V}$ of all antichains $W$ in $X^{+}$ with $W \bigcap X =\emptyset$
 and the set $\aleph$  consisting of all order ideals of monomials
 $\mathfrak{N}$, satisfying condition \textbf{N1} and \textbf{N2}.
  In notation as above this correspondence is defined as
\[ \begin{array}{lll}
\Phi:& \mathbb{V} \longrightarrow \aleph \quad  &W \mapsto \mathfrak{N}(W)\\
\Phi^{-1}:&  \aleph \longrightarrow \mathbb{V}\quad  &\mathfrak{N}  \mapsto W(\mathfrak{N}).
\end{array}\]
\item\label{N-WProp02}
There are equalities
\[\mathfrak{N}(W(\mathfrak{N})) = \mathfrak{N}; \quad W(\mathfrak{N}(W))=W,\]
and each pair $(\mathfrak{N}(W), W)$ (respectively $(\mathfrak{N}, W(\mathfrak{N}))$  obtained
via this correspondence satisfies condition \textbf{N3}.
\item
\label{N-WProp01}
There exists a one-to-one correspondence between
the set $\mathbb{W}$ of all antichains $W$ of Lyndon words with
$X\bigcap W = \emptyset$ and the set $\mathbb{N}$ consisting of all
sets $N$ of Lyndon words satisfying conditions \textbf{C1} and
\textbf{C2}. In notation as above this correspondence is defined as
\[ \begin{array}{lll}
\phi:& \mathbb{W} \longrightarrow \mathbb{N} \quad  &W \mapsto N(W)\\
\phi^{-1}:&  \mathbb{N} \longrightarrow \mathbb{W}\quad  &N \mapsto W(N).
\end{array}\]
\item\label{N-WProp02}
There are equalities
\[N(W(N)) = N; \quad W(N(W))=W,\]
and each pair $(N= N(W), W)$ (respectively $(N, W=W(N))$  obtained
via this correspondence satisfies condition \textbf{C3}.
\item \cite{GIF}
\label{N-WProp04} Each finite antichain of Lyndon words $W \in \mathbb{W}$
determines a monomial algebra $A_W = K \asX /(W)$ of finite global
dimension, $\gldim A_W \leq |W|+1$  .
\item  \cite{GIF}
\label{N-WProp023}
If  $ N \in \mathbb{N}$ is a finite set of Lyndon words of order $d$, then the corresponding antichain $W = W(N)$ is also finite
with $|W| \leq d(d-1)/2$.
\item
\label{N-WProp30} Each $N\in \mathbb{N}$ with \textbf{C1} and \textbf{C2} determines uniquely a
monomial algebra  $A_W = K \asX /(W)$, with a set of defining relation $W = W(N)$
and a set
of Lyndon atoms precisely $N$.
The algebra $A_W$ has polynomial growth of degree $d$ \emph{iff} $|N|=d$.
\end{enumerate}
\end{proposition}

\begin{example}
$X = \{x < y\},  \; A =K\asX/(xy^2+y^2x,\quad x^2y-yx^2 ).$
The algebra $A$ is s.f.p.- the defining relations form a Gr\"{o}bner basis of the ideal $I= (xy^2+y^2x,\quad x^2y-yx^2 )$.
The set of obstructions
$W= \{xxy,\quad xyy\}$
consists of Lyndon words.
The set of Lyndon atoms is
$N= \{d_1 = y > d_2= xy > d_3 = x\}$,
and $(N,W)$ is the corresponding Lyndon pair.
The normal basis of $A$ is
\[
\mathfrak{N} = \{d_1^{\alpha_1}d_2^{\alpha_2}d_3^{\alpha_3}\mid
\alpha_i \geq 0\}.
\]
$A$ is a classical AS-regular algebra of $\gldim A =3$, \textbf{type A}.
\end{example}

\section{Some combinatorial results on Lyndon pairs}
\label{More results}
In this section $X = \{x_1< x_2 < \cdots < x_n\}$, and $(N,W)$ is a Lyndon pair in $X^{+}$.

\subsection{Proof of Theorem II}
\label{proof2}
\begin{lemma}
\label{VIPlemma1}
Suppose $X = \{x_1, x_2, \cdots, x_n\}$, and let $W \subset X^{+} \setminus X$ be an antichain of monomials of arbitrary cardinality.
Let $\mathfrak{N}= \mathfrak{N}(W)$ be  the set of normal words modulo $W$, and
let $N= \mathfrak{N}\bigcap L$ be the set of normal Lyndon words.
The following conditions are equivalent.
\begin{enumerate}
\item
 $W$ is an antichain of Lyndon words.
 \item The set $\mathfrak{N}$ of normal words has the shape given in \ref{normalbasiseq_inf}.
\end{enumerate}
\end{lemma}

\begin{proof}
As usual,  $A_W = K\asX/(W)$ denotes the monomial algebra with defining relations $w=0,  w \in W$.  Then $A_W\in \mathfrak{C} (X, W)$, and
the set $\mathfrak{N}$ is a $K$-basis of $A_W.$
Note that the set of normal Lyndon words $N = \mathfrak{N}\bigcap L$ satisfies conditions \textbf{C1} and \textbf{C2},
whenever $W$ is an arbitrary antichain of monomials, but $W$ is not necessarily a set of Lyndon words. Indeed,   $X \subseteq \mathfrak{N}\bigcap L$, so $N$ satisfies condition
\textbf{C 1}. Suppose $u\in N$, and $v\in L$ is a subword of $u$. Then, since $u \in \mathfrak{N}$, the word $v$ is also normal, (see \textbf{N2}) thus $v \in
\mathfrak{N}\bigcap L = N$, and therefore condition \textbf{C2} also holds.
Use now the canonical duality from Proposition \ref{N-WProp0} and consider the pair $(N, W_0=W(N)).$
Notice that if $W$ is an arbitrary antichain of monomials, and $N = \mathfrak{N}\bigcap L$, then, in general,  an inequality
$W_0=W(N) \neq W$ is possible.
We claim that $W_0 = W$ holds if and only if (\ref{normalbasiseq_inf}) is in force.

(1) $\Longrightarrow$ (2). Assume that $W$ is an antichain of Lyndon words, so $N= N(W)$ is the set of Lyndon atoms corresponding to $W$. Then
$A_W$ is a monomial algebra defined by Lyndon words in the sense of \cite{GIF}, hence, by Theorem A \cite{GIF},
the set of normal words $\mathfrak{N}=\mathfrak{N}(W)$ has the shape  (\ref{normalbasiseq_inf}).

 (2) $\Longrightarrow$ (1).
Assume that the set of normal words $\mathfrak{N}$ is presented via (\ref{normalbasiseq_inf}).

We have seen that the canonical duality from Proposition \ref{N-WProp0} defines uniquely the pair $(N, W_0),$
where $W_0 = W(N)$ is an antichain of Lyndon words.
Then the corresponding set of Lyndon atoms, $N(W_0) = N(W(N)) = N$.  Now
$A_0 = K\asX /(W_0)$ is a monomial algebra defined by Lyndon words in the sense of \cite{GIF}, and its set of Lyndon atoms is $N$.
Use now the first duality in Proposition \ref{N-WProp0}, and consider the pair $(W_0, \mathfrak{N}_0 = \mathfrak{N}(W_0))$.
  The set $\mathfrak{N}_0$ is the normal basis for the monomial algebra $A_0$. Hence by Theorem A \cite{GIF} again, $\mathfrak{N}_0$
  has a "PBW"-type presentation via the Lyndon atoms in $N$
which is exactly the set given in (\ref{normalbasiseq_inf}). Therefore $\mathfrak{N}(W_0) = \mathfrak{N}(W)$, so
$W_0= W(\mathfrak{N}(W_0)) = W(\mathfrak{N}(W)) = W $. It follows that the obstruction set $W$ is an antichain of Lyndon words.
\end{proof}

\begin{proof2}
Part (1) of Theorem II is proven by Lemma \ref{VIPlemma1}.

Suppose that $W$ is an antichain of Lyndon words, so $(N,W)$ is a Lyndon pair.
 Let $A = K \asX/ I \in \mathfrak{C} (X, W)$, let $A_W = K \asX/ (W)$ be the corresponding monomial algebra. Clearly, $A_W\in \mathfrak{C} (X, W)$.

 (\ref{theorem2_2}).
The two algebras $A$ and $A_W$
share the same set of obstructions $W$, the same set of Lyndon atoms $N = N(W)$, the same normal $K$-basis $\mathfrak{N}$ and the same Hilbert seies. In particular, $A_W$ is a monomial algebra
defined by Lyndon words in the sense of \cite{GIF}, so the results in \cite{GIF} are true for $A_W$.

There is an obvious equality $\gkdim A = \gkdim A_W.$ Therefore by  \cite{GIF}, Theorem A, (2) each of the algebras $A$ and $A_W$ has a polynomial growth of degree $d$ \emph{iff} $|N| = d$.
Assume $A$ has polynomial growth of degree $d$. Then $N$ has order $d$, say,  $N = \{l_1 > l_2> \cdots > l_d\}$.
Then Theorem A, \cite{GIF} implies $\gldim A_W = d$.  Hence there is a $d-1$-chain on
$W$ and there is no $d$-chain on $W$. Now the algebra $A$ satisfies the hypothesis of Theorem I, so $\gldim A = \gkdim A=d$, which gives (\ref{theorem2_2a}).

The two algebras $A$ and $A_W$
share the same  normal $K$-basis $\mathfrak{N}$ and the same Hilbert series, therefore (\ref{theorem2_2b}) is in force.
The monomial algebra $A_W$ satisfies the hypothesis of  \cite{GIF}, Theorem B,
and therefore its set of obstructions $W$ satisfies the inequalities
(\ref{bounds_W_eq}), which implies (\ref{theorem2_c}).
Part (\ref{theorem2_3}) follows from Theorem III, and is given for completeness.
The theorem has been proved. $\quad \quad\quad\quad\quad\quad \quad\quad\quad\quad\quad\quad\quad\quad\quad\quad \Box$
\end{proof2}

\subsection{Up to isomorphism, for every $d\geq 2$  there exists unique Lyndon pair $(N,W)$, where $N$ is connected and $|W| = |N|-1 = d-1$.}
Suppose $(N,W)$ is a Lyndon pair  in $X^{+}$, where $N=N(W)$ has finite order $|N|=d$.
It follows from Theorem II that
$W$ is a finite set with
\[
d-1 \leq |W|
\leq \frac{d(d-1)}{2}.\]
It follows from \cite{GIF}, Theorem B (2) that the upper bound
$|W| = d(d-1)/2$ is attained if and only if  $W = \{x_i x_j \mid 1 \leq i < j \leq d\}$, or equivalently,
$d=n$ and $N= X$, see Theorem II (2). In this case the class $\mathfrak{C}(X, W)$ contains all binomial skew polynomial rings with square-free relations generated by $X$, in the sence
 of \cite{GI96}. Each such a ring $A$ is an AS-regular PBW algebra of global dimension $n$. For example, for $g = 8$ there are more than $2
000$
non-isomorphic skew-polynomial rings with binomial square-free solutions, each of which is an AS-regular and defines a set-theoretic solution of YBE.
It is interesting to describe the other "extreme" case of Lyndon pairs, when the order $|W|$ attains the lower bound.
Theorem  \ref{Minim_W_Theor1} shows that \emph{if $N$ is connected} then the pure \emph{numerical datum}
 \begin{equation}
\label{numericaleq}
 |W|=d-1, \quad |N|= d
\end{equation}
 determines uniquely (i) the order of $X$, $\;|X|=2$, and (ii) unique Lyndon pair $(N, W)$ in $X^{+}$
 (up to isomorphisms of monomial algebras $A_W$).
Corollary \ref{MinimalWCor} implies that in this case the class $\mathfrak{C}(X, W)$
  contains
the universal enveloping $U=U \Lcal_{d-1}$,  where $\Lcal_{d-1}$ is the standard filiform
algebra of dimension $d$ and nilpotency class $d-1$. $U$ is an AS-regular algebra with $\gldim U =d$.

\begin{theorem}
\label{Minim_W_Theor1}
For every integer $d\geq 2$,  there exists unique  (up to isomorphism)
Lyndon pair $(N,W)$, where $N$ is a connected set of finite order $d$ and $|W| = d-1$.
Moreover, the following conditions are
equivalent:
\begin{enumerate}
\item
\label{Minim_W_Theor1a}
$|W|= d-1$ and $N$ is \emph{a connected set of Lyndon atoms}, see Def \ref{connect_def};
\item
\label{Minim_W_Theor2}
$N$ is \emph{a connected set of Lyndon atoms} and
\begin{equation}
\label{Weq}
W = \{l_il_{i+1}\mid 1 \leq i \leq d.\}
\end{equation}
\item
\label{Minim_W_Theor3}
 The set $X$ has order $2$, so
 $X = \{x< y\}$, and (up to isomorphism of monomial algebras $A_W$) the set of atoms $N= N(W)$ has the shape
\begin{equation}
\label{NFiliformeq}
N = \{x < xy <  xyy < \cdots <xy^{d-3}< xy^{d-2}< y\}.
\end{equation}
\item
\label{WFiliformeq}
\[W = \{xy^{i}xy^{i+1}\mid 0 \leq i \leq d-3 \}\bigcup\{xy^{d-1}\}. \]
\end{enumerate}
\end{theorem}
\begin{proof}
(\ref{Minim_W_Theor1a}) $\Longleftrightarrow$  (\ref{Minim_W_Theor2}). By Proposition \ref{N_W_Prop} (2) there is an inclusion of sets
\[\{l_il_{i+1}\mid 1 \leq i \leq d-1 \} \subseteq W.\] The equivalence of (1) and (2)  is straightforward.

(\ref{Minim_W_Theor2} ) $\Longrightarrow$  (\ref{Minim_W_Theor3}).
Assume $W$ has the shape given in (\ref{Weq}). We shall prove (\ref{Minim_W_Theor3}) in several steps.

We start with an useful lemma.
\begin{lemma}
\label{usefulrem}
Suppose (\ref{Weq}) holds, and $N$ contains a subchain  $a < b < c$. Then
\begin{enumerate}
\item[(a)] $ac$ is not in $W$.
\item[(b)] Furthermore, if each proper Lyndon subwords $w \sqsubset ac$
is an atom, $w \in N$, one has  $ac \in N.$
\end{enumerate}
\end{lemma}
\begin{proof}
We shall prove (a). By hypothesis the set of Lyndon atoms is $N = l_1 <l_2 < \cdots < l_{d+1},$ and by (\ref{Weq}) $W$ is the set of words $l_il_{i+1}, 1 \leq i \leq d.$  Assume now that $ac\in W$, so $ac= l_il_{i+1}$ for some $i, 1 \leq i \leq d.$
Note that $a \neq l_i,$ otherwise $c= l_{i+1},$ and $l_i=a < b < c= l_{i+1}$, with $b \in N$, is impossible. Therefore $a \neq l_i$ and either
(i) $l_i$ is a proper left segment of $a$, or (ii) $a$ is a proper left segment of $l_i$. Suppose (i) holds. Then
$l_i$ is a proper left segment of $a$ and $a$ overlaps with a proper left segment of $l_{i+1}$, that is
$\overbrace{l_i, a, l_{i+1}}$, see Notation \ref{overlapnotation}.  Now Lemma \ref{MonLynLemOverlap} part (3) implies that
\[l_i < a < l_{i+1}.\]
which is impossible, since $a \in N$. In the case (ii)  $c$ overlaps with a right segment of $l_{i}$, and $l_{i+1}$ is a proper right segment of $c$.  So $\overbrace{l_i, c, l_{i+1}}$, and Lemma \ref{MonLynLemOverlap} part (3) implies that
\[l_i < c < l_{i+1}, \]
a contradiction. It follows that $ac$ is not in $W$.
Part (b) is straightforward.
\end{proof}

\emph{Step 1.}  $|X| = 2$. Assume the contrary, then $X$ has at least three distinct elements, say
$x<y<z$. By assumption $W$ contains only products of Lyndon atoms which are "neighbors", so by Lemma \ref{usefulrem} $xz \in N$, and $N$ contains the subset
$x < xz<y< z$. Now by induction on $k$ (using analogous argument) one shows that $N$ contains an infinite subchain
\[x< xz < xz^2 <xz^3 < \cdots< x^k< \cdots <y < z,\]
which contradicts the hypothesis. By convention $X$ has at least 2 elements, so $X = \{x< y\}$.
We shall describe $N$.

The case $|N| = 3$ is trivial, one has $N = \{x < xy < y\}$. Assume $|N| \geq 4$.

\emph{Step 2.}  We prove that $N$ does not contain simultaneously $xxy$ and $xyy$.
Otherwise, $N$ contains the subchain $x < xxy<xy <xyy<y$, hence $xxy$ and $y$ are not neighbours, so $u=xxyy$ is not in $W$, and since every proper Lyndon subword of $u$ is in
$N$, Lemma \ref{usefulrem} implies $xxyy \in N$. Therefore $N$ contains $x < xxy<xxyy<xy <xyy<y$.
Lemma \ref{usefulrem} again implies $xxyxy \in N$, so
$\{x< xxy < xxyxy <xxyy< xy <xyy<y\} \subseteq N$ .
 Next one proves by induction on $k$ (and using Lemma \ref{usefulrem}) that $N$ contains an infinite set of Lyndon atoms:
$\{xxy(xy)^k, k \geq 1\} \subseteq N$,  which contradicts the hypothesis again.

From now on, without loss of generality, we may assume $xxy \in W.$ (If we assume $xyy \in W$ we shall obtain an isomorphic monomial algebra $A_W$). Since $N$ is connected, it
contains a Lyndon word of length 3,  hence $xyy \in N$. Therefore
$l_1 = x < xy =l_2 <xyy=l_3< \cdots <l_d = y$ is a subchain in $N$.

\emph{Step 3.} We shall prove that $N$ satisfies (\ref{NFiliformeq}).
Assume the contrary, then $N$ contains a Lyndon atom of the shape $xy^jxy^{k}\cdots, \; k\geq j \geq 1$.
Let $u_0$ be \emph{the shortest Lyndon atom of such shape}. By hypothesis $N$ is closed under taking Lyndon subwords, and it is easy to prove that $u_0= xy^jxy^{j+1}$, for some
$j$. It follows that
\[m(N) \geq |u_0|= 2j+3.\]

Note that $w= xy^{2j+1}\in N$, in particular $xy^{j+2}\in N$.
Indeed, by assumption $u_0 \in N$ is the shortest atom with
$u_0 \neq xy^t$. Moreover, since $u_0 \in N, |u_0|= 2j+3$ and $N$ is connected, one has
\begin{equation}
xy^{s}\in N, 1 \leq s \leq 2j+1.
\end{equation}
One has $|xy^{2j+1}|= 2j+2$, and $j+2 \leq 2j+1$.

Consider the inequalities of Lyndon atoms
\[xy^j < u_0= xy^jxy^{j+1} <xy^{j+1}<y, \]
we claim that the product  $u_1 = u_0y= xy^jxy^{j+2}$ is in $N$. Indeed, by  Lemma
\ref{usefulrem} (a) $u_1 = u_0y= xy^jxy^{j+2} $ is not in $W$.
The set of proper Lyndon subwords of $u_1$ is  $\{xy^s, 0 \leq s \leq j+2\} \cup \{u_0= xy^jxy^{j+1}\}\subset N$, and therefore  by  Lemma
\ref{usefulrem} (b), $u_1 = xy^jxy^{j+2}\in N.$
The following  sub-chain of atoms is contained in $N$:
\begin{equation}
\label{Neq1}
u_0= xy^jxy^{j+1} <   u_1 = xy^jxy^{j+2} <   \cdots  <xy^{j+1}\in N.
\end{equation}
A similar argument implies that
$u_0= xy^jxy^{j+1} < xy^j(xy^{j+1})^2   \in N$, and using induction on $k$, Lemma \ref{usefulrem},  and a similar argument as above
one proves
that $N$ contains an infinite increasing chain of Lyndon atoms:
\begin{equation}
\label{Neq2}
\begin{array}{ll}
u_0          &= xy^jxy^{j+1} < xy^j(xy^{j+1})^2 < xy^j(xy^{j+1})^3 < \cdots \\
             &< xy^j(xy^{j+1})^k < \cdots
             \in N, \; k \geq 1,
             \end{array}
\end{equation}
which contradicts $|N|= d+1.$
It follows that $N$ satisfies (\ref{NFiliformeq}).

The implications (\ref{Minim_W_Theor3} ) $\Longrightarrow$  (\ref{WFiliformeq}) and (\ref{Minim_W_Theor3} ) $\Longrightarrow$  (\ref{Minim_W_Theor2})
are clear.
 This proves the theorem.
\end{proof}

\section{Classes $\mathfrak{C}(X, W)$ containing the enveloping algebra of a Lie algebra}
\label{SectionLyndon words and  Lie algebras}
\subsection{Finitely generated graded Lie algebras and their enveloping algebras, standard Lyndon bases}
Our main references are
\cite{Lo}, \cite{Lo02} \cite{Re}, \cite{LR}, \cite{jacobson}.

In this subsection we set notation and recall some classical facts about finitely generated graded Lie algebras and their enveloping algebras
which will be used in the sequel.
As usual, $X = \{x_1, \cdots, x_n\}$,  $n \geq 2,$ is a finite set,
$K$ is a field of characteristic $0$,  $\mathfrak{L}=\Lie (X)$ denotes the free Lie $K$- algebra generated by $X$, and $\mathfrak{U}= \mathfrak{U}(\mathfrak{L})$ is the
universal enveloping algebra of $\mathfrak{L}$. It is known that $\mathfrak{U}$ and $K\asX$ are isomorphic as associative algebras, and we shall identify them.
$X^m$ denotes the set of all words of length $m \geq 1$ in $X^{*}$.
The set of all Lyndon words in $X^{+}$ is denoted by $L$, and $L_m = L\bigcap X^m$ is the set of all Lyndon words of length $m$.
The free associative algebra $\mathfrak{U}=K \asX$ is naturally graded by length.
An element $f \in \mathfrak{U}$ is called \emph{a Lie element} (or \emph{a Lie polynomial}) if $f \in \mathfrak{L}$.
It is known that if $f$ is a Lie element, then each homogeneous component of $f$ is also a Lie element.
A special case of \emph{homogeneous Lie elements} are the so called $(m-1)$-\emph{left nested Lie monomials}:
\begin{equation}
\label{Lmeq01}
f_m= [[[x_{i_1}, x_{i_2}], \cdots ], x_{i_m}], \; \; x_{i_s} \in X, \; 1 \leq s \leq m.
\end{equation}
 Clearly, the Lie monomial $f_m$ may also be considered as an element of the free associative algebra $\mathfrak{U}=K\asX$, via the
 well-known equality $[a,b]= ab-ba$, so $f_m\in \mathfrak{U}_m $.
The free Lie algebra $\mathfrak{L}= \Lie (X)$ is also naturally graded by length,
\begin{equation}
\label{Lmeq0}
\begin{array}{ll}
\mathfrak{L}&=\bigoplus_{m\geq 0} \mathfrak{L}_m, \;\;\text{where}\;\;  \mathfrak{L}_0 = K,\;  \; \mathfrak{L}_1 = KX,\\
 \mathfrak{L}_m&= \Span _K\{f_m=[[\cdots [x_{i_1}, x_{i_2}], \cdots ], x_{i_m}]\mid
 x_{i_s} \in X, 1 \leq s \leq m\}, \;   m \geq 2.
 \end{array}
\end{equation}

Throughout the paper  \emph{the right standard factorisation} $w = (u,v)$ will be called simply "\emph{the standard factorisation of} $w$".  Recall that $v$ is the longest
proper right segment of $w$ which is a Lyndon word, in this case $u$ is also a Lyndon word.
 \begin{definition} \cite{Lo}
\label{rem_rightbracket}
 Let $l \in L.$ Consider the unique \emph{(right) standard factorization} denoted $l = (u,v)_{r}$.
 The (right) standard bracketing $[l]= [l]_{r}$ on the set $L$ of all Lyndon words is defined  inductively as follows.
 \[[x]= x, \forall x \in X,\quad  [l]_{r}:= [[u]_{r}, [v]_{r}],\;\text{where}\; l \in L -X, \; l = (u,v)_{r}.\]
 The left standard bracketing $[w]_{l}$ is defined analogously, but using the left standard factorization.
\end{definition}

\begin{defconvention}
\label{defLiemonomial}
The notation $[a]$ stands for the right standard bracketing, and $[a]_l$ will be used for the left standard bracketing of $a$.
The Lie element $[u]$ is called also \emph{a Lyndon-Lie monomial} (\emph{corresponding to} $u$).
Given a nonempty set $V$ of Lyndon words, $V \subseteq L$, we denote by $[V]$ the set of standard bracketings of all elements of $V$:
\begin{equation}
\label{[V]eq}
[V] = \{[v]\mid v \in
V\},
\end{equation}
and refer to it as  \emph{"the bracketing of $V$"}.
\end{defconvention}

\begin{remark}
\label{VIPFacts_remark}
\begin{enumerate}
\item
For each Lyndon word $u$, the standard bracketing
$[u]$ considered as an element of the algebra $\mathfrak{U}= K\asX$ is a (noncommutative) polynomial
which is a linear combination
\begin{equation}
\label{Lmeq01}
[u]= u+\sum c_a a,
\end{equation}
 where
each nonzero coefficient $c_a$ is an integer,
all (distinct) monomials $a\in X^{+}$ in this presentation satisfy $a \prec u, |a|= |u|$, in particular $\overline{[u]}= u$, see \cite{Lo}, Lemma 5.3.2, or \cite{Re}, Theorem 5.1.
\item
It is not difficult to see that considered as "commutative terms" all monomials occurring in the right-hand side of (\ref{Lmeq01}) are equal, that is the monomials $a$ and $u$
have \emph{the same multi-degree} as
$u$, (see Definition-Notation \ref{multidegreedef}).
\item
 The elements
\[[L]= \{[u]\mid u\in L\}\]
form a basis of  $\Lie (X)$, see \cite{Lo}, Theorem 5.3.1, or \cite{Re}, Theorem 4.9.

In particular, for each $m \geq 1$ the set $[L_{m}]$ of \emph{Lyndon Lie monomials} of degree $m$,
\[[L_{m}]= \{[u]\mid u \in L_{m}\}\]  is a $K$-basis of the graded component
$\mathfrak{L}_{m}$ (considered as a vector space).
\item
 On the other side, for each $m \geq 2$ the graded component $\mathfrak{L}_{m}$ is spanned by all left-nested Lie monomials
 $[[[x_{i_1}, x_{i_2}], \cdots ], x_{i_m}]$, $x_{i_j}\in X$, hence there are equalities of sets
 \begin{equation}
\label{Lmeq1}
\Span_K \{[u]\mid u \in L_{m}\} = \mathfrak{L}_{m} =
\Span_K \{f_m=[[\cdots[[x_{i_1}, x_{i_2}], \cdots ], x_{i_{m-1}}],x_{i_{m}}]\}.
\end{equation}
\end{enumerate}
\end{remark}

The works \cite{LR}and \cite{GIF} are central for this subsection.
These two works use different terminology for equivalent notions and, for convenience, we adapt the terminology from \cite{LR}
to our equivalent terminology given in \cite{GIF}, and typical for the theory of non-commutative Gr\"{o}bner bases. The equivalence is
given in Remark \ref{FactsLaRam}.
The notion of  "\emph{an atom}" was introduced by Anick, \cite{Anick85} in a more general context, the special properties of
 "the atoms" in the sense of Anick are listed in \cite{Anick85}, Theorem 5. It is easy to see that if $(N,W)$ is a Lyndon pair,
 then the set of Lyndon words $N$ has the properties analogous to the properties of a set of atoms, in the sense of Anick.
 In fact $N$ coincides with the set of atoms for the monomial algebra $A_W$,
 see also our comments in \cite{GIF}.
Recall that, by convention, all Gr\"{o}bner bases of  (associative) ideals $I$ in $K
\asX $ and all Lyndon-Shyrshov bases of  Lie ideals $J$ in
$\Lie(X)$ are considered with respect to "$\prec$" -the degree-
lexicographic well-ordering on $X^{\ast}$.

\begin{notconvention}
\label{notconventionB}
From now on $J$ will denote a Lie ideal in $\mathfrak{L}=\Lie (X)$ \emph{generated by homogeneous Lie elements},(see the natural grading   of  $\Lie (X)$, (\ref{Lmeq0}).  $I$ will denote the
two-sided ideal in $K \asX$ generated by $J$, where $J$ is considered as a set of
associative polynomials, $I =(J_{ass})$. Let
\begin{equation}
 \label{liealgeq1}
\mathfrak{g} = \Lie(X)/J, \quad U= U\mathfrak{g}= K \asX/ I .
\end{equation}
It is known that $U$ is the enveloping algebra of the Lie algebra
$\mathfrak{g}$, see for example \cite{LR}, p. 1823. The algebras $\mathfrak{g}$ and $U$ are naturally graded.
$W= W(I)$
denotes the corresponding set of obstructions for the associative algebra $U= U\mathfrak{g}$.
As usual, $\mathfrak{N} = \mathfrak{N}(I)= \mathfrak{N}(W)$ denotes the set of normal words modulo $I$
(w.r.t. $\prec$).
$N = N(I) = N(W)$, is the set of all Lyndon words which are normal mod $I$, we may also write $N=N(J)$.
By definition $N = \mathfrak{N} \bigcap L$, where $L$ is the set of Lyndon words.
We shall prove that the obstruction set $W$ is \emph{an antichain of Lyndon words}, see Proposition \ref{UgProp} (1), so, by convention, we call $N = N(I)= N(W)$
\emph{the set of Lyndon atoms}  \emph{for} $U$. In this case $(N, W)$ is \emph{a Lyndon pair}.
\end{notconvention}

The remark given below provides an useful interpretation of the terminology of  \cite{LR}  in terms of our equivalent notions introduced in this paper and in \cite{GIF}.
Note that our
terminology is introduced  in more general settings not necessarily in the context in Lie algebras, and independently.
We use results from \cite{GIF} and \cite{LR}, see Theorem (2.1), Corollaries (2.5) and  (2.8)
to deduce straightforwardly the following.

\begin{remark}
\label{FactsLaRam}
In notation and assumption as above, $\mathfrak{g} = \Lie (X)/J$, $U= U\mathfrak{g}= K \asX/ I$, and $W$ is the set of obstructions for $U$.
(1) A word $a\in X^{*}$ is \emph{standard} modulo $I$ in the sense of \cite{LR}   \emph{iff} $a$ is \emph{normal} modulo $I$,
that is no $f \in I$ satisfies $\overline{f}\sqsubseteq a$ (equivalently, no $w\sqsubseteq a$, for some $w \in W$).
Moreover, the set $\mathfrak{N}(I)$ of normal (mod $I$) monomials  coincides with the set of "\emph{standard mod} $I$ \emph{monomials}" in the sense of \cite{LR}.
(2) The set
$N$ of Lyndon atoms, where
$N =\mathfrak{N}(I) \bigcap L$ coincides with the set of "\emph{Lie-standard Lyndon words}", in the sense of \cite{LR}.
Moreover, the set $[N]\subset \mathfrak{g}$ is a $K$-basis of $\mathfrak{g}$.
(3) The enveloping algebra $U = U\mathfrak{g}$ has a $K$-basis of "Poincar\'{e}-Birkhoff-Witt" type expressed in
terms of its Lyndon atoms via (\ref{normalbasiseq_inf}).
This basis coincides with the set of normal words $\mathfrak{N}$ (modulo $I$).
\end{remark}

\subsection{Basic information on classes $\mathfrak{C}(X, W)$ containing the enveloping algebra $U$ of a graded Lie algebra}
\begin{proposition}
\label{UgProp}
Suppose $\mathfrak{g}$ is a graded Lie algebra, $\mathfrak{g} = \Lie(X)/J$, $U= U\mathfrak{g}= K\asX/ I$ is its enveloping algebra.
Let $W$ be the set of obstructions for $U$,
as usual, $\mathfrak{N} = \mathfrak{N}(I)$ denotes the set of normal monomial, and $N=N(W)= \mathfrak{N}\bigcap L$ is the set of normal Lyndon words.
The following conditions hold.
\begin{enumerate}
 \item
 \label{UgProp1}  $W$ is an antichain of Lyndon words, so $(N,W)$ is a Lyndon pair, and
 $A_W = K\asX/ (W)$
 is \emph{a monomial algebra defined by Lyndon words}, in the sense of \cite{GIF}.
 \item
 \label{UgProp3}
 Assume furthermore, that $J$ is generated by homogeneous Lie elements,
 so $\mathfrak{g}$  and the enveloping algebra $U$ are canonically graded.
 Then the algebras $U$, and $A_W$ are in the class $\mathfrak{C}(X, W)$.
The following conditions are equivalent.
 \begin{enumerate}
 \item
 \label{UgProp31}
 $U$ is an Artin-Schelter regular algebra.
 \item
 \label{UgProp32}
 $U$ has polynomial growth.
 \item
 \label{UgProp33}
 The Lie algebra $\mathfrak{g}$ is finite dimensional.
 \item
 \label{UgProp34}
 The set of Lyndon atoms $N$ is finite.
 \end{enumerate}
 Each of these equivalent conditions implies that
 \[\gldim (U) = \gkdim (U)  = \dim _K \mathfrak{g} = |N|=d.\]
 Moreover,
 $U$ is a standard finitely presented algebra and
 \begin{equation}
\label{orderW}
 d-1 \leq |W| \leq d(d-1)/2,\; \;\text{where}\; \; d = |N|.
 \end{equation}
 \end{enumerate}
\end{proposition}
\begin{proof}
(\ref{UgProp1}).
The algebra $U$ has a $K$-basis of "Poincar\'{e}-Birkhoff-Witt" type  expressed in
terms of its Lyndon atoms via (\ref{normalbasiseq_inf}), see Remark \ref{FactsLaRam}. Hence by Lemma \ref{VIPlemma1} the set of obstruction $W$ is an antichain of Lyndon words. Clearly, then $(N,W)$ is a Lyndon pair.
Moreover, $A_W = K\asX/(W)$ is a monomial algebra defined by Lyndon words and all results of \cite{GIF} are applicable.

(\ref{UgProp3}).
The algebras $A_W$ and  $U$  belong to $\mathfrak{C} (X, W)$,
they have
 the same $K$-basis, $\mathfrak{N} =\mathfrak{N}(W)$,  and the same set of Lyndon atoms $N= N(W)= \mathfrak{N}\bigcap L$, in particular,
 $\gkdim U\mathfrak{g} =\gkdim A_W$.
The equivalence below follows from \cite{GIF}, Theorem A (2).
\begin{equation}
\label{eq_implications11}
\gkdim (A_W)= d \Longleftrightarrow |N|= d.
\end{equation}
Moreover, \cite{GIF}, Theorem B (2) gives the following relations
 \begin{equation}
\label{eq_implications12}
\gkdim (A_W)= d \Longrightarrow \gldim (A_W) = d,  \; \quad\; d-1 \leq |W|\leq d(d-1)/2.
\end{equation}
Now
\cite{LR}, Theorem (2.1) implies straightforwardly
that the set $[N]= \{[u]\mid u \in N\}$ is a $K$-basis of the Lie algebra $\mathfrak{g}$, there is an isomorphism of vector spaces
$\mathfrak{g}  \cong \Span_{K} [N]$. Thus
$\mathfrak{g}$ is finite dimensional, if and only if the set of atoms $N$ is finite, in this case
$\dim \mathfrak{g} = |N|.$
Assume now that $U$ has polynomial growth of degree $d$, so \[\gkdim (U)= d = \gkdim (A_W).\]
It follows from (\ref{eq_implications11}) that $|N|= d$, hence (\ref{UgProp32})$\Longleftrightarrow$ (\ref{UgProp34}),
and (\ref{UgProp32}) $\Longleftrightarrow$  (\ref{UgProp33}).
By (\ref{eq_implications12})  $\gldim (A_W)= d$, and by \cite{GI89}, Theorem II,
 $\gldim(U) = \gldim (A_W)= d.$
The obstruction set $W$ is finite (by (\ref{eq_implications12}) again) therefore $U$ is standard finitely presented.

Suppose now $U$ is Artin-Schelter regular, then by definition it has polynomial growth of degree, say $d$,
so  (\ref{UgProp31})$\Longrightarrow$ (\ref{UgProp32}), and therefore the remaining conditions (c) and (d) are also in force.
The discussion above implies $\gldim U= d = \gkdim U = \dim \mathfrak{g} = |N|$.

Conversely, assume that $\mathfrak{g}$ is finite dimensional.
It is known by the experts that the enveloping algebra $U=U\mathfrak{g}$ of a finite dimensional positively graded
Lie algebra $\mathfrak{g}$ is Artin-Schelter regular. A reference to the original result is difficult to find, and
we refer to the proof provided by Fl{\o}ystad and  Vatne, see \cite{FlVa}, Theorem 2.1.
\end{proof}

We have seen that the enveloping algebra $U\mathfrak{g}$ of any Lie algebra $\mathfrak{g}$
generated by $X$ has a set of obstructions $W$, consisting of Lyndon words, in general, $W$ may be infinite.
Moreover, if $\mathfrak{g}$ has finite dimension $\dim \mathfrak{g}= d$,
then the enveloping algebra
$U=U\mathfrak{g}$ is a standard finitely presented Artin-Schelter regular algebra of global dimension $d$.
The number $r= |W|$ of its relations satisfies the inequalities (\ref{orderW}).
The following question arises naturally.

\begin{question}
Let $(N, W)$ be a Lyndon pair, where $N$ is a finite set.
Under  what combinatorial conditions
is  $W$ the set of obstructions for the enveloping $U\mathfrak{g}$ of some
Lie algebra $\mathfrak{g}$?
\end{question}
Proposition \ref{VIPproposition} gives a necessary condition in terms of $N=N(W)$, the set of Lyndon atoms.
In section \ref{Sec.MonLiealgebras}
we introduce and study
\emph{the monomial Lie algebras} $\mathfrak{g} = Lie(X)/([W])$. In this case the enveloping algebra $U\mathfrak{g}$ belongs to $\mathfrak{C}(X, W)$ if and only if $[W]$ is a Gr\"{o}bner-Shirshov basis of the Lie
ideal $J =([W])_{Lie}$ of $\Lie(X)$.

\subsection{If the class $\mathfrak{C}(X, W)$ contains the enveloping algebra $U$ of a graded Lie algebra then the set $N= N(W)$ is connected}
\begin{remark}
\label{remarkEqualityof Classes}
 Due to the duality $W = W(N)$, $N = N(W(N))$, see Proposition \ref{N-WProp0},
  each class of associative graded algebras $\mathfrak{C}(X, W)$, with obstruction set $W$ consisting of Lyndon words is also uniquely determined
  by the set of Lyndon atoms $N=N(W)$.
  Conversely, each set $N$ of Lyndon words closed under taking Lyndon subwords, and with $X \subseteq N$,
  determines uniquely a class $\mathfrak{C}(X, W)$, where $W=W(N)$, and $(N,W)$ is a Lyndon pair.
  It is interesting to classify  all Lyndon pairs $(N, W)$ such that
  the monomial algebra $A_W$ has polynomial growth and
 $\gldim A_W = d$,  where $d\geq 1$ is fixed.
  We use our result $\gldim A_W= \gkdim A_W=|N|= d$ (in this case $W$ is always finite).
  Our method consists of two steps:
  \begin{enumerate}
  \item Classify the sets of Lyndon atoms $N$ of a fixed finite order  $|N|= d$, $N$ must satisfy \textbf{C1, C2}.
  \item  Write down the corresponding set of obstructions $W= W(N)$, so that $(N, W)$ is a Lyndon pair.
   \end{enumerate}
   \emph{In this case (whenever we start with $N$) we shall often use notation $\mathfrak{C}(X, N)$ for the class of associative graded algebras with set of obstructions $W = W(N)$}.
   A Lyndon pair $(N, W)$ is uniquely determined by each of its component $N$, or $W$, and by convention we shall use both notation
 \begin{equation}
\label{C(X,N)eq}
\mathfrak{C}(X, N): = \mathfrak{C}(X, W),\; \text{where $(N,W)$ is a Lyndon pair}.
\end{equation}
  \end{remark}
  Till the end of the paper we shall consider only the special case when \emph{the ideal $I =(J_{ass})$ of $K \asX$ is graded by length, respectively, multi-graded,}
thus  the enveloping algebra $U= K \asX/ I$  is also graded, respectively, multi-graded. Moreover, $U$ has a set of obstructions $W$ consisting of Lyndon
words, and $U \in \mathfrak{C}(X, W)$.

Note that every Lyndon word $u$ of  length $|u|\geq 3$ in a finite alphabet $X$, contains a Lyndon subword of length $2$.
Moreover, if $N$ is a finite set of Lyndon monomials
in $X = \{x,y\}$ with  \textbf{C1, C2}, and $N$ contains a word of length $\geq 4,$ then $N\bigcap L_3 \neq \emptyset$.
\begin{definition}
\label{connect_def}
Suppose $N$ is a set of Lyndon words (possibly infinite) with conditions \textbf{C1} and \textbf{C2}.
(i.e. $X \subseteq N$, and $N$ is closed under
taking Lyndon subwords).
\begin{enumerate}
\item
We say that $N$ is \emph{connected} if
\begin{equation}
\label{connectedeq}
[N\bigcap L_m \neq \emptyset,\; m \geq 2] \quad\text{implies} \quad[N\bigcap L_s \neq \emptyset, \;\forall s, 1\leq s \leq m].
\end{equation}
In particular, if $N=  \{l_1<
l_2 < l_3\cdots < l_d \}$ is a finite set,
then $N$ is \emph{connected} \emph{iff} (\ref{connectedeq}) is in force
for  $m = \max \{|l_i| \mid 1 \leq i \leq d \}$, and all $s, 1 \leq s \leq m$.
\item
Suppose $N$ is not connected, that is $N \bigcap L_j = \emptyset, N \bigcap L_{j+1} \neq \emptyset$ for some $j$.
Then \emph{the connected component of $N$}, denoted by  $N_{con}$ is defined as
\begin{equation}
\label{Nconeq}
\begin{array}{l}
N_{con}: = \bigcup_{1 \leq s\leq k} (N \bigcap L_s),\;  \text{where}\; N\bigcap L_s  \neq \emptyset, \forall s, 1 \leq s\leq k,\\
  \text{and}\; N\bigcap L_{k+1} = \emptyset.
\end{array}
\end{equation}
For completeness we set $N_{con}= N$, whenever $N$ is connected.
Note that $N_{con}$  is the maximal \emph{connected subset} of $N$, which satisfies \textbf{C1} and \textbf{C2}.
\end{enumerate}
\end{definition}
One can find explicit examples in the last section, where we have
classified the Lyndon pairs $(N,W)$ in the alphabet $X = \{x < y\}$, with $|N|= 7$, see Subsec. \ref{list7}.
Each of the first 12 pairs  $(N,W)$, (7.4.1) through (7.6.12), has connected set of atoms $N$. Each of the remaining pairs has a disconnected set of Lyndon atoms,
but we have identified explicitly the connected components $N_{con}$.
We illustrate this with the Lyndon pair $(N,W)$ given in  (7.7.15). The set of Lyndon atoms,
 $N=\{x < xy < xy^2 <xy^2xy^3< xy^3 <xy^4 < y\}$ is disconnected.  One has $N\bigcap L_7 = \{u =xy^2xy^3\}$, and $N\bigcap L_6 = \emptyset$.
 In this case  $N$ has a connected component $N_{con}= N\setminus \{xy^2xy^3\} = N(\Lcal_5)$, the set of atoms corresponding to the Filiform Lie algebra $\Lcal_5$ of dimension $6$
 and nilpotency class of degree $5$.

Suppose $N$ is finite. Clearly, the property \emph{"$N$ is connected"} is recognizable.
Moreover, this is a necessary condition so that "$\mathfrak{C}(X, W) = \mathfrak{C}(X, N)$
contains the enveloping algebra $U=U\mathfrak{g}$ of some Lie algebra $\mathfrak{g}$", as shows Proposition \ref{VIPproposition}.

\begin{definition}
\label{nilpdef}
Recall that
$\mathfrak{g}$ is \emph{a nilpotent Lie algebra of nilpotency class}
$m$ if $m$ is the minimal value for which $\mathfrak{g}^{m+1} =0$, where
\[\mathfrak{g}^1 =\mathfrak{g},  \;\text{and}  \; \mathfrak{g}^t = [\mathfrak{g}^{t-1}, \mathfrak{g}],\; t\geq 2\]
is the lower central series of $\mathfrak{g}$.
In this case we shall also write $\mathfrak{g} \in \mathfrak{N}_m$.
\end{definition}

\begin{proposition}
\label{VIPproposition}
 Suppose  $(N,W)$ is a Lyndon pair, such that the class
 $\mathfrak{C}(X, W)$
contains the enveloping algebra $U=U\mathfrak{g}$ of a graded Lie algebra $\mathfrak{g}$. Then the following conditions hold.
\begin{enumerate}

 \item
\label{VIPproposition1}
If $N\bigcap L_s = \emptyset$, for some $s \geq 2$ then $N\bigcap L_{s+k} = \emptyset$, for all $k \geq 1$. In particular, $N$ is finite.
Moreover, if $s$ is the minimal positive integer with this property, then $\mathfrak{g}$ is a nilpotent Lie algebra of nilpotency class $s-1$.
\item
\label{VIPproposition2}
If $N$ is an infinite set, then $N\bigcap L_m \neq \emptyset$ for all $m \geq 1$, so $N$ is connected.
\item
\label{VIPproposition4}
Suppose $N$ is a finite set, $|N|=d$ and let $m = \max \{|u|\mid u \in N\}$. Then the Lie algebra $\mathfrak{g}$
 is nilpotent of nilpotency class $m \leq d -|X|+1.$
\item
\label{VIPproposition3}
 The set $N$ of Lyndon atoms is connected, and $[N]$ is a
 $K$-basis for $\mathfrak{g}$ .
 \end{enumerate}
 \end{proposition}
\begin{proof}
 It follows from the hypothesis
that there exists a Lie algebra $\mathfrak{g} = Lie(X) /J$, where $J$ is a (graded) Lie ideal whose enveloping $U=U\mathfrak{g} = K\asX /I$, has set of obstructions $W$ ($I$ is
the associative two-sided ideal  generated by $J$). In fact $U$ determines uniquely the Lyndon pair $(N,W)$, given in the hypothesis.

(\ref{VIPproposition1}) Assume $N\bigcap L_s = \emptyset$, for some integer $s \geq 2$. It follows from the equalities (\ref{Lmeq1}) that there is an inclusion of subspaces
\[
\mathfrak{L}_{s} = \Span_K \{[[[[x_{i_1}, x_{i_2}], x_{i_3}] \cdots ], x_{i_s}]\} \subset J.
\]
This implies that for $q \geq s$ every nested Lie monomial $w_q= [[[[x_{i_1}, x_{i_2}], \cdots x_{i_s}], \cdots ], x_{i_q}] \in J,$ i.e. $\mathfrak{g}$ is a nilpotent Lie
algebra of nilpotency class $\leq s-1$.  Suppose that $s$ is the minimal integer with $N_s = \emptyset$. Then
$\mathfrak{g}^{s-1} \neq 0$, and $\mathfrak{g}^{s} = [[\mathfrak{g}^{s-1},\mathfrak{g} ]= 0$,
hence $\mathfrak{g}$ is a nilpotent Lie algebra of nilpotency class $s-1$, see Definition \ref{nilpdef}.

Part (\ref{VIPproposition1}) implies straightforwardly part (\ref{VIPproposition2}).

(\ref{VIPproposition3}).
Suppose that $N$ is a finite set. It follows from Definition
\ref{connect_def}
and part (\ref{VIPproposition1}) that $N$ is connected. By part (\ref{VIPproposition2}) $N$ is connected, whenever it is an infinite set.

(\ref{VIPproposition4}) Assume that $N$ has finite order $|N|=d$ and $m = \max \{|u|\mid u \in N\}$. Without loss of generality we may assume $m \geq 2.$
The set $N$ is connected, hence $N_s =N \bigcap L_s \neq \emptyset, 1 \leq s\leq m$. The equality $N = \bigcup_{1 \leq s \leq m} N_s$ implies
\[d = |N|= |X| + \sum_{2 \leq s \leq m} |N_s| \geq |X| + m-1, \] therefore $m \leq d +1 - |X|.$
It follows from part (\ref{VIPproposition1}) that $\mathfrak{g}$ is a nilpotent Lie algebra of nilpotency class $m \leq d -|X|+1$.
\end{proof}

\begin{corollary}
\label{NdisconnCor}
Let $(N,W)$ be a Lyndon pair, where $N$ is finite. Suppose $N$ is not a connected set of Lyndon words. Then (i) the class $\mathfrak{C}(X, W)$
does not contain the enveloping algebra $U=U\mathfrak{g}$ of a graded Lie algebra $\mathfrak{g}$. (ii) In particular,   $[W]$
is not a Lyndon-Shirshov basis of the Lie ideal $J=([W])_{Lie}$ of $\Lie(X)$.
\end{corollary}

\section{Monomial Lie algebras defined by Lyndon words}
\label{Sec.MonLiealgebras}
As usual, $X = \{x_1, \cdots , x_n\}$ is a fnite set and $W$ denotes an antichain of Lyndon words in $X^{+}$.
Our goal is to find explicitly (in terms of generators and relations) classes $\mathfrak{C}(X, W)$ containing enveloping algebras $U\mathfrak{g}$ of prescribed
global dimension. As a first step in this direction we introduce and study the \emph{monomial Lie algebras defined by Lyndon words}.
\subsection{Basic definitions and first properties of Monomial Lie algebras}
\begin{definition}
\label{monLiedef}
We say that  $\mathfrak{g} = \Lie(X)/J$  is \emph{a monomial Lie algebra} \emph{defined by Lyndon words}, or shortly, \emph{a monomial Lie algebra},
if $J$ is a Lie ideal
generated by the Lie elements $[W]= \{[w]  \mid w \in W \}$, where $W$ is an antichain of Lyndon words.
If moreover, the set of Lie monomials $[W]$ is a Gr\"{o}bner-Shirshov  basis of the Lie ideal $J = ([W])_{Lie})$ (see Definition \ref{GS-basis_def})
we shall refer to $\mathfrak{g}$ as \emph{a standard monomial Lie algebra} and denote it by
 $\mathfrak{g}_{W}$.
 \end{definition}
Definition \ref{GS-basis_def} is one of the several equivalent definitions of a Gr\"{o}bner-Shirshov basis of a Lie ideal in $\Lie(X)$ (or shortly a GS basis).
Proposition \ref{theor_easy_vip} gives several equivalent conditions each of which can be used for a definition. For now we shall use the fact that
 $[W]$ is a Gr\"{o}bner-Shirshov  basis of the Lie ideal $J$ \emph{iff} the set  $[W]_{ass}$
considered as "associative" elements of $K \asX$
is a \emph{Gr\"{o}bner  basis of the associative ideal} $I = (J_{ass})$, w.r.t the ordering $\prec$ on $X^{\ast}$, see Proposition \ref{theor_easy_vip}.
Our first examples are classical.
\begin{example}
\begin{enumerate}
\item
The free-nilpotent Lie algebra $\mathfrak{L}(m+1)$,  $m \geq 2$,  generated by $X$, is a standard monomial Lie algebra
defined by Lyndon words, see Corollary \ref{freeNilp_cor}.
\item The standard Filiform Lie algebra $\mathcal{L}_m$, of dimension $m+1$ and  nilpotency class $m$, see Subsec \ref{MonFiliformSec}.
\end{enumerate}
\end{example}

Each monomial Lie algebra defined by Lyndon words, $\mathfrak{g} = Lie(X)/([W])_{Lie}$ is graded by length, and is also positively $\mathbb{Z}^n$-graded. Moreover,
the two sided ideal $I =(J_{ass})$ of $K\asX$ is generated by homogeneous elements
and is also $\mathbb{Z}^n$-graded, so the enveloping algebra $U =U\mathfrak{g}= K\asX/ I$ is (non-negatively) $\mathbb{Z}^n$-graded.

\begin{defnotation}
\label{multidegreedef}
A monomial $w \in X^{\ast}$, has (\emph{a non-negative})\emph{multi-degree} $\alpha =(\alpha_1, \cdots, \alpha_n) \in \mathbb{Z}^n$, $\alpha_i \geq 0$, if $w$, considered as a
commutative term,
can be written as $w = x_1^{\alpha _1}x_2^{\alpha _2}\cdots x_n^{\alpha _n}$. In this case we write $\deg (w) =\alpha$.
In particular, the unit $1\in X^{\ast}$ has multi-degree $\textbf{0} =(0, \cdots, 0)$, and $\deg(x_1) = (1,0,\cdots, 0), \; \deg(x_2)
= (0, 1, \cdots, 0), \cdots, \deg(x_n)= (0,0,\cdots, 1).$
 For each $\alpha= (\alpha _1, \alpha _2,\cdots, \alpha _n) \in \mathbb{Z}^n, \alpha _i \geq 0$,
 denote by $X_{\alpha}$ the set
 $X_{\alpha} = \{w \in \asX \mid \deg(w)= \alpha \}$.
  Then the free monoid $X^{\ast}$ is naturally $\mathbb{Z}^n$-graded:
\[X^{\ast} = \bigsqcup_{\alpha\in \mathbb{Z}^n, \alpha \geq \textbf{0}}  X_{\alpha}, \quad
X_{\textbf{0}} = \{1\}, \; X_{\alpha}.X_{\beta}\subseteq X_{\alpha+\beta}.\]
The set of Lyndon words $L$ inherits this grading canonically, it splits as a disjoint union
\[L= \bigsqcup_{\alpha \in \mathbb{Z}^{n}, \alpha >\textbf{0}} L_{\alpha}, \;\text{where}\; L_{\alpha} = L \bigcap X_{\alpha}.\]
The free associative algebra
generated by $X$ is also canonically $\mathbb{Z}^{n}$ -graded:
\[U= K \asX = \bigoplus_{\alpha \in \mathbb{Z}^{n}} U_{\alpha}, \;\text{where}\; U_{\alpha}= \Span  X_{\alpha}, \;   \alpha \in \mathbb{Z}^{n}.\]
\end{defnotation}

For example, if $X = \{x < y\}$, then
\[L_{(3,3)} = \{xxxyyy < xxyxyy< xxyyxy\}= \{xxxyyy \succ xxyxyy\succ xxyyxy\}  \]

Recall that a graded Lie algebra is a Lie algebra endowed with a gradation compatible with the Lie bracket.
The free Lie algebra is naturally  $\mathbb{Z}^n$-graded:
\[\mathfrak{L}=\Lie (X) = \bigoplus_{\alpha \in \mathbb{Z}^{n}} \mathfrak{L}_{\alpha},\;\;   \mathfrak{L}_{\alpha} = \Span \{[u]\mid u \in L_{\alpha}\}, \alpha
=(\alpha_1, \cdots, \alpha_n) \in \mathbb{Z}^n, \;  \alpha_i \geq 0 \} \]
If $w \in L$ has multi-degree $\deg (w) = \alpha$, then the standard (right) bracketing $[w]$, considered as an associative polynomial
in $\mathfrak{U}$ is a linear combination of words $u \in X_{\alpha}, u\preceq w,$ so $[w]\in \mathfrak{L}_{\alpha},$
and its highest monomial w.r.t. $\prec$ is precisely
$\overline{[w]}= w$.
Sometimes we shall also write
\[L (w): = L_{\alpha}, \;\; L (\preceq w) := \{u \in L (w), u \preceq  w\},\;\; L (\prec w) := \{u \in L (w), u \prec w\}.\]
For example, $w = xxxyyy \in L_{(3,3)}, \; L (\prec w) = \{ xxyxyy\succ xxyyxy\}$.
Let $W$ be an antichain of Lyndon monomials in $X^{+}$.
Then the  Lie ideal $J=([W])_{Lie}$ is multi-graded, $J = \bigoplus_{\alpha} J_{\alpha},$ where $J_{\alpha} = J \bigcap \mathfrak{L}_{\alpha},  \alpha
\in\mathbb{Z}^n, \alpha \geq \textbf{0}$.

\begin{remark}
Recall that if $\;\Ical=(W)$ is a monomial ideal in $K\asX$, where $W$ is an arbitrary antichain of monomials ($W\bigcap X = \emptyset$), then
$\Ical$ is $X^{+}$-graded.
 If $f=\sum c_a a   \in \Ical,$ with $a \in X^{+}, c_a \in K^{\times}$ then every monomial $a$ in this presentation is also contained in $\Ical$, or equivalently, $w
 \sqsubseteq a$ for some $w \in W$.
In contrast, a monomial Lie ideal $J = ([W])_{Lie}$ generated by Lyndon Lie monomials may contain a Lie element $g \in J$, where
$g = \sum c_a [a], \; a \in L, \;c_a \in K^{\times}$, such that $[a]$ is not in $J$ for some $a$,  see Example \ref{mon-example1}.
Moreover, the reduced Gr\"{o}bner-Shirshov Lie basis of a monomial Lie ideal $J = ([W])$ in $\Lie(X)$ may contain a Lie element $f \in J$,
which is not a Lie monomial,
 as shows Example \ref{mon-nonmonExample}.
\end{remark}

\begin{example}
\label{mon-example1}
Let $W =\{xxy\}$, $J = ([xxy])_{Lie}$,
then  the Lie monomial $[xxy]$ is a Gr\"{o}bner-Shirshov basis of the Lie ideal $J$ (since there are no compositions to resolve).
The associative polynomial $h= [xxy]_{ass}= xxy-2yxy +yxx$ is a Gr\"{o}bner basis of the associative ideal $I = (J)_{ass}$, $\overline{h}= xxy.$
Consider the Lie element $g=[[[xxy],y],y]= [xxyyy]_l\in J$. Note that $g= [xxyyy]_l$, is the left standard bracketing of $xxyyy$.
 One has:
\[ \begin{array}{ll}
f =&[[xxy],y] = [x, [xyy]] \in J,      \\
g =&[f, y]= [[[xxy],y],y]=[[x, [xyy]],y]= [[xy], [xyy]] + [x, [xyyy]]\\
   =& [xyxyy]+[xxyyy] \in J.
             \end{array}
\]
The Lie monomial  $[xyxyy]$ is not in $J$ since $xyxyy$  is normal modulo $W$.
One has
\[ \begin{array}{ll}
xxyyy = &\overline{[xxyyy]} \succ \overline{[xyxyy]}= xyxyy \\
         &\text{hence}\; \overline{g} =\overline{[[[xxy],y],y]} = xxyyy= \overline{[[x, [xyyy]]} , \; \text{and}\; w= xxy \sqsubset \overline{g}.
             \end{array}
\]
\end{example}

\subsection{Gr\"{o}bner-Shirshov bases of monomial Lie ideals, preliminaries}
\label{GSbasesSection}
In this subsection we present in terms of Lyndon words
some notions and results from Shirshov's theory based on Shirshov's \emph{regular words}, or as they call them recently "Lyndon-Shirshov words".
During the last  decades numerous works appeared involving Lyndon-Shirshov words in the study of  Gr\"{o}bner-Shirshov bases, normal forms, combinatorial and decision problems, and the interested reader can find details in \cite{BLSZ, Bokut14, Bokut20, ChenTang}, et al, and the references therein.

However, although the so called Lyndon-Shirshov words are analogous to Lyndon words (in the classical sense),
they are defined differently and neither Lyndon-Shirshov words, nor corresponding theory and results can be used directly (without certain modification) for our theory  based on the original notion of Lyndon word.

Our goal is to simplify the general procedure for finding Gr\"{o}bner-Shirshov  bases in the special case of
$\mathbb{Z}^n$-graded monomial Lie ideals $J= ([W])_{Lie}$ in $\Lie(X)$ and to apply it for our classification theorem
in Section \ref{class_sec}.

For the purpose we combine some known facts about Gr\"{o}bner-Shirshov  bases with
results proven in  \cite{Anick86}, \cite{LR}, \cite{GIF} and especially with our combinatorial results in this paper.
As a first step we yield Proposition \ref{theor_easy_vip},  which can be used
in some special cases to decide \emph{straightforwardly} (avoiding computations) whether $[W]$ is a Gr\"{o}bner-Shirshov  basis of the
monomial Lie ideal $J= ([W])_{Lie}$.
We next prove Propositions \ref{SimplifyProp}, and \ref{connected_Prop} which (in various cases)
can be and will be used to find straightforwardly the reduced Gr\"{o}bner-Shirshov  basis of
the Lie ideals $J= ([W])_{Lie}$ (and avoid
the need of heavy computations prescribed by the standard procedures).
As an application, in Section \ref{class_sec} we classify all 2-generated AS-regular algebras $U$ of global dimension $6$ and $7$ which
occur as enveloping algebras $U=U\mathfrak{g}$, of standard Lie monomial algebras
$\mathfrak{g} = \Lie(x,y) /([W])$, (in this case $[W]$ is a Gr\"{o}bner-Shirshov Lie basis).

Recall that if $f \in K\asX$ is a nonzero polynomial, then its highest
monomial with respect to the degree-lexicographic order $\prec$ on $X^{\ast}$ (with $x_1 \succ x_2\succ \cdots \succ x_n$) is denoted by $\overline{f}$. One has
$f = \alpha\overline{f} + \sum_{1 \leq i\leq m} \alpha_i u_i$, where
$\alpha, \alpha_i \in K$, $\alpha \neq 0,$ $u_i\in X^{\ast}, u_i
\prec \overline{f}$.

Let $u\in L$ be a Lyndon word of length $\geq 2.$ We know that  the right bracketing $[u]=[u]_r$ and the left bracketing $[u]_l$ are Lie elements which (considered as
associative elements) are homogeneous monic polynomials with equal highest terms \[\overline{[u]}=u =\overline{([u]_l)}.\]
More generally,
consider a new bracketing  ("rebracketing") $[u]_{*}$ of $u$, that is a Lie monomial, obtained by writing Lie brackets in some way. It is called
\emph{a regular bracketing}, or \emph{a regular Lie monomial} of $u$ if it is a monic (associative) polynomial with highest monomial $\overline{([u]_{*})}=u$.
\begin{lemma}
\label{lemma_bracketing}
Let $\omega$ be a Lyndon word, $\omega \in L_{\alpha}$, and  let $\omega=uv,$ where $u,v$ are Lyndon words.
If  $[u]_{*}$, and $[v]_{*}$ are regular bracketings (regular Lie monomials) of $u$ and $v$,  then
$([uv])_{*}:= [[u]_{*}, [v]_{*}]$ is a regular bracketing of $\omega=uv$ respecting $[u]_{*}$ and $[v]_{*}$.
In particular, $[\omega])_{u}= [[u], [v]]$  is a regular bracketing of $\omega$ respecting $[u]$.
\end{lemma}
\begin{proof}
It follows from the hypothesis that $u < uv=\omega < v< vu.$  Recall that for any two elements of the free associative algebra,  $f, g \in K\asX$
one has
\begin{equation}
\label{eqHM1}
\overline{f.g}  = (\overline{f})(\overline{g}).
\end{equation}
 The following relations are considered in the free associative algebra:
\[
\begin{array}{ll}
[[u]_{*},[v]_{*}] &= ([u]_{*}).([v]_{*})-([v]_{*}).([u]_{*})\\&\\
\overline{([u]_{*}).([v]_{*})}&=( \overline{[u]_{*}} ).( \overline{[v]_{*}} ) \\&\\
                              &= u.v < v.u = ( \overline{[v]_{*}} )( \overline{[u]_{*}} ) = \overline{([v]_{*}.[u]_{*})}
                              \end{array}
\]
For monomials of the same length in $X^{+}$ we have $a < b$ \emph{iff} $a\succ b$, so $\overline{[u].([v]_{*})} = u.v \succ v.u = \overline{([v]_{*}).[u]}$, and
therefore
\begin{equation}
\label{eqHM2}
\overline{[[u]_{*},[v]_{*}]}  = uv,
\end{equation}
as desired. Clearly this bracketing of $uv$ respects both $[u]_{*}$ and $[v]_{*}$.
\end{proof}

The following is a modification of the original result  \cite{Shirshov58}, Lemma 4, see also \cite{Bokut14}, p.372.

\bigskip

\emph{Shirshov's special bracketing}.  Suppose $w = aub \in L$, $u \in L$ ($a=1$, or $b = 1$ is possible). Then the right standard bracketing satisfies
\begin{enumerate}
\item $[w]= [a[uc]d]$, where $b=cd$, $uc \in L$, and possibly $c=1$.
\item If $c \neq 1$, express $c$ as a product of Lyndon words $c= c_1c_2 \cdots c_s$, $c_i \in L$, where
$c_1\leq c_2 \leq \cdots \leq c_s$. Replacing $[uc]$ by $[\cdots[[[u],[c_1]],[c_2]],\cdots, [c_s]]$ we obtain the Lie monomial
\[
[w]_u= [a[\cdots[[[u],[c_1]],[c_2]],\cdots, [c_s]]d],
\]
which is called \emph{the Shirshov special bracketing of $w$ relative to $u$}.
\item There is an equality
$\overline{[w]_u}= w$.
\end{enumerate}

To emphasise that we use Lyndon words (and, in general, not Lyndon-Shirshov words) we slightly change the terminology.
\begin{defnotation}
\label{defnotationregularbracketing}
 Let  $w = aub \in L$, $u \in L$, where $a=1$, or $b = 1$ is possible.
We shall refer to $[w]_u$ as \emph{the regular bracketing of $w$ respecting $[u]$} and denote it also by  $(a[u]b):= [w]_u$.
\end{defnotation}

\begin{remark}
\label{remarkShirshovBracketing}
We shall give  \emph{a sketch of proof of the existence of the regular bracketing $(a[u]b)$}  in terms of Lyndon words which could be used effectively. (It is not analogous to Shirshov's proof).

We use induction on the length of $w$.
If $|w|)=2$, then each proper segment $u$ is in $X$ and obviously $[w]=[w]_u$.

Suppose the statement is true for every pair $w, u\in L,$ where $u\sqsubset w$, and  $|w|\leq m$.

Assume $|w|= m+1$, and $u$ is a proper segment of $w$. Three cases are possible (i) $w= au$, (ii) $w= ub$, and (iii) $w = aub,$
where $a, b \in X^{+}$.  Clearly, in case (i) $[w]_u = [w]$.
Assume (ii),  $w= ub$, with $b \in X^{+}$. The left standard factorisation $w = (p, q)$ satisfies $w = pq, p, q \in L$
 and $p= uv$, possibly $v = 1$. If $p = u$, then $[w]_u = [[p], [q]]$. If $p = uv, v \in X^{+}$, by  the induction  hypothesis
   $[p]_u$ is well defined and we set $[w]_u = [[p]_u, [q]]$.
    Case (iii). Suppose now $w = aub,$
where $a, b \in X^{+}$. Consider the right standard factorization of $w = (p,q)$ ($q$ is the longest Lyndon right segment of $w$, and $p\in L$). By Lemma \ref{VIP_lemma11}
either $u \sqsubset p$, or $u \sqsubset q.$ We apply the induction hypothesis and in the first case
we set $[w]_u = [[p]_u, [q]]$, while in the second case $[w]_u = [[p], [q]_u]$.   $\quad \quad \quad \quad   \quad \quad\quad \quad   \quad \quad  \quad \quad  \quad \quad   \quad \quad\quad \quad \Box$
\end{remark}

Suppose $w \in L_{\alpha}$, then its standard bracketing $[w] \in \Span [L_{\alpha}]$, and
the regular bracketing  $[w]_u= (a[u]b)$ is a Lie element satisfying
\begin{equation}
\label{regularbrack_eq}
\begin{array}{ll}
\overline{(a[u]b)}  &=aub = w, \quad \\
(a[u]b) &= [w]+ \Sigma_{v\prec w} c_v [v] \in \Span [L_{\alpha}], \; c_v \in K,  \; \forall v.
\end{array}
\end{equation}
Moreover, the regular bracketing $(a[u]b)$ is an element of the Lie ideal $([u])_{Lie}$ of $\Lie(X)$, generated by $[u]$.

\begin{example}
(1) Let  $u= xxyyxy, v = xyy, \omega = uv =xxyyxyxyy$. Then $[u]= [xxyyxy], [v] =[xyy]$, and, by Lemma
\ref{lemma_bracketing}, the Lie monomial $[[u], [v]]= [[xxyyxy],[xyy]]$ is a regular bracketing of $\omega=uv$ respecting the bracketings $[u]$ and $[v]$.
Here $[\omega]_u = [[u], [v]]=  [[xxyyxy],[xyy]]$ is the Shyrshov special bracketing of $\omega$ relative to $u$.
Note that
the standard right factorisation of $\omega = uv$ is $\omega = (xxyy)(xyxyy)$ which implies $[\omega] = [[xxyy], [xyxyy]]$, so the standard bracketing $[\omega]$ does not "respect" the standard bracketing $[u]$ of the Lyndon subword $u\sqsubset \omega$.

(2) Consider $\omega = x(xyy)xyxyy = aub,$ where $a=x, u = xyy, b=xyxyy$. Then
$[\omega]_u = [[a,[u]], [b]]= [[x,[xyy]], [xyxyyy]]$.
 \end{example}

\begin{definition}
 \label{GS-basis_def}
Let $W$ be an antichain of Lyndon words.  The set of Lie monomials $[W]$ is \emph{a Gr\"{o}bner-Shirshov basis of the Lie ideal $J=([W])_{Lie}$}
if every nonzero Lie element $f \in J$, can be written as
\[
f = c[u]+\sum_{v\prec u} c_v[v]\in \Span [L],
\]
where $u = \overline{f}, c \in K^{\times}$, and
  $u=awb$, for some $w \in W, a, b \in X^{*}$.
Note that in this case $[W]$ is \emph{the reduced Gr\"{o}bner-Shirshov basis} of $J$, and $[W]_{ass}$ is \emph{the reduced Gr\"{o}bner basis} of the associative ideal $I =
(J_{ass})$ in $K\asX$.
  \end{definition}

\begin{definition}
 \label{def_composition}
Let $W$ be an antichain
of Lyndon words, $W \subset L \setminus X$.
\begin{enumerate}
\item \cite{Anick85, Anick86}  \emph{A 2-prechain} is a word $\omega = hbt$, where $h, b, t \in X^{+}$,
 and $u= hb\in W, \;v = bt\in W$. A 2-prechain $\omega$ is a \emph{2-chain} if no proper left segment of
 $\omega$ is a 2-prechain.
 Note that  every 2-prechain $\omega $ is a Lyndon
word,  see for example \cite{GIF}, Proposition 6.4.

\item Suppose $\omega = hbt$ is \emph{ a 2-prechain},
 where $u= hb\in W, \;v = bt\in W$, and $h, b, t \in X^{+}$.
The difference of Lie elements
\[(u,v)_{\omega}= ([u]t)-(h[v])=([u]t)-[\omega] \]
is called \emph{a Lie composition of overlap}, or also \emph{a Lie composition of intersection} (in the sense of Shirshov).
\item Suppose $u,\omega\in L$, with $\omega =aub$.
\emph{The composition of inclusion} (in the sense of Shirshov), corresponding to the pair $u,\omega$, is defined as
\[(\omega,u)_{\omega}:= [\omega]- (a[u]b). \]

Note that by assumption $W$ is \emph{an antichain} of Lyndon words, that is no $u\in W$ is a proper sub-word of some $v\in W$.
Thus compositions of inclusion, in the sense of Shirshov, will not occur in the initial step of the procedure for finding a Gr\"{o}bner-Shirshov basis
of the Lie ideal $([W])_{Lie}$. However, in the process of resolving a composition of overlap, compositions of inclusion may also occur.
\end{enumerate}
\end{definition}

\begin{remark}
\label{remark}
Notation as above.
\begin{enumerate}
\item
The Lie element $(u,v)_{\omega}= ([u]t)-[\omega]$ is in the ideal $J$.
Indeed, $u, v \in W$, thus $([u]t)\in J$ and $(h[v])= [hv]= [\omega]\in J$.
\item
One has $\overline{([u]t)} = \omega =\overline{([\omega])}$, hence, either $(u,v)_{\omega} =0,$ or
 \[0 \neq (u,v)_{\omega},\; \text{and}\; \overline{(u,v)_{\omega}}\prec\omega.\]
\item
Suppose a two-prechain $\omega \in L_{\alpha}$, notation as above.
Then each of the Lie elements  $([u]t)$, and $(h[v])=[\omega]$ is in the graded component $\mathfrak{L}_{\alpha}=\Span [L_{\alpha}]$. Therefore
$(u,v)_{\omega}= ([u]t)-[\omega]$ is a homogenous Lie element of multi-degree $\alpha$, or $(u,v)_{\omega}=0$.
More precisely, $(u,v)_{\omega} \in \Span [L_{\alpha}(\prec \omega)]$.
In particular, if $\omega$ is the minimal element of $L_{\alpha}$ (w.r.t. $\prec$) then, clearly,  $(u,v)_{\omega}= ([u]t)-[\omega]= 0$, and the composition is solvable
\end{enumerate}
\end{remark}

\begin{definition}
\label{solvabledef}
Suppose $\omega$ is a 2-prechain of multi-degree $\alpha$, so $\omega\in L_{\alpha}$. The composition of overlap $(u,v)_{\omega}$ is \emph{solvable mod $[W]$} if either
$(u,v)_{\omega}=0$, or
there exist an explicit presentation
\begin{equation}
\label{composition_eq}
(u,v)_{\omega} = c_0(a_0[w_0]b_0) + c_1(a_1[w_1]b_1) + \cdots + c_{j}(a_j[w_j]b_j),
\end{equation}
where
(i) $a_jw_jb_j\prec \cdots \prec a_1w_1b_1\prec a_0w_0b_0 \prec \omega$ are words
in $L_{\alpha}$; (ii) $w_i \in W$, and $(a_i[w_i]b_i)$ is a regular bracketing respecting the right standard bracketing $[w_i], \; \forall i, \;0 \leq i \leq j$.
\end{definition}
One can find effectively whether a presentation of the form (\ref{composition_eq}) exists.
If not, there exists a presentation  \[(u,v)_{\omega}= c_0[v_0] + \cdots + c_s[v_s], \quad \text{mod}\; J,\]
where $c_0 \neq 0,$ $\omega\succ v_0$, and
$v_0\succ \cdots \succ v_s$ are Lyndon atoms in $N \bigcap L_{\alpha}$.
In this case $W$ is not the obstruction set of the associative ideal $I= (J_{ass})$, neither is $N$ the set of Lyndon atoms for $U$.

The next corollary follows from Anik's results, see \cite{Anick86}, 2.\emph{Diamond Lemma revisited} (pp 646-647), part (ii) follows from \cite{LR}.
\begin{corollary}
\label{Anick_DLrevis_cor}
If all compositions of overlap on 2-chains are solvable modulo $[W]_{ass}$ in the free associative algebra $K\asX$) then the compositions on 2-prechains are also solvable, so (i) $[W]_{ass}$
is a Gr\"{o}bner basis of the associative ideal $I= ([W]_{ass})$ in $K\asX$; (ii)   $[W]$ is a  Gr\"{o}bner-Shirshov basis of the Lie ideal $([W])_{Lie}$.
 \end{corollary}

The following statement is a consequence from well-known results on Gr\"{o}bner bases, Corollary \ref{Anick_DLrevis_cor},
and our assumption that $W$ is an antichain of Lyndon words.
\begin{proposition}
\label{theor_easy_vip}
Let $(N,W)$ be a Lyndon pair, let $J=([W])_{Lie}$ be the Lie ideal of $Lie(X)$, generated by $[W]$,  let $I$ be the two sided ideal $I = ([W]_{ass})$  in the free associative
algebra $K \asX$. Suppose
$\mathfrak{g}=Lie(X)/J$, so $U= U\mathfrak{g} = K \asX/I$ is the enveloping algebra of $\mathfrak{g}$.
The following conditions are equivalent.
\begin{enumerate}
\item
 \label{theor_easy_vip1}
 The set of Lie monomials $[W]$ is a Gr\"{o}bner-Shirshov basis of the Lie ideal $J$, see Definition \ref{GS-basis_def}.
 \item
 \label{theor_easy_vip2}
For every 2-chain $\omega=abc, u=ab\in W, v =bc\in W$,
    the Lie composition of overlap $(u,v)_{\omega}$ is solvable mod $[W]$.
 \item
\label{theor_easy_vip3} The set
\[[N]= \{[u]\mid u \in N \}\]
is a K-basis of $\mathfrak{g}$.
    \item
     \label{theor_easy_vip4}
    The set $[W]_{ass}$ (considered as a set of "associative" elements of $K\asX$) is a Gr\"{o}bner basis of the (associative) ideal $I = ([W]_{ass})$  (w.r.t. the ordering
    $\prec$ on $X^{*}$).
 \item
  \label{theor_easy_vip5}
 The enveloping algebra $U\mathfrak{g} = K \asX/I$ is a graded algebra whose set of obstructions is $W$,
 that is every $f\in I$ satisfies $\overline{f} = awb,$
 for some $w \in W, a, b \in X^{*}$.
     \item
      \label{theor_easy_vip6}
     The set of monomials
     \[\mathfrak{N}=\{l_1l_2\cdots l_s \mid s \geq 1, \; l_1 \geq  l_2 \geq \cdots \geq l_s, \;  l_i \in N, \; 1 \leq i \leq s\}\bigcup \{1\}\]
     is a normal $K$- basis of $U\mathfrak{g}$.
    \end{enumerate}
 In this case $N$ is a connected set of Lyndon atoms.
\end{proposition}

\begin{defnotation}
\label{m(N)c(W)}
Let $(N,W)$ be a Lyndon pair. We introduce the following numerical invariants.
\begin{equation}
\label{m(N)}
 \begin{array}{l}
 m = m(N) := \max \{|u|\mid u\in N\}\\
 c= c(W):= \min \{|\omega|\mid \omega \;\text{is a two-chain on } W=W(N)\}.
 \end{array}
  \end{equation}
\end{defnotation}
 The following Proposition gives \emph{some purely combinatorial conditions} sufficient for $[W]$ to be a Gr\"{o}bner-Shirshov basis of the Lie ideal $J = ([W])_{Lie}$.

\begin{proposition}
\label{SimplifyProp}
In notation and assumptions of Proposition \ref{theor_easy_vip}.
Let $(N,W)$ be a Lyndon pair.
\begin{enumerate}
\item Suppose $u=ab, v =bc\in W$, and $\omega=abc$ is a 2-chain on $W$ of multi-degree $\deg\omega =\alpha$, or equivalently, $\omega \in L_{\alpha}$. If $(L_{\alpha}(\prec \omega)) \bigcap N =
    \emptyset$,
then the composition of overlap $(u,v)_{\omega}$ is solvable mod $[W]$. Clearly, this is so, whenever $m(N) < |\omega|.$
\item Suppose $u \in W$, and $\omega = aub\in L_{\alpha}$.
The composition of inclusion corresponding to the pair $\omega, u$,
$(\omega,u)_{\omega}:= [\omega]-(a[u]b)$, is solvable mod $[W]$, whenever $(L_{\alpha}(\prec \omega)) \bigcap N = \emptyset$.
In particular, this is so if $m(N) < |\omega|$.
\item Suppose $|N|= d <\infty$, $N$ is connected and some of the following conditions is in force:
 \begin{enumerate}
 \item
 \label{SimplifyPropa}
 $m(N) < c(W)$ that is each $2$-chain $\omega$ on $W$ has length $|\omega|> |u|, \forall u \in N$;
 \item
  \label{SimplifyPropb}
 Every 2-chain $\omega$ on $W$ is a minimal word in its $\mathbb{Z}^n$-component, $L_{\alpha}$, that is $L_{\alpha}(\prec \omega)= \emptyset$.
 \item
  \label{SimplifyPropc}
 Every $u \in N$,  belongs to some $\mathbb{Z}^n$-component, $L_{\alpha}$, which does not contain any $2$-chain $\omega$.
 \end{enumerate}
Then $[W]$ is a Gr\"{o}bner-Shirshov  basis of the Lie ideal $J = ([W])_{Lie}$ in $\Lie(X)$,
the Lie algebra $\mathfrak{g} = \mathfrak{g}_W$ is a standard monomial Lie algebra
with $\dim\mathfrak{g} = d$ and
a $K$-basis $[N]$.
In this case the enveloping algebra $U = U\mathfrak{g}$  belongs to the class $\mathfrak{C}(X,W)$ and is an Artin-Schelter regular algebra of $\gldim U =
d$.
\end{enumerate}
\end{proposition}

\section{Some classical monomial Lie algebra defined by Lyndon words}
\label{nilpotentsec}
\subsection{The free nilpotent Lie algebra $\mathfrak{L}(m)$ of nilpotency class $m$ is a standard monomial Lie algebra defined by Lyndon words}

\begin{notation}
\label{J_m_notation}
Denote by $J_{m+1}$ the Lie ideal of $\Lie(X)$ \emph{generated} by all $m$-(left) nested Lie
monomials $f_{m+1}$ (length $m+1$), see (\ref{Lmeq2}).
\end{notation}
The definition of $J_{m+1}$ and (\ref{Lmeq1}) imply
\begin{equation}
\label{Lmeq2}
\begin{array}{lll}
 J_{m+1} &\supseteq & \Span \{f_{m+1}= [[[[x_{i_1}, x_{i_2}], \cdots], x_{i_m}],
x_{i_{m+1}}]\mid x_{i_s} \in X, \;1 \leq s \leq m+1\}\\
&=& \Span[L_{m+1}].
\end{array}
\end{equation}
Moreover,
\begin{equation}
\label{Lmeq3}
J_{m} \supseteq J_{m+1} \supseteq \cdots, \quad m \geq 2.
\end{equation}

\begin{remark}
\label{JacobsonNil} Suppose that $J$ is a Lie ideal in $\Lie (X)$ and $\mathfrak{g} = \Lie (X)/J$ is a graded
 \emph{nilpotent Lie algebra of nilpotency class}
$m$.
Then all $m$-left nested Lie monomials $f_{m+1}$ of degree $\geq m+1$ are zero in $\mathfrak{g}$.

The Lie algebra $\mathfrak{L}(m) = \Lie(X)/ J_{m+1}$ is the
\emph{free nilpotent Lie algebra of degree} $m$, see \cite{jacobson}.
\end{remark}

Recall that $L_s$ denotes the set of all Lyndon words of length $s, s = 1, 2, 3, \cdots$.
Denote by $L(m)$, $m\geq 2$,  the set of all Lyndon words of length $\leq m$.
Clearly, $L(m)$  is closed under taking Lyndon subword, so \textbf{C1} and \textbf{C2} are in force.
Let
$W(m)$ be (the unique) antichain of Lyndon words whose set of Lyndon atoms is $L(m)$, so
\begin{equation}
\label{freeNilp_eq0}
\begin{array}{ll}
L(m)&= \{u \in L\mid |u|\leq m\}= \bigcup_{1 \leq s \leq m} L_s, \\
W(m)&:=W(L(m)), \; \; \text{and $(L(m), W(m))$ is a  Lyndon pair}
\end{array}
\end{equation}
 Consider the Lie algebra
\begin{equation}
\label{freeNilp_eq1}
\mathfrak{F}_m = \Lie(X)/J(m), \quad J(m) : = ([W(m)])_{Lie},
\end{equation}
where $J(m)$ denotes the Lie ideal of $\Lie (X)$ generated by the set of Lyndon-Lie monomials
\[[W(m)]= \{[w]\mid w \in W(m)\}.\]

\[  W(m)\bigcap L_s = \emptyset, \; 1 \leq s \leq m, \quad W(m) \supset L_{m+1} \]

\begin{corollary}
\label{freeNilp_cor}
In notation as above the following conditions hold.
\begin{enumerate}
\item
The set $[W(m)]$ is a Lyndon-Shirshov Lie basis of the Lie ideal $J(m)$, so $\mathfrak{F}_m$ is \emph{a standard monomial Lie algebra},  see Definition \ref{monLiedef}
\item
The Lie ideal $J_{m+1}$ generated by all $m$-nested Lie
monomials (Notation \ref{J_m_notation}) coincides with the ideal $J(m)$,
generated by the Lyndon-Lie monomials $[W(m)]$,
$J_{m+1}=J(m)$.
\item The free nilpotent Lie algebra $\mathfrak{L}(m) $  of nilpotency class $m$
is isomorphic to the standard monomial Lie algebra $\mathfrak{F}_m$, given in (\ref{freeNilp_eq1}).
\item
Moreover,
the enveloping algebra $U\mathfrak{L}(m)$ is standard finitely presented as
\[U\mathfrak{L}(m) \cong K\asX/([W(m)]_{ass}) = U\mathfrak{F}_m,\]
where the finite set $[W(m)]_{ass}$  is an associative Gr\"{o}bner basis for the ideal $I = ([W(m)]_{ass})$ of $K\asX$.
\item The class $\mathfrak{C} (X, W(m))$ contains an AS-regular algebra of global dimension $d=|L(m)|$,
namely the enveloping algebra $U(m) = U\mathfrak{L}(m)$.
\end{enumerate}
\end{corollary}
\begin{proof}
(1) Every 2-chain $\omega$ on $W(m)$ has length $|\omega|\geq m+2$, and the corresponding composition is solvable, since there are no atoms of length $\geq m+1$. Hence the set
$[W(m)]$ is a Lyndon-Shirshov basis of the Lie ideal $J(m)= ([W(m)])_{Lie}$.
(2)
Clearly, every $w \in W(m)$ has length $q= |w|\geq m+1$, hence $[w]\in [L_{q}] \subset J_q \subseteq J_{m+1}$  (by (\ref{Lmeq2}) and (\ref{Lmeq3})). Since the set of generators
$\{[w]\mid w  \in W(m)\}$ for the Lie ideal $J(m)$ is contained in $J_{m+1}$, one has
 \[J(m) \subseteq J_{m+1}.\]
But $L_{m+1}\subset W(m)$, and (\ref{Lmeq1}) gives
\[
Span_K \{[u]\mid u \in L_{m+1}\}
= \Span_K \{[[\cdots[[x_{i_1}, x_{i_2}], \cdots ], x_{i_{m}}],x_{i_{m+1}}]\}, \]
so $J_{m+1} \subseteq J(m)$,
and therefore $J(m) = J_{m+1}$.
(3)
The equality of the two ideals implies straightforwardly  \[\mathfrak{F}_m = \Lie(X)/J(m) \cong \Lie(X)/J_{m+1} =\mathfrak{L}(m),\]
so $\mathfrak{F}_m$ is isomporphic to the free nilpotent Lie algebra of
 of degree $m$ (or nilpotency class $m$)
see Remark \ref{JacobsonNil}.
(3) implies (4). (5) follows straightforwardly from Proposition \ref{UgProp}.
\end{proof}

\begin{remark} The free nilpotent algebra $\mathfrak{L}(m)$ with
$g$ generators
$x_1, \cdots, x_g$ and
nilpotency class $m$,
 has a $K$-basis
\[\{[w]\mid w \in \bigcup_{1 \leq s \leq m} L_{s}\}, \]
and dimension
\[\dim \mathfrak{L}(m)= |L(m)|= \sum_{1 \leq i \leq m} M(g,i),\]
where $|X|=  g$,  $M(g,i)$ is the necklace polynomials, or
(Moreau's) necklace-counting function
\[
g^n = \sum_{d|n} d
M(g, d).\]
One can also use the Witt's formula
\[
|L_n| = \frac{1}{n} \sum_{d|n} \mu (d)  (g)^{n/d} ,\; \text{where}\; \mu\;\text{is the M\"{o}bius function}.
\]
 The numbers of binary Lyndon words of each length, starting
with length one, form the integer sequence

\[  2, 1, 2, 3, 6, 9, 18, 30, 56, 99, 186, 335, \cdots .\]
\end{remark}

\begin{corollary}
Let $J$ be a Lie ideal in $\Lie(X)$, $\mathfrak{g} = \Lie(X)/J$, $U= U\mathfrak{g}= K \asX/ I$ (as in Notation-Convention
\ref{notconventionB}), and let $(N,W)$ be the corresponding Lyndon pair for the
enveloping algebra $U$.
Suppose the set of Lyndon atoms $N =N(W)$ is finite,  and
$m > 2$ is the minimal integer, such that $N$ contains atoms of length $m$, but there are no atoms of length $m+1.$
Then all $m$-nested Lie monomials $f_{m+1}$  belong to $J$, so $J_{m+1}\subseteq J$, and therefore the algebra $\mathfrak{g} = Lie(X) /J$ is nilpotent of degree $m$.
In particular, $m = \max\{|u|\mid u \in N\}$.
\end{corollary}

\begin{remark}
\label{Obstructionsrem}
 Let $\mathfrak{g} = \Lie(X)/J$, where $J$ is a Lie ideal in $\Lie(X)$, and $U= U\mathfrak{g}= K \asX/ I$ be its enveloping algebra.
Suppose $(N, W)$ is the  Lyndon pair for the enveloping algebra $U$, so  $W$ is the set of obstructions for $U$. In general, this datum does not determine uniquely the
Gr\"{o}bner -Shirshov basis of the Lie ideal $J$. It is possible that $[W]$ is a Gr\"{o}bner-Shirshov basis of the Lie ideal $T = ([W])_{Lie}$, and $W$ is the obstructions set of
$J$, but $J \neq T$. However in some "extreme cases", see Corollaries \ref{Obstructionscor1} and \ref{Obstructionscor2}
the set of obstructions determines uniquely the defining relations of $U$, and the ideal $I$.
\end{remark}
\begin{corollary}
\label{Obstructionscor1}
Let $J$ be a Lie ideal in $\Lie(X)$, where $\mathfrak{g} = \Lie(X)/J$ is a graded Lie algebra, $U= U\mathfrak{g}= K \asX/ I$ (as usual, in Notation-Convention \ref{notconventionB}) and suppose the  Lyndon pair for the
enveloping algebra $U$ is exactly $(N,W)= (L(m),W(m))$, \ref{freeNilp_eq0}.
Then $J =([W(m)])_{Lie}$
\end{corollary}

\subsection{The $\mathbb{Z}^2$-graded Filiform Lie algebras, basic facts}
\label{MonFiliformSec0}
Filiform Lie algebras have been actively studied and numerous works on their algebraic and geometric properties can
 be found in the literature.
As an application of our approach (and for completeness) we shall show that each of the classical Filiform algebras
$\Lcal_n$ and  $\Qcal_n$
has a canonical presentation as a
naturally  $\mathbb{Z}^2$ -graded standard monomial Lie algebra defined by Lyndon words.
 Similar presentation is given for the infinite Filiform Lie algebras $\mathfrak{M}_0$, but, of course, its set of defining Lyndon-Lie monomial relations is infinite.

Recall that \emph{a finite dimensional filiform  Lie algebra}  is a nilpotent Lie algebra $\mathfrak{F}$ whose nil index is maximal, that is
$\mathfrak{F}$ has nil index $n$, (nilpotency class $n$), where $\dim \mathfrak{F}= n+1$. It is well-known that, \emph{up to isomorphism}, there are
two families of filiform Lie algebras of rank 2.

\[\{\mathcal{L}_n:\;\; n \geq 2 \}\;\; \text{and} \;\; \{\mathcal{Q}_n:\;\; n \geq 5, \; n\;\; \text{odd} \}. \]

Both algebras $\mathcal{L}_n$  and $\mathcal{Q}_n$ have dimension $n+1$. Each of them has a basis
 $X_0, X_1, \cdots , X_n$,
 their nonzero brackets of basis elements are given below.
 \begin{equation}
\label{filiformL_n}
\begin{array}{ll}
 \Lcal_n: \quad &[X_0, X_j]= X_{j+1}, \; \text{for}\; 1 \leq j \leq n-1.
 \end{array}
 \end{equation}

 \begin{equation}
\label{filiformQ_n}
\begin{array}{ll}
 \Qcal_n: \quad &[X_0, X_j]= X_{j+1}, \; \text{for}\; 1 \leq j \leq n-1\\
                &[X_j, X_{n-j}]= (-1)^i X_{n}, \; \text{for}\; 1 \leq j \leq n-1 \\
                &n\geq 5, \; n \; \text{is odd}.
                \end{array}
 \end{equation}

 Infinite Filiform Lie algebra $\mathfrak{M}_0$, \cite{FialWag}, is presented via generators $e_i, i \geq 1$, and relations $[e_1, e_i]= e_{i+1}, \forall i\geq 2.$
 This algebra was first studied in \cite{M.Smith}, where was shown that its enveloping $U$ is an integral domain with a subexponential growth, not bounded by any polynomial.

Note that if $\mathfrak{g}$ is a graded Lie algebra generated by generators of degree $1$,  $\mathfrak{g}_1 = Span X$,
then the data
$\dim \mathfrak{g} = n+1; \; \mathfrak{g} \in \mathfrak{N}_n$
implies
\begin{equation}
\label{filiformeqA}
(i)\;\;  2 =\dim \mathfrak{g}_1 = |X|; \quad (ii) \;\; \dim \mathfrak{g}_k = 1,  \; 2 \leq k \leq n.
\end{equation}

At the same time we have a result on "the minimality of the order $|W|$ of $W$ independently of the context of Lie algebras.
Recall that if $(N,W)$ is a Lyndon pair in $X$, with  $|N|=d$,
Theorem I implies that
$W$ is a finite set with $d-1 \leq |W|
\leq \frac{d(d-1}{2}$.  Theorem  \ref{Minim_W_Theor1} shows that \emph{if $N$ is connected} then the pure \emph{numerical datum}
 \begin{equation}
\label{numericaleq}
 |N|=d, \quad |W|= d-1
\end{equation}
 determines uniquely (up to isomorphisms of monomial algebras $A_W$): (i) the order of $X$, $\;|X|=2$, and (ii) the Lyndon pair $(N, W)$ in $X^{+}$
 with $|W|= d-1$:
 \begin{equation}
\label{NWFiliformeq}
\begin{array}{ll}
N &= \{x < xy <  xyy < \cdots <xy^{d-3}< xy^{d-2}< y\}.\\
W &= \{xy^{i}xy^{i+1}\mid 0 \leq i \leq d-3 \}\bigcup\{xy^{d-1}\}.
\end{array}
\end{equation}

\begin{corollary}
\label{MinimalWCor}
Let $X = \{x< y\}$, let $d\geq 2$ be an integer, and let
 $(N,W)$ be the Lyndon pair (\ref{NWFiliformeq}).
 Then
\begin{enumerate}
\item
The set of Lie monomials $[W]$ is a Lyndon-Shirshov Lie basis of the ideal $J= ([W])_{Lie}$, so
$[N]$ is a $K$-basis of $\mathfrak{g}$, and $W$ is the set of obstructions of the enveloping algebra $U\mathfrak{g}$.
\item
The 2-generated monomial Lie algebra $\mathfrak{g} = Lie(X)/ J$ is isomorphic to the standard  filiform Lie algebra $\Lcal_{d-1}$ of dimension $d$ and  nilpotency class
$d-1$.
\item
In particular, the class
     $\mathfrak{C} (X, W)$ contains the enveloping $U =
     U\mathfrak{g}$ which is an AS-regular algebra of global dimension $d$.
     \end{enumerate}
\end{corollary}
\begin{proof}
Each two-chain $\omega$ on $W$ has multi-degree $\deg (\omega) =\alpha =(k,j)$, where $k \geq 2$. But $L_{\alpha}\bigcap N = \emptyset$, hence by Proposition
 \ref{SimplifyProp} part (\ref{SimplifyPropc})
 $[W]$ is a
 Lyndon-Shirshov basis of the Lie ideal $J= ([W])_{Lie}$. This implies that $[N]$ is a $K$ basis of $\mathfrak{g}$, and $W$ is the set of obstructions of the enveloping algebra $U= K\asX/([W])_{ass}$.
 Clearly, $\mathfrak{g}\cong \Lcal_{d-1},$ the standard filiform Lie algebra of dimension $d$, and nilpotency class $d-1$.
 \end{proof}
\subsection{Proof of Theorem III}
\label{proof3}
\begin{proof3}
(1). Theorem \ref{Minim_W_Theor1} proves the equivalence of conditions (1.a), (1.b), (1.c).
In this case the Lyndon pair $(N,W)$ is given by (\ref{NWFiliformeq}), hence the hypothesis of Corollary
\ref{MinimalWCor} is satisfied. This implies that
the set of Lie monomials $[W]$ is a Lyndon-Shirshov Lie basis of the ideal $J= ([W])_{Lie}$ in $\Lie(X)$  so
$\mathfrak{g} = Lie(X)/ ([W])_{Lie} = \mathfrak{g}_W$ is isomorphic to the standard  filiform Lie algebra $\Lcal_{d-1}$ of dimension $d$ and  nilpotency class
$d-1$.
In particular, the class
     $\mathfrak{C} (X, W)$ contains the enveloping $U =
     U\Lcal_{d-1}$ which is an Artin-Schelter regular algebra of global dimension $d$. This proves (1)

 (2). The equivalences of conditions (2.a), (2.b) and (2.c) follow from \cite{GIF}, Theorem B (2).
 In this case $n = d$ and $W = \{x_i x_j \mid 1 \leq i < j \leq n\}$. The class $\mathfrak{C}$ contains all binomial skew polynomial rings introduced and studied by the author,
 see \cite{GI96}. Each of them is an Artin-Schelter regular PBW algebra of global dimension $n$, see \cite{GIVB}, and also  \cite{GI12}, and defines in a natural way a square-free solution of the Yang-Baxter equation.
 When $|X|= 8$ more than $2400$ (non-isomorphic) square-free solutions were found by Schedler via a computer program.
\end{proof3}

\subsection{The $\mathbb{Z}^2$-graded Filiform Lie algebras presented as standard monomial Lie algebras defined by Lyndon words}
\label{MonFiliformSec}
In this subsection, as an application of our approach (and for completeness) we show that each of the classical Filiform algebras
$\Lcal_n$ and  $\Qcal_n$
has a canonical presentation as a
naturally  $\mathbb{Z}^2$ -graded standard monomial Lie algebra defined by Lyndon words.
\begin{theorem}
\label{filiformProp}
Let $(N,W)$ be a Lyndon pair, in the alphabet $X = \{x <y\}$.
Let $J= ([W])_{Lie}$, and $I$ be the corresponding ideals, and
$\mathfrak{g} = Lie(X)/J, \; U = U\mathfrak{g}= K\asX /I$.
Assume that $N$ is connected, with order $|N|\geq 4$, and
\begin{equation}
\label{narroweq}
|N\bigcap L_m| \leq 1, \forall  m \geq 2.
\end{equation}
Without loss of generality we may assume that $xyy\in N$.

The following conditions hold.
\begin{enumerate}
\item
$N$ is infinite if and only if
\begin{equation}
\label{InfiniteNeq}
\begin{array}{l}
N = \{xy^k \mid k \geq 0\}\bigcup \{y\},\quad\text{or equivalently,}\quad
W = \{xy^kxy^{k+1} \mid k \geq 0\}.
\end{array}
\end{equation}
$[W]$ is a Gr\"{o}bner-Shirshov basis of $J$, so $\mathfrak{g}= \mathfrak{g}_W$ is a standard monomial Lie algebra with a $K$-basis $[N]$, and
$\mathfrak{g} \simeq \mathfrak{M}_0$, the infinite-dimensional Filiform algebra, see \cite{FialWag}.
\item
$N$ is a finite set of order $d$ if and only if one of the conditions (a), or (b) is in force.
\begin{enumerate}
\item
\begin{equation}
\label{finiteNeq1}
\begin{array}{l}
 N = \{xy^k \mid 0 \leq k \leq d-2\}\bigcup \{y\} = N(\Lcal_{d-1})\\
 W = \{xy^kxy^{k+1} \mid  0 \leq k \leq d-2 \}\bigcup \{xy^{d-1}\}.
\end{array}
\end{equation}
In this case $\mathfrak{g}=\mathfrak{g}_W$ is a standard monomial algebra with a $K$-basis $[N]$, moreover,
 $\mathfrak{g} \cong \Lcal_{d-1}$, the standard filiform Lie algebra of dimension $d$ and nilpotency class $d-1$.
\item One has $d = 2m \geq 6$ and $(N, W) = (N_{2m}, W_{2m})$, where
\begin{equation}
\label{filiformeq2}
\begin{array}{ll}
N_{2m} &= \{xy^j \mid 0 \leq j \leq 2m-3\}\bigcup \{xy^{m-2}xy^{m-1}\}\bigcup \{y\} = N(\Qcal_{2m-1})\\
&\\
W_{2m} = &\{xy^{2m-2}, \;  (xy^{m-2})(xy^{m-2}xy^{m-1}),\;
           (xy^{m-2}xy^{m-1})(xy^{m-1})\}\\
         &\\
        &\bigcup\{xy^kxy^{k+1}\mid 0 \leq k \leq 2m-4, \; \text{and} \; k \neq m-2 \}.
         \end{array}
\end{equation}
Again, $\mathfrak{g}=\mathfrak{g}_W$ is a standard monomial algebra with a $K$-basis $[N]$,
 $\mathfrak{g} \cong \Qcal_{2m-1}$, the second type Filiform Lie algebra of dimension $d= 2m\geq 6$ and nilpotency class $d-1$.
\end{enumerate}
\end{enumerate}
\end{theorem}

\begin{lemma}
\label{filiformLemma}
In notation and assumptions as in the theorem, let $d=2m \geq 6$.
Let $(N, W)$ be a Lyndon pair, where $N = \{x =l_1<
l_2 < \cdots < l_{d}=y\}$  satisfies:
\begin{equation}
\label{filiformeq0}
|N \bigcap L_k|= 1,   \; \forall k, 2\leq k \leq d.
\end{equation}
Suppose $xxy \in W$. The following are equivalent.
\begin{enumerate}
\item  $[W]$ is a Lyndon-Shirshov Lie basis of the ideal $J$, and  $N$ contains an atom different from $xy^k$, for $k \geq 1$.
\item
\begin{equation}
\label{filiformeq1}
N = N_{2m} = \{xy^k\mid 0 \leq k\leq 2m-3\}\bigcup \{y\}\bigcup
\{xy^{m-2}xy^{m-1}\},
\end{equation}
and $[N]$ is a $K$-basis of $\mathfrak{g}$.
\item
The Lie algebra $\mathfrak{g}$ is generated by the set $X= \{x < y\}$, $W=W_{2m}$
and its Lyndon-Lie monomial relations $[W]$ form a Gr\"{o}bner-Shirshov basis, where
\begin{equation}
\label{filiformeq2}
\begin{array}{ll}
W = W_{2m} = & \{xy^kxy^{k+1}\mid k \neq m-2,  0 \leq k  \leq 2m-4\}\bigcup \{xy^{2m-2}, xy^{m-2}xy^{m}\})\\
          &\bigcup\{(xy^{m-2})(xy^{m-2}xy^{m-1}),\;
           (xy^{m-2}xy^{m-1})(xy^{m-1})\}.
\end{array}
\end{equation}
\end{enumerate}
\end{lemma}
\begin{proof}
By hypothesis $xxy\ W$, so $xyy\in N$.
 It is clear that $(N_{2m}, W_{2m},)$ given in (\ref{filiformeq1}) and (\ref{filiformeq2}), respectively, is a Lyndon pair.
The implication $(2) \Longrightarrow (1)$ is obvious.

We shall prove that for $W= W_{2m}$, the bracketing $[W]$ is a Gr\"{o}bner-Shirshov basis of the Lie ideal $J= ([W])_{Lie}$, we use Proposition \ref{SimplifyProp}.
Indeed, there are exactly three-types of two-chains $\omega$ on $W$. (a) $\omega \in L_{\alpha}$, where $\alpha= (2, k), k\geq 2m-2$. Hence $|\omega|\geq 2m > max \{|l|\mid l \in N\}= 2m-1$,
and $N \bigcap L_{\alpha} = \emptyset$;  (b)   $\omega \in L_{\alpha}$ where $\alpha= (3, s), s \geq 3$, then, clearly,
$N \bigcap L_{\alpha} = \emptyset$;
(c) $\omega \in L_{\alpha}$ where $\alpha= (4, s), s \geq 2m-1$, hence $N \bigcap L_{\alpha} = \emptyset$ again. It follows from Proposition \ref{SimplifyProp} part (1) that in each of the cases (a), (b), and (c)  the corresponding composition of overlap is solvable, therefore $[W]$  is a Gr\"{o}bner-Shirshov basis of the Lie ideal $J$.
Now the equivalence $(2) \Longleftrightarrow (3)$ is clear.

(1) $\Longrightarrow$ (2).
By the hypothesis of (1), $[W]$ is a Gr\"{o}bner- Shirshov basis of the ideal $J$. Moreover $N$ contains an atom of multi-degree $(k, s)$, $k \geq 2, s \geq 1$
Since $xxy \in W$, and $N$ is closed under taking Lyndon subwords, it follows that $N$ contains a subword of
the shape $xy^kxy^m\cdots$,  with $k, m \geq 1$.
Let $u$ be the shortest word in $N$ with this property. Clearly, $u = xy^{s-1}xy^{s}\in N$, for some $s \geq 2.$ Moreover,
$N$ contains a convex subchain
\begin{equation}
\label{filiformeq3}
l_1= x <l_2 = xy<\cdots l_s =  xy^{s-1} <l_{s+1}=xy^{s-1}xy^{s}.
\end{equation}
In particular $xy^{s-1}xy^{s}\in N$, with $|xy^{s-1}xy^{s}|= 2s+1$, and since $N$ contains no other atom of length $2s+1$, one has
\begin{equation}
\label{filiformeq3}
xy^{2s} \in W, \quad xy^{2s+t} \; \text{is not in} \; N, \forall t \geq 0.
\end{equation}
We claim that  $u=l_{s+1}=xy^{s-1}xy^{s}$ is the longest Lyndon atom in $N$, and $s = m$.
Note first that $|u|=|l_{s+1}|=|xy^{s-1}xy^{s}| = 2s+1$, and
our assumption (\ref{filiformeq0}) and (\ref{filiformeq3}) give
(a) all elements  $l \in N$ with $2 \leq |l|< 2s+1$
have the shape $l = xy^{k}, 1 \leq k \leq 2s-1$.
This implies an inclusions of sets
\begin{equation}
\label{filiformeq4}
\{xy^{k}xy^{k+1}\mid 0 \leq k \leq s-2\}\bigcup \{xy^{2s}\}\subset W.
\end{equation}
We shall prove that $N\bigcap L_{2s+2}= \emptyset$. Indeed, $xy^{2s+1}$ is not in $N$, and there are no atoms $l \in N$ of length $|l| = 2s+2$
of multi-degree $\alpha = (3, k)$, hence
it will be enough to prove that
\begin{equation}
\label{filiformeq5}
w =xy^{s-1}xy^{s+1}\in W.
\end{equation}

It follows from (\ref{filiformeq4}) that $a=xxy, b= xy^{2s}\in W$, hence
$\omega = xxy^{2s} = (xxy)y^{2s-1}= x(xy^{2s})$ is a two-chain on $W$,  $\omega \in \Lcal(2,2s)$.
So there is a composition of overlap:
\[(xxy,xy^{2s})_{\omega}= ([xxy]y^{2s-1})-(x[xy^{2s}])= [xxy^{2s}]_l -[xxy^{2s}] \in J.\]
which must be solvable, since by assumption $[W]$ is a GS-Lie basis.
Detailed computation implies
\[ [xxy^{2s}]_l -[xxy^{2s}]\equiv  c[xy^{s-1}xy^{s+1}]  \; \text{(modulo $J$)}, \]
for some integer $c\neq 0$.
Hence $[xy^{s-1}xy^{s+1}] \in J$, and therefore $v= xy^{s-1}xy^{s+1}$ is not an atom.
All proper Lyndon subwords of $v$ are atoms, indeed these are
\[xy^{s-1}xy^{s}  \in N, xy^{k} \in N, 0 \leq k \leq s+1, y \in N,\]
therefore $l=xy^{s-1}xy^{s+1} \in W.$

This together with the information we already have (i) $xy^{2s}, xy^{s-1}xy^{s+1} \in W$ and $xy^{s-1}xy^{s}$ is the shortest Lyndon word in which $x$ occurs more than once imply that that there is no Lyndon atom $a \in N$ of length $|a|= 2s+2$.
By assumption $[W]$ is a GS-Lie basis, so by Lemma \ref{VIPproposition}
$N$ is connected, thus $N\bigcap L_t =\emptyset, t \geq 2s+2 $.

It follows that $s = m-1$, and $N$ satisfies(\ref{filiformeq1}), so $N=N_{2m}$, and therefore
$|N| = \dim \mathfrak{g}= 2m$. Clearly, $\mathfrak{g}$ is a filiform Lie algebra of nilpotency class $2m-1,$
and its defining Lyndon Lie monomial relations are $[W_{2m}]$. This proves (1) $\Longrightarrow $(2).
\end{proof}

\begin{proofTheor}\ref{filiformProp}.
Without loss of generality we have assumed that $xyy \in N$, hence $xxy \in W$.
Two cases are possible:  (i) all atoms in $N$ have the shape $xy^k, k \geq 1$, or (ii) $N$ contains an atom different from $xy^k$, for $k \geq 1$.
Assume  (i) is in force.
Then every atom $u\in N$ of length $|u| \geq 2$ has the shape $u_s=xy^s$, where $1 \leq s,$ and exactly
 one of the cases  (i.1), or (i.2) given below  is in force.
(i.1) $N$ contains all words $xy^s, s \geq 0$, then $|N \bigcap L_s| = 1, \forall s \geq 2$. Thus $N$ and  $W$ satisfy (\ref{InfiniteNeq}).
In this case $[W]$ is an (infinite) Gr\"{o}bner-Shirshov basis of the Lie ideal $J$ (this can be proved recursively),  $[N]$
is a $K$-basis of $\mathfrak{g}$ and $\mathfrak{g} \cong \mathfrak{M}_0$, the infinite Filiform Lie algebra.
The case (i.2) is explicitly described by Corollary \ref{MinimalWCor}.
 Assume (ii) is in force, then
 Lemma \ref{filiformLemma}
completes the proof of the theorem .
$\quad \Box$
\end{proofTheor}

\begin{corollary}
\label{Obstructionscor2}
Let $J$ be a  $\mathbb{Z}^n$ (non-negatively) graded Lie ideal in $\Lie(X)$, $\mathfrak{g} = \Lie(X)/J$, $\; U= U\mathfrak{g}= K \asX/ I$ (in usual notation), and suppose the  Lyndon pair for the
enveloping algebra $U$ satisfies one of the conditions
(a) $(N, W) = (N(\Lcal_{d}), W(\Lcal_{d}))$, is the pair corresponding to the Filiform Lie algebra $\Lcal_{d}$ of nilpotency class $d$ and  dimension  $d+1$, see (\ref{finiteNeq1});  or (b)
$(N, W) =  (N_{2m},  W_{2m})$,
is the Lyndon pair corresponding to  Fillifotm Lie algebra of second type,
$\Qcal_{2m-1}$, see (\ref{filiformeq1}), and (\ref{filiformeq2}).
Then the set of obstructions determines uniquely the Gr\"{o}bner-Shirshov basis of $J$ and the defining relations of $U$.
More precisely, in case (a) $J = [W(\Lcal_{d})]_{Lie}]$, $\mathfrak{g}= \Lcal_{d}$, and $U$ is the enveloping algebra of the Filiform Lie algebra  $\Lcal_{d}$, $\gldim U = d+1$;
in case (b)  $J = [W(\Qcal_{2m-1})]_{Lie}]$, $\mathfrak{g}= \Qcal_{2m-1}$, and $U$ is the enveloping algebra of the Filiform Lie algebra  $\Qcal_{2m-1}$, $\gldim U = 2m$.
\end{corollary}
\begin{proof}
Note that in each of the cases (a) and (b) every obstruction $w \in W$ is minimal (w.r.t. $\prec$) in its $\mathbb{Z}^n$-component, $L_{\alpha}$, where $\deg (w) =\alpha$. Hence the only monic homogeneous Lie element $f\in \Lie(X)$ with multi-degree $\alpha$ and with highest monomial $\overline{f} = w$ is $f=[w]$. Therefore in the special cases when one of the conditions (a) or (b) holds, the obstruction set $W$ determines uniquely the set of defining relations $\mathfrak{R}$ for $U$, one has $\mathfrak{R} = [W]$ .
\end{proof}

\section{More on Gr\"{o}bner-Shirshov bases of monomial Lie ideals. Applications}
\label{more.applications}
\subsection{More on Gr\"{o}bner-Shirshov bases of monomial Lie ideals}
We have a natural problem.
\begin{problem}
\emph{Given} $X$ and a Lyndon pair $(N,W)$, where $|N|=d.$
\emph{Decide} whether $[W]$ is a Gr\"{o}bner-Shirshov basis of the Lie ideal $J = ([W])_{Lie}$ in $\Lie(X)$.
\end{problem}
Proposition \ref{theor_easy_vip} (\ref{theor_easy_vip2}) implies an effective method (Procedure) to check whether $[W]$ is a  Gr\"{o}bner-Shirshov basis of the Lie ideal
$J=([W])$.  Moreover, Proposition \ref{SimplifyProp} can be and will be used to simplify the procedure by
omitting the unnecessary computation of compositions corresponding to "long" $2$-chains.

\begin{definition}
\label{standar_dpair_def}
\begin{enumerate}
 \item We say that a Lyndon pair $(N,W)$ is \emph{regular} if (i) $N$ is finite, and (ii) the class $\mathfrak{C} (X, W)$ contains the enveloping $U =
     U\mathfrak{g}$ of a graded Lie algebra $\mathfrak{g}$. In this case (i) $U$ is an s.f.p. AS-regular algebra of global dimension $d= |N|,$ and
(ii) the set $[N]$ is a $K$-basis of  $\mathfrak{g}$.
\item We say that the Lyndon pair  $(N,W)$ is \emph{standard}, if $[W]$ is a Gr\"{o}bner-Shirshov basis of the Lie ideal $J=([W])_{Lie}$ in $\Lie(X)$.
Every standard Lyndon pair $(W,N)$ with a finite set $N$ is regular, since the class $\mathfrak{C} (X, W)$ contains the enveloping algebra $U$ of the standard monomial Lie algebra $\mathfrak{g}_W = \Lie(X) /J$.
\end{enumerate}
\end{definition}

We are interested in Lyndon pairs $(N,W)$, where $N$ is a finite set of order $d$.

\begin{remark}
\label{mon-nonmonrem}
Suppose  $(N,W)$ is a Lyndon pair, and $J = ([W])_{Lie}$  is the corresponding monomial Lie ideal in $\Lie(X)$, $\mathfrak{g}= \Lie(X)/J$, $U = U\mathfrak{g}$ .
\begin{enumerate}
\item
In general, the reduced Gr\"{o}bner-Shirshov basis of the Lie ideal $J = ([W])_{Lie}$ may contain a Lie element $f \in J$,
which is not a Lie monomial, as shows Example \ref{mon-nonmonExample}.
\item If $N$ is a finite set of order $d$, then the monomial Lie ideal $J = ([W])_{Lie}$ has a finite
reduced Gr\"{o}bner-Shirshov basis and determines uniquely the corresponding Lyndon pair $(\widetilde{N}, \widetilde{W})$, where $\widetilde{W}$ is the set of obstructions for
the enveloping algebra $U = U\mathfrak{g}$. One has $\widetilde{N} \subseteq N$.
Moreover, the Lie algebra $\mathfrak{g}= \Lie (X)/J $ is finite dimensional with a $K$-basis $[\widetilde{N}]$ and $\widetilde{d}= \dim \mathfrak{g} = |\widetilde{N}| \leq
d$.
The
enveloping $U = U\mathfrak{g}$ is in the class $\mathfrak{C} (X, \widetilde{W})= \mathfrak{C} (X, \widetilde{N})$, $U$ an AS-regular algebra of $\gldim U = \widetilde{d}$.
The set $[\widetilde{W}]$ is not necessarily a Gr\"{o}bner-Shirshov basis of the Lie ideal $\widetilde{J}= ([\widetilde{W}])_{Lie}$, see Example \ref{mon-nonmonExample}.
\end{enumerate}
\end{remark}
\begin{example}
\label{mon-nonmonExample}
Let $X=\{x<y\}$ and let $(N,W= W(N))$ be the Lyndon pair, determined by the set of Lyndon atoms $N$ given below, $\mathfrak{g}= \Lie(X)/([W])$,notation as usual :
\[
\begin{array}{ll}
N=  &\{x < xxy< xxyxy <xxyxyy < xxyxyy< xxyxyyy <xxyy\\
    &< xxyyxy < xxyyxyy< xxyyy < xxyyyxy< xy<xyxyxyy\\
    &<xyxyy < xyxyyy < xyy < xyyy < y \} \\
|N|&= 18, \; \; m(N) = 7, \quad N \; \; \text{is connected}.\; \\
W& =\{xxxy, xyyyy\}\bigcup\{\text{appropriate Lyndon words of length }\;\geq 7\}.\\
\omega& = (xxxy)yyy= xx(xyyyy) \; \text{ is the unique 2-chain of length}\;|\omega| = 7.\\
      & \text{Every 2-chain $\rho$ on $W$, $\rho \neq \omega$ has length}\; |\rho| > 7.
\end{array}
\]
By Proposition \ref{SimplifyProp}, if a $2$-chain $\rho$ on $W$ has length $|\rho| >m(N)=7$, then the corresponding composition is solvable.
Thus there is only one composition of overlap which needs computation, namely the unique composition corresponding to the two-chain $\omega$ of length $\leq 7$.
One has
\begin{equation}
\label{mon-nonmoneq}
\begin{array}{ll}
 (xxxy, xyyyy)_{\omega}& =[[[[xxxy], y]y]y] - [xxxyyy] \\
                       & =[xxyxyyy]-3[xxyyxyy]-4[xxyyyxy]+3[xyxyxyy] = f
 \end{array}
\end{equation}
The Lie element $f$ is in $J$, it is presented as a linear combination of Lyndon Lie monomials, each of which is in $[N]$ (i.e. normal modulo $[W]$).
Its highest monomial (w.r.t. $\prec$) is $\overline{f}= \overline{[xxyxyyy]}= xxyxyyy$ of length $7$.
The set $\Gamma_1= \{f\}\bigcup [W]$ generates $J$.
Moreover, every composition of overlap is solvable, all
compositions of inclusion correspond to certain monomials of length $\geq 8$, and therefore are solvable.
The unique composition of overlap corresponding  to a two-chain of order $7$ is now solvable, so $\Gamma_1$ is a GS-Lie basis of $J$. Denote
\[W_0= W \setminus \{u\mid u \in W, |u| \geq 8, \overline{f} \sqsubseteq u\}\]
Then the set
\[\Gamma =\{f\}\bigcup [W_0] \]
is the reduced GS Lie basis of the Lie ideal $J$. Note that $[\widetilde{W}]= \{[xxyxyy]\}\bigcup [W_0]$ is not a GS Lie basis of $J$, thus .
Moreover the Lyndon pair $(\widetilde{N},\widetilde{W})$, where
\[
\begin{array}{l}
\widetilde{N}= N \setminus \{xxyxyyy\}, \\
\widetilde{W}= \{xxyxyy\}\bigcup (W \setminus \{u\mid u \in W, |u| \geq 8, \overline{f}= xxyxyyy \sqsubseteq u \})
\end{array}
\]
corresponds to the enveloping algebra $U = U\mathfrak{g}, \mathfrak{g}= \Lie/([W])_{Lie} = \Lie /(\Gamma)_{Lie}.$

Using the terminology from Definition
\ref{standar_dpair_def}, the Lyndon pair $(\widetilde{N},\widetilde{W})$ is regular, i.e. $\widetilde{W}$ is
the set of obstructions of the enveloping $U = U\mathfrak{g}$  (equivalently, $[\widetilde{N}]$ is a $K$-basis of $\mathfrak{g}$).
However, the set $[\widetilde{W}]$ is not a GS-Lie basis of the Lie ideal $([\widetilde{W}])_{Lie}$, hence by the same definition, part (2)  this Lyndon pair is not standard.
\end{example}

Recall that given a Lyndon pair $(N, W)$,
the connected component $N_{con}$ of $N$ is defined in Definition \ref{connect_def} (2).
\begin{proposition}
\label{connected_Prop}
Let $(N,W)$ be a Lyndon pair, $N=  \{l_1< l_2 < l_3\cdots < l_d \}$.
Let $J= ([W])_{Lie}$ be the Lie ideal in  $\Lie(X)$, generated by $[W]$, $\mathfrak{g}= \Lie(X)/J$, and let $I$ be the corresponding associative ideal in $K \asX$.
Denote by $(\widetilde{N}, \widetilde{W})$  the Lyndon pair corresponding to the enveloping algebra $U=U\mathfrak{g} = K \asX/I$.
Let $N_{con}$
be the connected component of $N$, say $N_{con}=  \{l_1= u_1< u_2 <\cdots < u_s \}\subseteq N$.
Then the following conditions hold.
\begin{enumerate}
\item
\label{connected_Prop1}
 $\widetilde{N} = N$ if and only if $[W]$ is the reduced Gr\"{o}bner-Shirshov basis of $J$.
 In this case $N$ is connected.
 \item
 \label{connected_Prop2}
 In the general case, one has $\widetilde{N} \subseteq N_{con} \subseteq N$.
\item
\label{connected_Prop3}
$\widetilde{W}$ is a finite antichain of Lyndon words, so the monomial Lie ideal $J= ([W])_{Lie}$ has a finite reduced Gr\"{o}bner-Shirshov basis $\mathfrak{R}$, and the enveloping $U =
K \asX/ I$
is standard finitely presented.
\item
\label{connected_Prop4}
Assume that $N$ is disconnected, then \[\widetilde{d}=|\widetilde{N}| \leq |N_{con}| = s < |N|=d\] and $U$ is an s.f.p. Artin-Schelter regular algebra of global dimension
$\widetilde{d} < d$.
  \item
  \label{connected_Prop5}
  The Lie algebra $\mathfrak{g}$ has nilpotency class $m(\widetilde{N}) \leq m(N_{con})$.
  \end{enumerate}
\end{proposition}
\begin{proof}
 Note first that if $[W]$ is a Gr\"{o}bner-Shirshov basis of the Lie ideal $J$, then it is a reduced Gr\"{o}bner-Shirshov basis.

(\ref{connected_Prop1}). Assume  $\widetilde{N} = N$, then $\widetilde{W} = W$, and $W$ is the set of obstructions of the ideal $I$, hence
$[W]$ (considered as a set of associative polynomials) is the reduced Gr\"{o}bner basis of the ideal $I$, and therefore  $[W]$ is the reduced Gr\"{o}bner-Shirshov basis of the
Lie ideal $J$.  Conversely, if $[W]$ is the reduced Gr\"{o}bner-Shirshov basis of the Lie ideal $J$ then $[W]$
(considered as a set of associative polynomials) is the reduced Gr\"{o}bner basis of the (associative) ideal $I$, hence $W$ is the set of obstructions for $U$, and
$\widetilde{W}=W,$ hence $\widetilde{N}=N$.

(\ref{connected_Prop2})
Suppose $a \in \widetilde{N}$ is an atom for $U$, then $a$ is normal modulo $I$  (equivalently, $a$ does not contain as a subword the highest term of any element
$g \in I$). In particular, for every $w \in W$ we have $[w]_{ass} \in I$, and $\overline{[w]_{ass}}= w$, hence $w \sqsubseteq a$ is impossible, so $a \in N$. This proves
$\widetilde{N}\subseteq N$, but since $\widetilde{N}$ is connected, one has
$\widetilde{N} \subseteq N_{con} \subseteq N.$

(\ref{connected_Prop3})
By (\ref{connected_Prop2}) $\widetilde{N}$ is finite, with $|\widetilde{N}| = \widetilde{d} \leq d = |N|$, therefore, by Theorem B, $\widetilde{W}$ is also finite and satisfies
\[\widetilde{W} \leq
\frac{\widetilde{d}(\widetilde{d}-1)}{2} < \frac{d(d-1)}{2}.\] It follows that $U$ is standard finitely presented, and the reduced Gr\"{o}bner-Shirshov basis of the Lie ideal $J$ is also finite.

(\ref{connected_Prop4}). The set of Lyndon atoms $\widetilde{N}$ is connected and $\widetilde{N}\subseteq N$ implies  $\widetilde{N} \subseteq N_{con}\subseteq N$.
Part (\ref{connected_Prop5}) is clear.
\end{proof}

We remind that in Section 5 was set Notation-Convention \ref{notconventionB}  valid till the end of the paper, and in particular,  for the following corollary.

\begin{corollary}
\label{AS-corollary}
Suppose $(N, W)$ is a standard Lyndon pair, that is $[W]$ is a Gr\"{o}bner-Shirshov basis of the Lie ideal $J= ([W])_{Lie}$ in $\Lie(X)$), $\mathfrak{g}_W = \Lie(X) /J$.
Suppose $|N|=d$.
Then the universal enveloping algebra
$U = U\mathfrak{g}_W = K\asX / I$ is an AS-regular algebra
of global dimension $\gldim U = d = |N|= \gkdim U.$
Clearly, then the class $\mathfrak{C} (X, W)$ contains AS-regular algebras of global dimension  $d=|N|$.

In particular, every standard Lyndon pair $(W,N)$ with a finite set $N$ is regular.
\end{corollary}

\subsection{Nontrivial disconnected extensions of standard Lyndon pairs}
In this subsection  $X$ has arbitrary finite order $n \geq 2$.  In the cases when Filiform algebras occur, $X = \{x< y\}$, as usual.

\begin{definition}
\label{disconnected_ext_def}
Let $(N^{\star},W^{\star})$ be a standard Lyndon pair  (that is $[W^{\star}]$ is a Gr\"{o}bner-Shirshov basis of the Lie ideal $J= ([W^{\star}])_{Lie}$ in $\Lie(X)$).
A Lyndon pair $(N, W)$ is called \emph{a nontrivial disconnected extension} of $(N^{\star},W^{\star})$ if $N$ is disconnected and $N^{\star} = N_{con} \subsetneqq N$.
\end{definition}

\begin{proposition}
\label{FiliformWProp2}
Suppose $(N,W)$ is a nontrivial disconnected extension of the standard Lyndon pair $(N^{\star},W^{\star})$, where $|N^{\star}|= |N_{con}|= d$.
In notation as usual, $J= ([W])_{Lie}$ is the corresponding Lie ideal in $\Lie(X)$, $\mathfrak{g}= \Lie(X)/J$, and $U=U\mathfrak{g}$
is the corresponding enveloping algebra. The following conditions hold.
\begin{enumerate}
\item
\label{FiliformWProp21}
 $[N^{\star}]$  is a $K$-basis of $\mathfrak{g}$.
  \item
  \label{FiliformWProp22}
  The enveloping algebra  $U$ is in the class  $\mathfrak{C}(X, W^{\star})= \mathfrak{C}(X, N^{\star})$.
  $U$ is an s.f.p.  Artin-Schelter regular algebra with $\gldim U = d <|N|$.
 \item
 \label{FiliformWProp23}
 Furthermore, assume that $(N^{\star},W^{\star})$ satisfies one of the following conditions:
(a) $N^{\star}= N(\Lcal_{d})$; (b) $N^{\star}= N(\Qcal_{d})$; (c) $N^{\star}= L(m)$, (the set of all Lyndon words of length $\leq m$).
Then
\begin{enumerate}
\item
$[W^{\star}]$ is the Gr\"{o}bner-Shirshov basis of the Lie ideal $J$.
\item
The monomial Lie algebra $\mathfrak{g}=\Lie(X)/ ([W])_{Lie}$ coincides with the standard monomial algebra  $\mathfrak{g}_{W^{\star}} = \Lie(X)/ ([W^{\star}])_{Lie}$
\item
The enveloping algebra
$U= U\mathfrak{g} = K\asX/([W])= K\asX/I$
has defining relations $[W^{\star}]_{ass}$ which form a (finite) Gr\"{o}bner basis of the corresponding ideal $I$ in $K\asX$ .
\end{enumerate}
     \end{enumerate}
\end{proposition}

\begin{proof}
(\ref{FiliformWProp21})
As in Proposition \ref{connected_Prop}
denote by $(\widetilde{N}, \widetilde{W})$  the Lyndon pair corresponding to the enveloping algebra $U=U\mathfrak{g} = K \asX/I$.
It will be enough to prove $(\widetilde{N}, \widetilde{W})= (N^{\star}, W^{\star})$.
By Proposition \ref{connected_Prop}, and by the hypothesis one has
\begin{equation}
\label{atoms_eq1}
\widetilde{N} \subseteq N_{con}= N^{\star}.
\end{equation}
 We shall prove that $ N^{\star} \subseteq \widetilde{N},$ that is every Lyndon word $a\in N^{\star}$ is normal modulo $J$.
We claim that a word $a \in N^{\star}$ can not be affected in the process of resolving a composition of overlap
corresponding to some $2$-chain on $W$. Indeed, consider the set of two-chains on $W$. Denote by $m^{\star} = m(N^{\star})$ the maximum of all lengths of words in $N^{\star}$.
Note that if a word $w \in W$ has length $|w|\leq m^{\star}$, then $w \in W^{\star}$.
Suppose $\omega$ is a two-chain on $W$. Two cases are possible
(i)  $\omega = ut =hv, u, v \in W,$ has length $|\omega|\leq m$, then $\omega$ is composed out of the Lyndon words $u,v \in W$, of lengths $|u|< m^{\star}, |v|< m^{\star}$.
It follows hat $u, v \in W^{\star}$, so $\omega$ is a two-chain on $W^{\star}$. Then the composition of overlap induced by $\omega$
does not affect any of the elements
$a \in N^{\star}$, since $(N^{\star}, W^{\star})$  is a standard Lyndon pair. (ii) $\omega = ut =hv$ has length $|\omega| \geq m+1$.
 Then the composition of overlap, induced by $\omega$ is either $0$, or it is a linear combination of elements $[b_i],$ with the same multi-degree as $\omega$,
 $b_i \in L (\omega), b_i \prec \omega$.
 But $|\omega| \geq m+1$, hence $L (\omega)\bigcap N^{\star} = \emptyset$. It follows that the composition of overlap induced by $\omega$  can not affect any of the Lyndon words in
 $N^{\star}$. We have shown that every Lyndon word in $N^{\star}$ is normal modulo the ideal $([W])_{Lie}= J$. Therefore
 $N^{\star} \subseteq \widetilde{N}$, which together with (\ref{atoms_eq1}) implies the desired equality  $N^{\star} = \widetilde{N}$.
 This implies, that there is an equality of Lyndon pairs
 \begin{equation}
\label{atoms_eq2}
(\widetilde{N}, \widetilde{W})= (N^{\star}, W^{\star}).
\end{equation}
Hence the set $[N^{\star}]$ is a $K$-basis of the Lie algebra
$\mathfrak{g}$.
Part (2) is straightforward. In particular $U\in \mathfrak{C}(X, W^{\star})$.

However, as we have already discussed, in general, the set of obstructions $W^{\star}$ of $U$ does not determine the Gr\"{o}bner-Shirshov basis of the ideal $J= ([W])_{Lie}$,
its elements are not necessarily Lyndon Lie monomials.

Part (\ref{FiliformWProp23}) follows from Corollary  \ref{Obstructionscor1}, and  Corollary
 \ref{Obstructionscor2}. We see that under the restrictive hypothesis of part (\ref{FiliformWProp23}) the set of obstructions $W^{\star}$ determines explicitly the the Gr\"{o}bner-Shirshov basis of the Lie ideal $J= ([W])_{Lie}$, and the relations of the AS regular algebra $U$.
\end{proof}

\subsection{The Fibonacci monomial algebras $F_n$, and the Lie algebras $\mathfrak{g}_n = \Lie(X)/([W(F_n)])$}
\label{Fib_section}
In this subsection $X = \{x<y\}$. In \cite{GIF}, section 7,
we define the sequence of {\it Fibonacci-Lyndon words} $\{\fib_n(x,y)\}$ recursively, as follows:
$\fib_0 := x, \fib_1 := y$ and, then for $n \geq 1$
\begin{equation} \label{FibLigDef}
 \fib_{2n} = \fib_{2n-2}\fib_{2n-1}, \quad \fib_{2n+1} = \fib_{2n} \fib_{2n-1}=\fib_{2n-2}\fib_{2n-1}\fib_{2n-1}.
\end{equation}
This give the sequence

\vskip 3mm
\begin{tabular}{c|c|c|c|c|c|c}
 0 & 1 & 2 & 3 & 4 & 5 & 6 \\
\hline
 $x$ & $y$ & $xy$ & $xyy$ & $xyxyy$ & $xyxyyxyy$ & $xyxyyxyxyyxyy$. \\
\end{tabular}

\vskip 3mm \noindent Note that if we let $a$ be $\fib_2(x,y) = xy$
and $b$ be $ \fib_3(x,y) = xyy$, then the Fibonacci-Lyndon word
$\fib_m(a,b) = \fib_{m+2}(x,y)$.

\begin{facts}
\label{Fib_facts}
\cite{GIF}
The following holds:
\begin{itemize}
\item[a.] The word $\fib_n(x,y)$ is a Lyndon word and its length
is the $n$'th Fibonacci number.
\item[b.] For the lexicographic order we have
\[ \fib_0 < \fib_2 < \cdots < \fib_{2n} < \cdots < \fib_{2n+1} < \cdots < \fib_3
< \fib_1. \]
\end{itemize}
Let $W(F_{\infty})$ consist of all Lyndon words $\fib_{2n-2}\fib_{2n}$ and
 $\fib_{2n+1}\fib_{2n-1}$, where $n \geq 1$.
 A Lyndon word $w$ in $x$ and $y$ is not in the ideal
$(W(F_{\infty}))$ of $K\asX$ if and only if $w$ is a Fibonacci-Lyndon word.

Let $W(F_n)$ be the antichain of all minimal elements in $W(F_{\infty}) \cup \{ \fib_n(x,y) \}$,
with respect to the divisibility order $\sqsubset$.
(This is a finite set, since $\fib_n$ is
a factor of the Fibonacci-Lyndon words later in the sequence.)
The corresponding monomial algebra
\[ F_n = K\asX /(W(F_n))\]
is called
 {\it the  Fibonacci-Lyndon monomial algebra}, \cite{GIF} (or shortly, \emph{the Fibonacci algebra}).
   \end{facts}
The
set of  Lyndon atoms with respect to the ideal $(W(F_n))$ in $K\asX$ is
$N(F_n) = \{\fib_0, \ldots, \fib_{n-1} \}$. We shall call the corresponding Lyndon pair $(N(F_n), W(F_n))$ \emph{the $n$th Fibonacci Lyndon pair}.
The monomial algebra $F_n$
has Hilbert series
$H_{F_n} = \prod_{i =0}^{n-1} \;(1-t^{|f_i|})^{-1},$
where $f_i$ is the $i$'th Fibonacci number.
Clearly,  the global
dimension and the Gelfand-Kirillov dimension of $F_n$ are both $n$.
Proposition 8.1., \cite{GIF} states that
the Fibonacci-Lyndon monomial algebra $F_6$
does not deform to a bigraded Artin-Schelter regular algebra.

\begin{corollary}
\label{Fibcor}
\cite{GIF}
The class $\mathfrak{C} (X, W(F_6))$ does not contain a bigraded Artin-Schelter regular algebra.
\end{corollary}
The next corollary is a particular case of Proposition \ref{FiliformWProp2}, part (\ref{FiliformWProp23}),
and shows that each Fibonacci pair $(N(F_n), W(F_n))$, with $n \geq 4$ induces different defining relations, $[W(F_n)]$, for the corresponding monomial Lie algebra
$\mathfrak{g}_n=Lie(X)/([W(F_n)])_{Lie}$, but the result is isomorphic to the same Lie algebra: $\mathfrak{g}_n=Lie(X)/([W(F_n)])_{Lie} \cong \Lcal_3$, the $4$-dimensional Filiform Lie algebra!

\begin{corollary}
\label{Fibcor_n}
Let $X = \{x< y\}$.
Let $(N(F_n), W(F_n))$ be \emph{the nth Fibonacci Lyndon pair}, $n \geq 2$.
Denote by
$J(F_n)$ the Lie ideal generated by $[W(F_n)]$ in $\Lie(X)$, $\mathfrak{g}_n=Lie(X)/J(F_n)$,
$I_n$ denotes the two sided ideal $I_n = ([W(F_n)]_{ass})$  in  $K
\asX,$
 $U_n= U\mathfrak{g}_n = K \asX/I_n$ is the enveloping algebra of $\mathfrak{g}_n$.
 Then
\begin{enumerate}
\item
The class $\mathfrak{C}(X, W(F_n))= \mathfrak{C}(X, N(F_n))$ contains an AS-regular algebra occurring as an enveloping algebra $U$ of some graded Lie algebra $\mathfrak{g}$ if and only if
$2 \leq n \leq 4$.
\item $N(F_4)= \{x < xy< xyy< y\} = N(\Lcal_{3})$,  $\mathfrak{g}_4=Lie(X)/J_{(F_4)} \cong \Lcal_{3}$,
the standard Filiform Lie algebra of dimension $4$.
\item Let $n \geq 5$, then the connected component of $N(F_n)$ is exactly
$N(\Lcal_{3})$.
There are equalities $J(F_n) = J(F_4)$, $I_n = I_4$, and
\[
\mathfrak{g}_n=Lie(X)/J(F_n) \cong \Lcal_{3}, \quad U_n = U_4 \cong U\Lcal_{3}.
\]
In particular, $U_n \in \mathfrak{C}(X, W(F_4))$, $\gldim U_n = 4$.
\end{enumerate}
\end{corollary}

\section{Two-generated AS-regular algebras of global dimensions $6$ and $7$, 
occurring as enveloping of standard monomial Lie algebras}
\label{class_sec}
In this section we prove one of the main results of the paper, Theorem IV.
\subsection{More properties of Lyndon pairs}
\label{LyndonPairsProp}
\begin{remark}
\label{LyndonPairsProp_Rem}
Suppose $(N, W)$ is a Lyndon pair in the alphabet $X = \{x_1, \cdots, x_n\}$,
$N=  \{l_1<
l_2 < l_3\cdots < l_{d+1} \}$, let $m = \max \{|l|\mid l\in N\}$ . Then $l_1 = x_1, \; l_{d+1}= x_n$, and the following conditions hold
\begin{enumerate}
\item Assume $u \in N\setminus X$ is maximal w.r.t. $\sqsubseteq$, that is $u$ is not a proper subword of any $l \in N$ (this is always so if $|u|= m$).
Then the set $N_0 = N \setminus \{u\}$ satisfies conditions \textbf{C1} and \textbf{C2}. Let $W_0=W(N_0)$ be the corresponding antichain so $(N_0, W_0)$ is a Lyndon pair,
($|N_0| = d$). Let
$A_0 =A_{W_0}$, be the corresponding monomial algebra. Then every algebra $A \in \mathfrak{C}(X, W_0) = \mathfrak{C}(X, N_0)$ satisfies $\gkdim A= d = \gldim A$.
 \item Conversely, let $(N_0, W_0)$ be a Lyndon pair, with $|N_0| = d$. Suppose that $u, v \in N_0, u < v$ are such that $w= uv\in W_0$, or, equivalently,  $w =uv$ is not in
     $N_0$, but every Lyndon sub-word of $uv$ belongs to $N_0$. Then the set $N = N_0\bigcup \{w\}$ satisfies conditions \textbf{C1} and \textbf{C2}. Let $W=W(N)$ be the
     corresponding antichain of Lyndon words, so that $(N, W)$ is a Lyndon pair, one has $|N|= d+1$, and every algebra $A \in \mathfrak{C}(X, W) = \mathfrak{C}(X, N)$
     satisfies $\gkdim A= d+1 = \gldim A$.
\end{enumerate}
\end{remark}
\begin{problem}
\label{problem1}
\textbf{\emph{Given}} a full list $\mathfrak{P}_d$ of all Lyndon pairs $(N, W)$ with  $|N|= d$, (up to isomorphism of monomial algebras $A_W$).
\textbf{\emph{Find}} the list $\mathfrak{P}_{d+1}$ of all Lyndon pairs $(N, W)$ with  $|N|= d+1$, (up to isomorphism of monomial algebras $A_W$).
\end{problem}
Remark \ref{LyndonPairsProp_Rem} gives a base for the strategy, conditions (2) and (1) can be and will be used for inductive classification of all Lyndon pairs $(N, W)$ in $X^{+}$, with a fixed order $|N|= d+1$.
The interested reader may find a procedure for an effective solution of Problem \ref{problem1} in \cite{joveski}.

Till the end of the section we assume $X = \{x <y\}$.
Applying directly our strategy, we have found the sets $\mathfrak{P}_6$ and $\mathfrak{P}_7$ independently of \cite{joveski}, and without using computer programs.
In Subsections \ref{list6}, \ref{list7} we give lists of Lyndon pairs $(N, W)$ in $X^{+}$, where $X =\{x < y\}$ and $|N|= 6,7$.
  The enumeration (6.m.j), or (7.m.j) stands for a Lyndon pair $(N_j, W_j)$, where $m= m(N_j)$ is the maximal length of all word $l \in N_j$.
  In each case when $N$ is not connected we identify the connected component $N_{con}=: N^{\star}$, and we prove that (in these cases) $(N^{\star}, W^{\star})$ is a standard Lyndon pair, corresponding to an explicitly given standard monomial Lie algebra $\mathfrak{g}_{W^{\star}}$. Note that, in general, for large values of $d$ this is not true.

\subsection{A complete list of Lyndon pairs $(N,W)$  where $|N|=6$}
\label{list6}
\[
\begin{array}{llll}
(6.4.1) & N_{1}&= \{x < x^2y  < xy < xy^2< xy^3 < y\},&   \\
        & W_{1}&= \{x^3y,\;x^2y^2\;x^2yxy,  xyxy^2, \;xy^2xy^3,\; xy^4\}& \\
(6.4.2) & N_{2}&= \{x < x^2y < x^2y^2 < xy < xy^2 < y\}, & \\
        & W_{2}&= \{x^3y,\; xy^3, \;x^2yxy, \; xyxy^2,\; x^2yx^2y^2, \;x^2y^2xy\}&\\
(6.5.3) & N_{3}&=  \{x < xy < xy^2 < xy^3 < xy^4 < y\}= N(\Lcal_5)& \\
        & W_{3}&=  \{xy^ixy^{i+1},\; 0 \leq i \leq 3,  \;xy^5\}& \\
(6.5.4) & N_{4}&=  \{x < xy  <  xyxy^2 < xy^2< xy^3 < y \}=N(\Qcal_5) & \\
        & W_{4}&=  \{x^2y, \;  xyxyxy^2,\; xyxy^2xy^2,\;xy^2xy^3, \;xy^4 \}&\\
(6.5.5) & N_{5}&=  \{x < xxy  < xy<  xyxy^2 < xyy< y \}, & N^{\star} =N(\mathfrak{L}(3))\\
 (6.7.6)& N_{6}&=  \{x < xy  < xy^2 < xy^2xy^3 <xy^3< y \}, & N^{\star} =N(\Lcal_4))\\
 (6.7.7) & N_{7}&=  \{x < xy  < xyxyxy^2 < xyxy^2x <xy^2< y \}, & N^{\star} =N(\Lcal_3))\\
 (6.8.8) & N(F_6)&= \{x, y, xy, xyy, xyxyy, xyxyyxyy\},&\;N^{\star} =N(\Lcal_3))\\
\end{array}
\]
$N^{\star}:= N_{con}$ is the connected component of $N$.

\subsection{A complete list of Lyndon pairs $(N,W)$  where $|N|=7$}
\label{list7}
\[
\begin{array}{lll}
(7.4.1)&N_1= \{x < x^3y < x^2y < x^2y^2 < xy < xy^2 < y\},  &\\
       &W_1= \{x^4y,\;x^3y^2,\; x^3yx^2y,\; x^3yxy,\;x^2yxy,\;xyxy^2,\;x^2y^2xy,\;xy^3 \}&\\
(7.4.2)&N_2=\{x < x^3y < x^2y < xy < xy^2 < xy^3 < y\}, &\\
       &W_2=\{x^4y,\;x^2y^2,\; x^2y^2,\; x^3yx^2y,\; x^2yxy,\; xyxy^2,\;xy^2xy^3,\; xy^4\}&\\
(7.4.3)&N_3=\{ x < x^2y < x^2y^2 < xy < xy^2 < xy^3 < y \}, &\\
       &W_3=\{x^3y,\;x^2y^3,\;x^2yxy,\;  xyxy^2,\;x^2y^2xy,\; x^2yx^2y^2, \;xy^2xy^3, xy^4\}& \\
(7.5.4)&N_4=\{x < xy < xyxy^2 < xy^2 < xy^3 < xy^4 < y\}, &\\
       &W_4=\{x^2y,\;xyx^yxy^2,\; x^yxy^2xy^2,\;x^yxy^3,\; xy^2xy^3,\;xy^3xy^4, \;xy^5\} &\\
(7.5.5)&N_5=\{x < x^2y < xy < xy^2 < xy^3< xy^4 < y\}, &\\
       &W_5=\{x^3y,\; x^2y^2,\; x^2yxy, \;  xyxy^2,\; xy^2xy^3,\; xy^3xy^4,\; xy^4\}&\\
(7.5.6)&N_6=\{ x < x^2y  < xy < xyxyy < xy^2 < xy^3 < y\}, &\\
        & W_6=\{ x^3y,\; x^2y^2, \; xxyxy,\;xyxyxyy,\;xyxy^3, \; xyxy^2xy^2,\;xy^2xy^3,\; xy^4\}&\\
(7.5.7)&N_7=\{x < x^2y < x^2yxy < x^2y^2 < xy < xy^2 < y\}, &\\
       &W_7=\{x^3y, x^2yx^2yxy, x^2yxyxy, x^2yxyx^2y^2, x^2yx^2y^2, xyxy^2,&\\
       &\quad \quad \quad x^2y^2xy, xy^3\}&\\
(7.5.8)&N_8=\{x < x^2y < x^2y^2  < xy < xyxy^2 < xy^2< y \},&\\
        &W_8=\{x^3y,\; x^2yxy, \;x^2yx^2y^2,\;x^2y^2xy,\;xyxyxy^2,\;xyxy^2xy^2,\; xy^3\},&\\
(7.6.9)&N_{9}=\{x < xy < xy^2 < xy^3 < xy^4 < xy^5 < y\},\;&\\
       & W_{9}=\{ xy^{i}xy^{i+1},  0 \leq i \leq 4\}\bigcup \{xy^6\}&\\
(7.5.10)&N =\{ x < x^2y < x^2yxy< xy  < xy^2 < xy^3 < y\},\; &\\
         &W =\{ x^3y, \;x^2y^2,\;x^2yx^2yxy,\; x^2yxyxy,\; xyx^2,\;xy^2xy^3,\;xy^4\}&\\
 (7.5.11)&N =\{ x < x^3y <x^2y<xy< xyxy^2< xy^2 < y\},\; &\\
         &W =\{ x^4y, \;x^2y^2,\;xy^3,\;\cdots\}&\\
(7.6.12) &N=\{x < xy < xyxy^2 < xyxy^3 < xy^2 < xy^3 < y\},\; &\\
           &W=\{ xxy, \;xy^4,\; \cdots  \}&\\
(7.5.13)&N=\{x < x^2y < x^2yxy <xy < xyxy^2 < xy^2 < y\},\;&N^{\star} =N(\mathfrak{L}(3))\\
(7.6.14)&N=\{x < x^2y < x^2y^2 < x^2y^2xy < xy < xy^2 < y\}, &N^{\star} =N_2\\
        &N^{\star}=\{x < x^2y < x^2y^2  < xy < xy^2 < y\}= N_2,&\text{see }(6.4.2)\\
(7.7.15)&N=\{x < xy < xy^2 <xy^2xy^3< xy^3 <xy^4 < y\},  m=7, &N^{\star} =N(\Lcal_5)\\
(7.7.16)&N=\{x < xy < xyxy^2 < xy^2< (xy^2)(xy^3) <xy^3 < y\}, &N^{\star} =N(\Qcal_5)\\
(7.7.17)&N=\{x < x^2y < xy< xy^2< (xy^2)(xy^3) < xy^3 < y\},  &N^{\star} =N_1\\
        &N^{\star}=\{x < x^2y < xy < xy^2 < xy^3< y\}= N_1,\;&\text{see }(6.4.1)\\
(7.7.18)&N=\{x < xy < (xy)(xyxy^2) < xyxy^2 < xy^2< xy^3 < y\}, &N^{\star} =N(\Qcal_5)\\
(7.7.19)&N=\{x < x^2y < (x^2y)(x^2y^2)< x^2y^2 < xy < xy^2 < y\},&N^{\star} =N_2\\
        &N^{\star}=\{x < x^2y < x^2y^2  < xy < xy^2 < y\}= N_2, &\text{see }(6.4.2)\\
(7.7.20)&N=\{x < x^2y < xy< xy(xyxy^2)<  xyxy^2< xy^2 < y\},  &N^{\star} =N(\mathfrak{L}(3))\\
(7.8.21)&N=\{x < x^2y < xy< xyxy^2<  (xyxy^2)(xy^2)< xy^2 < y\},&N^{\star} =N(\mathfrak{L}(3))\\
(7.8.22)&N=\{x < xy< xyxy^2 <(xyxy^2)(xy^2)< xy^2< xy^3 < y\}, &N^{\star} =N(\Qcal_5)\\
(7.8.23)&N=\{x < xy< (xy)(xyxy^2) < xyxy^2 <(xyxy^2)(xy^2)< xy^2 <y\}, &N^{\star} =N(\Lcal_3)\\
(7.9.24)&N=\{x < xy<  xy^2< xy^3< (xy^3)(xy^4)< xy^4<y\},   &N^{\star} =N(\Lcal_5)\\
(7.9.25)&N=\{x < xy< (xy)(xyxyxy^2) <  (xy)(xyxy^2)< xyxy^2 < xy^2 <y\},   &N^{\star} =N(\Lcal_3)\\
(7.10.26)&N=\{x < xy <xy^2 < (xy^2)(xy^2xy^3) < xy^2xy^3< xy^3 <y\},  &N^{\star} =N(\Lcal_4)\\
(7.11.27)&N=\{x < xy <xy^2 < xy^2xy^3< (xy^2xy^3)(xy^3)< xy^3< y\},  &N^{\star} =N(\Lcal_4)\\
(7.11.28)&N=\{x, xy,  xyxy^2,   (xyxy^2)(xy^2), (xyxy^2xy^2)(xy^2), xy^2< y\},  &N^{\star}=N(\Lcal_3)\\
(7.12.29)&N=\{x, xy, xy(xyxy^2), (xyxyxy^2)(xyxy^2), xyxy^2, xy^2< y\},&N^{\star} =N(\Lcal_3)\\
(7.13.30)&N=N(F_7)=\{x,  y, xy, xyy, xyxyy, xyxyyxyy, (xyxyy)(xyxyyxyy) \},&N^{\star} =N(\Lcal_3).
 \end{array}
\]
The enumeration (7.m.j) indicates that $m = m(N_j) := \max \{|l|\mid l \in N_j\}$.  The connected component of $N$ is $N^{\star}:= N_{con}$,
 $\mathfrak{L}(3)$ denotes the free nilpotent algebra of nilpotency class $3$, and $\Lcal_d$ is the Filiform Lie algebra of dimension $d+1$.
\begin{theorem4}
\label{theorem4}
Let $X = \{x < y\}$, and $(N, W)$ be the Lyndon pairs in $X^{+}$ given in subsections \ref{list6} and \ref{list7}. For each such pair
$J=([W])_{Lie}$ is the Lie ideal of $\Lie(X)$ generated by $[W]$, $\mathfrak{g}=\Lie(X)/J$,
$I$ denotes the two sided ideal $I = ([W]_{ass})$  in  $K
\asX,$
 so $U= U\mathfrak{g} = K \asX/I$ is the enveloping algebra of $\mathfrak{g}$.
\begin{enumerate}
\item[I]  Suppose $|N| = 6$. Then
\begin{enumerate}
\item There are eight Lyndon pairs $(N_i, W_i), 1 \leq j \leq 8$, up to isomorphism of monomial algebras $A_W$, see the list \ref{list6}. Suppose $(N_i, W_i)$
is from this list.
\item
The algebra $U_i = K \asX/([W_i]_{ass}) $
is an Artin-Schelter regular algebra of global dimension $6$ if and only if $1 \leq i \leq 4$. In this case $U_i \in \mathfrak{C}(X,W_i)$.
\item
The class $\mathfrak{C}(X, N(F_6)) =  \mathfrak{C}(X,W_8)$ does not contain any bigraded Artin-Schelter regular algebra.
\item The class $\mathfrak{C}(X,W_i)$, $5 \leq i \leq 7$ does not contain any AS regular algebra $U$ presented as an enveloping $U = U\mathfrak{g}$ of a graded
Lie algebra.
\item
For $5 \leq i \leq 8$ the algebra $U_i$
is an AS regular algebra with
$\gldim U_i \leq 5$.
\end{enumerate}
\item[II]  Suppose $|N| = 7$. Then
\begin{enumerate}
\item There are $30$ Lyndon pairs $(N_i, W_i), 1 \leq i \leq 30$, up to isomorphism of monomial algebras $A_W$, see the list \ref{list7}.  Suppose $(N_i, W_i)$ is a pair from this list.
 \item
 The algebra $U_i = K \asX/([W_i]_{ass})$
is an Artin-Schelter regular algebra of global dimension $7$ if and only if $1 \leq i \leq 9$. In this case $U_i \in \mathfrak{C}(X,W_i)$.
\item The class $\mathfrak{C}(X,W_i)$, $13 \leq i \leq 30$, does not contain any AS regular algebra $U$ presented as an enveloping $U = U\mathfrak{g}$ of a graded
Lie algebra.
\item
For $10 \leq i \leq 30$, $U_i$ is a finitely presented AS regular algebra of global dimension
$\gldim U_i \leq 6$. These algebras are identified in Corollary \ref{thm4cor}, and Remark \ref{thm4rem}.
\end{enumerate}
\end{enumerate}
\end{theorem4}
\begin{proof}
I.
(b)  Consider the Lyndon pairs $(N_j, W_j), 1 \leq j\leq 4$ given in subsection \ref{list6}.
We claim that each of the sets $[W_j]$ is a Gr\"{o}bner-Shirshov basis of the Lie ideal $J_j$.
Note that for $j = 1,2$ the number $m_j = \max \{|u|\mid u \in N_j\}$ is $m_j = 4$, and the minimal length of a $2$-chain
is $c_j = 6 > m_j$.
Indeed, $\omega = (xyxyy)y = xy(xyyy)$ is the only $2$-chain on $W_j$ with length $|\omega|=6$. All remaining $2$-chains have lengths $>6$.
It follows from Proposition \ref{SimplifyProp} (3.a), that all compositions of overlap are solvable,
hence $[W_j]$ is a Gr\"{o}bner-Shirshov basis of the Lie ideal $J_j = ([W_j])_{Lie}$.
For $j = 3,4$  we use the results from  subsections \ref{MonFiliformSec0} and \ref{MonFiliformSec}.
The Lyndon pair
$(N_3, W_3)$ determines $\mathfrak{g}_3 \simeq \Lcal_5$, the standard 6-dimensional Filiform Lie algebra of nilpotency class $5$, while the pair $(N_4, W_4)$ determines $\mathfrak{g}_4 \simeq \Qcal_5$,
the second type Filiform algebra of dimension 6.
   In particular, each of the sets $[W_i], i= 3, 4$ is a Gr\"{o}bner-Shirshov basis of the Lie ideal $J_i$, and the set $[N_i]$ forms a $K$-basis
    of $\mathfrak{g}_i$, $i=3,4$. It follows that  $U_i = U\mathfrak{g}_i \in \mathfrak{C} (X, W_i)$ for all $1\leq i \leq 4$. Moreover, $U_i$ is an Artin-Schelter reqular algebra of
    $\gldim U_i = 6$. Part (d) implies that for  $5\leq i \leq 8$, $[W]_i$ is not a Grobner-Shirshov basis, and $\gldim U_i \leq 5$. This completes the proof of  (b).

(c) By Corollary \ref{Fibcor}
the class $\mathfrak{C} (X, W(F_6))$ does not contain any bigraded AS-regular algebra.

(d) Each of the sets $N_i$, $5 \leq i \leq 8$ is not connected, hence by Corollary \ref{NdisconnCor},  the class $\mathfrak{C}(X, W_i)$
does not contain the enveloping algebra $U=U\mathfrak{g}$ of any graded Lie algebra $\mathfrak{g}$.
In particular,  $[W_i]$
is not a Lyndon-Shirshov basis of the Lie ideal $J_i=([W_i])_{Lie}$ of $\Lie(X)$.
For  $i, 5 \leq i\leq 8$ $(N_i, W_i)$ is a nontrivial disconnected extension of its connected components $N_i^{\star}= (N_i)_{con}$, which is identified explicitly in the list. We use Proposition
  \ref{FiliformWProp2} to
identify the Lie algebra
 $\mathfrak{g}_i=Lie(X)/[W_i])_{Lie}$, $5 \leq i \leq 8$, and $U\mathfrak{g}_i$.
In the case (6.5.5)  the connected component $N_5^{\star} = L(3)$ -the set  of all Lyndon words of length $\leq 3$.
 Proposition \ref{FiliformWProp2} implies that $W_5^{\star}= W(N_5^{\star})= W(L(3))$,
 satisfies $([W_5^{\star}])_{Lie}= ([W_5])_{Lie}= J_5$.
 The set $[W_5^{\star}]$ is a Gr\"{o}bner-Shirshov basis of the Lie ideal $J_5$. Moreover, $\mathfrak{g}_5=Lie(X)/J_5 \cong \mathfrak{L}(3)$, the free nilpotent Lie algebra of nilpotency class $3$. Its enveloping algebra $U_5 \cong U\mathfrak{L}(3) \in \mathfrak{C}(X, W(L_3) )$ is an Artin-Schelter regular algebra of $\gldim U_5= 5$.
 In case (6.5.6) $N_6^{\star} = N(\Lcal_4)$, so $(N_6, W_6)$ is a nontrivial disconnected extension of the standard Lyndon pair
 $(N(\Lcal_4), W(\Lcal_4))$, corresponding to the Filiform Lie algebra $\Lcal_4$. Proposition \ref{FiliformWProp2} again implies that $J_6 = ([W(\Lcal_4))])_{Lie}$, hence
 $\mathfrak{g}_6 \cong \Lcal_4$, and $U_6=U\mathfrak{g}_6 \cong U\Lcal_4$, $U_6 \in \mathfrak{C}(X, W(\Lcal_4))$ is an Artin-Schelter regular
 algebra of  $\gldim U_6 = 5.$
 In each of  the cases (6.7.7) and (6.8.8) $\mathfrak{g}_i \cong \Lcal_3$,  so $U = U\mathfrak{g}_i \in \mathfrak{C}(X, W(\Lcal_3))$
is an AS regular algebra with $\gldim U_i =4$, $i = 7, 8$.  Part I of the theorem has been proved.

 II.
(b) We shall prove that for $1 \leq i \leq 9$ the Lie algebra $[W]_i$ is a Gr\"{o}bner-Shirshov basis of the Lie ideal $J_i$.
We split this in several cases.

(b.1). The cases (7.4.i), $1 \leq i \leq 3$.  For each of the Lyndon pairs $(N_i, W_i), $ the maximal length of a Lyndon atom is $m= 4$, at the same time
  the minimal length of a two-chain on $W_i$ is $6$, so Proposition \ref{SimplifyProp} implies
that all compositions of overlap are solvable, and $[W_i]$ is a GS-basis of the Lie ideal $J_i = ([W_i])_{Lie}$.
Therefore
$U_i \in \mathfrak{C}(X,W_i)$ is an AS regular algebra of $\gldim U_i = 7$

(b.2) The Lyndon pairs $(N_4, W_4)$, $(N_7, W_7)$, $(N_8, W_8)$ are similar. For $i= 4, 7, 8,$ the maximal length of Lyndon atoms in $N_i$ is $m_i = 5$, and the minimal length
of
a two-chain on $W_i$ is $c_i = 6$. Proposition \ref{SimplifyProp} implies again that
  $[W_i]$ is a Gr\"{o}bner-Shirshov basis of the ideal $J_i$, and $U_i \in \mathfrak{C}(X,W_i)$ is an AS regular algebra of $\gldim U_i = 7$.

(b.3) The Lyndon pair $(N_9, W_9)$ corresponds to the standard filiform Lie algebra $\mathcal{L}_6$, in particular $[W_9]$ is a Gr\"{o}bner-Shirshov basis of the Lie ideal
 $J_9$.

(b.4) It remains to consider only two cases: (7.5.5) and (7.5.6.)
We have
\[
  \begin{array}{ll}
(7.5.5) &N_5=\{x < x^2y < xy < xy^2 < xy^3< xy^4 < y\}, \; m(N_5)=5\\
        &W_5=\{x^3y,\; x^2y^2,\; xxyxy, \;  xyxy^2,\; xy^2xy^3,\; xy^3xy^4,\; xy^4\}
             \end{array}
             \]

  \[   \begin{array}{ll}
    (7.5.6) &N_6=\{ x < xxy  < xy < xyxyy < xy^2 < xy^3 < y\}, \; m(N_6)=5\\
        &W_6=\{ x^3y,\; x^2y^2, \; xxyxy\;xyxyxyy,\;xyxy^3 \; xyxy^2xy^2\;xy^2xy^3,\; xy^4\}
             \end{array}
        \]

         For $i=5, 6$ there is unique 2-chain of length $5$, $\omega = (xxxy)y = x(xxyy)$, on  $W_i$.
           The 2-chain $\omega$ implies the following composition of overlap
         \begin{equation}
         \label{xxxyy_eq}
         \begin{array}{ll}
         (xxxy, xxyy)_{\omega}&= [[xxxy],y]- [x,[xxyy]] = [[x,[xxy]],y]- [x,[xxyy]]\\
                              &= [[[xy],[xxy]]+ [x,[[xxy],y]]] - [x,[xxyy]]\\
                              & = -[[xxy],[xy]]+ [[xxxyy] - [xxxyy]]\\
                              &= -[[xxy],[xy]]= -[xxyxy] \in Span [W_i], \; i = 5, 6.
         \end{array}
         \end{equation}
Hence this composition is solvable. All remaining $2$-chains on $W_i$ have length $\geq 6 > m(N_i)$, i = 5,6, hence the corresponding compositions are solvable. It follows that
          $[W_i]$ is a GS basis of the Lie ideal $J_i$.

The cases (7.x.i), where $13 \leq i \leq 30$. Then each of the pairs $(N,W)$ is a nontrivial disconnected extension of a standard Lyndon pair $(N^{\star},W^{\star})$, therefore
by Corollary \ref{NdisconnCor}, $[W_i]$ is not a Lyndon-Shirshov basis of the Lie ideal $J_i=([W_i])_{Lie}$ of $\Lie(X)$.
 Moreover, the class $\mathfrak{C}(X, W_i)$, $13 \leq i \leq 30$
does not contain the enveloping algebra $U=U\mathfrak{g}$ of any graded Lie algebra $\mathfrak{g}$.
Explicit details for  the Lyndon-Shirshov basis of the Lie ideals $J_i$, the corresponding monomial algebras $\mathfrak{g}$ and their enveloping $U$ in these cases are given in Corollary \ref{thm4cor}.

 It remains to study the cases (7.5.10), (7.5.11), (7.6.12), where the Lyndon pair $(N,W)$ has connected set $N$, but $[W]$ is not a Gr\"{o}bner-Shirshov basis of the Lie ideal $J$.
  Detailed information for each of these cases is given in  Remark \ref{thm4rem}.
  \end{proof}
 \begin{remark}
 \label{thm4rem}
 In notation and assumptions of Theorem IV.
 \begin{enumerate}
 \item
 The Lyndon pair $(N_{10}, W_{10})$ given in (7.5.10) defines a Lie algebra $\mathfrak{g}$ whose enveloping $U= U\mathfrak{g}$
 has Lyndon atoms $\widetilde{N}$ and relations $[\widetilde{W}]$, where
\begin{equation}
\label{N10eq}
\begin{array}{l}
\widetilde{N} =\{ x < xxy < xy  < xy^2 < xy^3 < y\} \\
\widetilde{W} =\{ x^3y, \; x^2y^2,\;x^2xy,\;\; xyx^2,\;xy^2xy^3,\;xy^4\}.
\end{array}
 \end{equation}
This is exactly the Lyndon pair $(N_1, W_1)$ given by (6.4.1). Moreover, the GS- basis of the Lie ideal $J_{10}$ is
$[W_1]$,
so $\mathfrak{g}= \mathfrak{g}_{W_1}$
has basis $[N_1]$,
and is $6$ dimensional,  $U_{10}\in \mathfrak{C}(X, W_1)$ is an AS regular algebra with $\gldim U_{10} = 6$.

\item
The Lyndon pair $(N_{11}, W_{11})$ given in (7.5.11) defines a Lie algebra $\mathfrak{g}$ whose enveloping $U= U\mathfrak{g}$
 has Lyndon atoms $\widetilde{N}$ and relations $[\widetilde{W}]$, where
 \[
 \begin{array}{l}
\widetilde{N} =\{ x < x^3y < x^2y < xy  < xy^2< y\} \\
\widetilde{W} =\{ x^4y, \; x^2y^2,\;x^3yx^2y, \;x^2yxy,\;\; xyx^2,\;xy^3\}.
\end{array}
 \]
 The monomial algebra $A_{\widetilde{W}}$ is isomorphic to $A_{W_1}$, corresponding to case (6.4.1).
 $U \in \mathfrak{C}(X,\widetilde{W})$ is an AS-regular algebra of $\gldim U = 6$.
 \item
For the Lyndon pair $(N_{12}, W_{12})$ given in (7.6.12) one has
$\mathfrak{g}= \Lie /([ W_{12}]) \simeq \Qcal_5$, the second type filiform Lie algebra of dimension $6$, $U = U\Qcal_5\in \mathfrak{C}(X,W(\Qcal_5))$
is an AS regular algebra with $\gldim U = 6$.
\end{enumerate}
\end{remark}
 \begin{proof}
We give a proof of (1).
 Consider the pair  $(N_{10}, W_{10})$, given in
 (7.5.10). For simplicity we set $N=N_{10}, W = W_{10}, J= J_{10}, I = I_{10}$.
         The two-chain $\omega = xxxyy= (xxxy)y = x(xxyy)$ on $W$ gives a composition of overlap, $(xxxy, xxyy)_{\omega}$, as in (\ref{xxxyy_eq}), so the same computations imply that $(xxxy, xxyy)_{\omega}= -[xxyxy]\in  J$.
          By assumption $xxyxy\in N,$ thus $[xxyxy] \in [N]\bigcap J$, and therefore
 the composition $(xxxy, xxyy)_{\omega}$ is not solvable. It follows that $[W_{10}]$ is not a Gr\"{o}bner Shirshov basis of the Lie ideal $J_{10}$.
Consider $N_0 = N\setminus \{xxyxy\}$. It is easy to show
that the (actual) Lyndon pair  corresponding to the ideal $I_{10}$ is the pair $(\widetilde{N}, \widetilde{W})$ given in (\ref{N10eq}).
 But this is exactly the Lyndon pair (6.4.1) $(N_1, W_1)$.
 Hence the the Gr\"{o}bner-Shirshov basis of  $J_{10}$ is $[\widetilde{W}]=[W_1]$, see (6.4.1),
$\mathfrak{g}= \mathfrak{g}_{W^1}$ is $6$ dimensional and $U_{10}$ is an AS regular algebra of global dimension $6$, $U_{10} \in \mathfrak{C}(X, \widetilde{W})$.

Parts (2) and (3) are verified by similar arguments. In the case (2), (7.5.11) one shows that $(xxyy, xyyy)_{\omega}= [xy,xyy]\in J_{11}\bigcap N_{11}$;
 and in the case (3) (7.6.12) the composition $(xxy, xyyyy)_{\omega}=[xyxyyy]\in J_{12}\bigcap N_{12}$. Analogous argument as in (1) completes the proof.
\end{proof}

\begin{openquestion}
\label{openquestion1}
Determine which of the three non-standard Lyndon pairs $(N_{10}, W_{10})$   $(N_{11}, W_{11})$ and $(N_{12}, W_{12})$, see (7.5.10), (7.5.11), (7.6.12), respectively,  is regular. In other words which of
the classes $\mathfrak{C} (X, W_i)$ with ($10 \leq i \leq 12$ contains the enveloping $U =
     U\mathfrak{g}$ of some graded Lie algebra $\mathfrak{g}$. In this case (i) $U$ is an s.f.p. AS-regular algebra of global dimension $d= |N|=7,$ and the set $[N_i]$ is a $K$-basis of  $\mathfrak{g}$.

More generally, it is a nontrivial question to decide whether given a non-standard Lyndon pair $(N, W)$, with connected $N$, the pair is regular.
\end{openquestion}

\begin{corollary}
\label{thm4cor}
Notation and assumptions as in Theorem IV, part II.
Let $(N, W)$ be a Lyndon pair listed in subsection \ref{list7}, denote its connected component by $N^{\star} =N_{con}$, $(N^{\star}, W^{\star})$ is the corresponding Lyndon pair.
\begin{enumerate}
\item
\label{thm4cor1}
In each of the cases (7.8.23),  (7.9.25),   (7.11.28), (7.12.29), (7.13.30)
\[\begin{array}{ll}
N^{\star}&=N(\Lcal_{3})= \{x < xy <  xyy < y\},\quad
W^{\star}=W(\Lcal_{3})= \{xxy, xyxy^{2}, xy^{3}\}.\\
J&= ([W])_{Lie} = ([W(\Lcal_{3})])_{Lie}, \quad
\mathfrak{g}= \Lie(X)/J \simeq \Lcal_{3}, \\
U &= U\mathfrak{g}\in \mathfrak{C}(X, W(\Lcal_{3})),\; \text{is an AS regular algebra with}\; \gldim U = 4.
   \end{array}
          \]
\item
\label{thm4cor2}
In each of the cases (7.7.20), (7.8.21)
\[\begin{array}{ll}
N^{\star} &= \{x < x^2y <  xy < xy^2< y\}= L(3)= N(\mathfrak{L}(3)), \; W^{\star} = W(3)=W(\mathfrak{L}(3)). \\
   J&= ([W])_{Lie} = ([W(3)])_{Lie}=J_{3}, \quad
\mathfrak{g}= \Lie(X)/J \simeq \mathfrak{L}(3), \\
U &= U\mathfrak{g}\in \mathfrak{C}(X, W(3)),\; \text{is an AS regular algebra with}\; \gldim U = 5.
   \end{array}
          \]
\item
\label{thm4cor3}
In each of the cases    (7.10.26), (7.11.27)
\[\begin{array}{ll}
N^{\star}&=N(\Lcal_{4})= \{x < xy <  xyy < xyyy< y\},\\
W^{\star}&=W(\Lcal_{4})= \{xy^kxy^{k+1}\mid 0 \leq k \leq 2\}\bigcup\{xy^4\}.\\
   J&= ([W])_{Lie} = ([W(\Lcal_{4})])_{Lie}, \quad
\mathfrak{g}= \Lie(X)/J \simeq \Lcal_{4}, \\
U &= U\mathfrak{g}\in \mathfrak{C}(X, W(\Lcal_{4})),\; \text{is an AS regular algebra with}\; \gldim U = 5.
   \end{array}
          \]
      \item
      \label{thm4cor4}
       In each of the cases (7.7.15), (7.9.24)
\[
 \begin{array}{ll}
N^{\star}&= N(\Lcal_{5})=\{x < xy <  xyy < xyyy< xyyyy< y\},\\
 W^{\star}&=W(\Lcal_{5})= \{xy^kxy^{k+1}\mid 0 \leq k \leq 3\}\bigcup\{xy^5\}.\\
   J&= ([W])_{Lie} = ([W(\Lcal_{5})])_{Lie}, \quad
\mathfrak{g}= \Lie(X)/J \simeq \Lcal_{5}, \\
U &= U\mathfrak{g}\in \mathfrak{C}(X, W(\Lcal_{5}))\; \text{is an AS regular algebra with}\; \gldim U = 6.
   \end{array}
   \]

\item
\label{thm4cor5}
In each of the cases  (7.7.16),  (7.7.18),  (7.8.22)
\[\begin{array}{ll}
N^{\star}&=N(\Qcal_{5})= \{x < xy <  xyxyy <  xyy < xyyy< y\},\\
W^{\star}&=W(\Qcal_{5})= \{xxy, (xy)^2xyy, xy(xyy)^2,  xy^2xy^3, xy^4\}.\\
   J&= ([W])_{Lie} = ([W(\Qcal_{5})])_{Lie}, \quad
\mathfrak{g}= \Lie(X)/J \simeq \Qcal_{5}, \\
U &= U\mathfrak{g}\in \mathfrak{C}(X, W(\Lcal_{4})),\; \text{is an AS regular algebra with}\; \gldim U = 6.
   \end{array}
          \]
\item
\label{thm4cor6}
Each of the cases  (7.6.14),  (7.7.19)   "degenerates" to (6.4.2), more precisely
\[\begin{array}{ll}
N^{\star}&= \{x < x^2y < x^2y^2 < xy < xy^2 < y\} = N_2, \; \text{see case (6.4.2)} \\
W^{\star}&=  \{x^3y,\; xy^3,\;x^2yxy, \; xyxy^2,\; x^2yx^2y^2, \;x^2y^2xy\}=W_{2},\text{see case (6.4.2)} \\
   J&= ([W])_{Lie} = ([W_2])_{Lie}, \quad\text{has a G-S basis the set }\;[W_2]\\
\mathfrak{g}&= \Lie(X)/J \simeq \Lie(X)/([W_2])_{Lie} \quad\text{has a $K$-basis}\; [N^{\star}] \\
U &= U\mathfrak{g}\in \mathfrak{C}(X, W_2),\; \text{is an AS regular algebra with}\; \gldim U = 6.
   \end{array}
          \]
\item
\label{thm4cor7}
The case (7.7.17) "degenerates" to (6.4.1), more precisely
\[\begin{array}{ll}
N^{\star}&=\{x < x^2y  < xy < xy^2< xy^3 < y\} = N_1,\; \text{see (6.4.1)} \\
W^{\star}&= \{x^3y,\;x^2y^2\;x^2yxy,  xyxy^2, \;xy^2xy^3\; xy^4\}= W_{1}, \text{see  (6.4.1)}\\
   J&= ([W])_{Lie} = ([W_1])_{Lie}, \quad\text{has a G-S basis the set }\;[W_1]\\
\mathfrak{g}&= \Lie(X)/J \simeq \Lie(X)/([W_1])_{Lie} \quad\text{has a $K$-basis}\; [N^{\star}] \\
U &= U\mathfrak{g}\in \mathfrak{C}(X, W_1),\; \text{is an AS regular algebra with}\; \gldim U = 6.
   \end{array}
          \]
           \end{enumerate}
 \end{corollary}
\begin{proof}
Recall that by Proposition \ref{FiliformWProp2}  a nontrivial disconnected extension  $(N,W)$  of a standard Lyndon pair $(N^{\star},W^{\star})$, (here $|N^{\star}|= |N_{con}|< |N|$) does not "contribute" new classes of Lie algebras, whenever $N^{\star}$ is one of the following
(a) $N^{\star}= N(\Lcal_{d})$; (b) $N^{\star}= N(\Qcal_{d})$; (c) $N^{\star}= L(m)$, the set of all Lyndon words of length $\leq m$.
Clearly,  parts (\ref{thm4cor1}) through (\ref{thm4cor5}) follow straightforwardly from   Proposition \ref{FiliformWProp2}.
To prove parts (\ref{thm4cor6}) and (\ref{thm4cor7}) one uses first Proposition \ref{FiliformWProp2} to deduce that the correct obstruction set for $U\mathfrak{g}$ is
exactly $W^{\star}$. Note that  every word $w \in W^{\star}$ is minimal (w.r.t.$\prec$) in its multi-degree component $L_{\alpha}$, where $\deg (w) =\alpha$, therefore
$[w]$ is the unique Lie element in $K[L_{\alpha}],$ whose highest monomial is $w$. This implies that the monomial Lie algebra $\mathfrak{g}= \Lie(X)/([W])_{Lie}$ is isomorphic to the standard monomial Lie algebra $\mathfrak{g}_{W^{\star}}= \Lie(X)/([W^{\star}])_{Lie}$ occurring in the list \ref{list6}.
          \end{proof}
\begin{remark}
We have seen that
each
non-standard Lyndon pair $(N, W)$ with $|N|=7$ "degenerates" to a standard Lyndon pair $(N^{\star}, W^{\star})$ with
$|N^{\star}| < 7$.
 Moreover, in these cases the monomial Lie algebra $\mathfrak{g}= \Lie(X)/([W])_{Lie}$ is isomorphic to some (already known) standard monomial Lie algebra $\mathfrak{g}_{W^{\star}} = \Lie(X)/([W^{\star}])_{Lie}$  with a $K$-basis $[N^{\star}] \subsetneqq [N]$, and the enveloping $U\mathfrak{g} \in \mathfrak{C}(X, W^{\star})$ is AS-regular with $\gldim U < d.$
However, Example \ref{mon-nonmonExample} shows that when $d$ is large enough, (in this case $d=18$) a Lyndon pair $(N, W)$, with $|N|= d$  may define a monomial Lie algebra $\mathfrak{g}= \Lie(X)/([W])_{Lie}$,
such that the reduced Gr\"{o}bner-Shirshov basis of the Lie ideal $J= ([W])_{Lie}$ contains Lie elements which are not Lyndon Lie monomials. This way we can obtain Artin-Schelter regular algebras $U = U\mathfrak{g}$ with obstruction set $\widetilde{W}\neq W$, and a set of atoms $\widetilde{N}\subsetneqq N$ (so $\gldim U < d$) and explicitly given more sophisticated relations which are non-trivial linear combinations of Lyndon-Lie monomials.  In this case the corresponding Lyndon pair $(\widetilde{N}, \widetilde{W})$ is regular, but it may be non-standard.

It is a nontrivial question to decide whether a given non-standard Lyndon pair $(N, W)$, with connected $N$, is regular, see
Open Question \ref{openquestion1}.
\end{remark}

\subsection*{Acknowledgements}
This paper was written during my visit to Max Planck Institute for Mathematics (MPIM), Bonn in 2019. I thank MPIM for the wonderful creative and inspiring atmosphere.  My cordial thanks to the referee of my paper for their amazing review with numerous corrections and suggestions for improvements.

\end{document}